\RequirePackage{fix-cm}
\documentclass{svjour3}                     
\smartqed  
\usepackage{graphicx}
\usepackage{bm}
\usepackage{amscd, amsfonts, amssymb, graphicx, xcolor}
\usepackage{wrapfig}
\usepackage{authblk}
\usepackage{multirow}
\usepackage{amsmath}
\usepackage{mathtools}
\usepackage{graphicx}
\usepackage{algorithm}
\usepackage{algorithmic}
\usepackage{subfig}
\usepackage{comment}
\usepackage{hyperref}

\usepackage{url}
\usepackage{tikz}
\usetikzlibrary{calc,positioning,fit,backgrounds,shapes,arrows.meta,decorations.pathreplacing}
\usetikzlibrary{patterns}
\usetikzlibrary{plotmarks}
\usepackage{comment}
\usepackage[backend=biber,style=alphabetic,sorting=none,locallabelwidth]{biblatex}
\begin{document}

\title{Model order reduction for parametrized variational inequalities: application to crowd motion
}

\titlerunning{Nonlinear ROM for parametrized variational inequalities}        

\author{\underline{Giulia Sambataro} \and Virginie Ehrlacher}


\institute{G. Sambataro \at
             Inria Nancy Grand Est, Université de Strasbourg, France \\
              \email{giulia.sambataro@inria.fr}           
           \and
          V. Ehrlarcher \at
Cermics, École Nationale des Ponts et Chaussées; Inria Paris, France
}

\date{Received: date / Accepted: date}

\maketitle

\begin{abstract}
This work investigates model order reduction for time-dependent parametrized variational inequalities, with a focus on discrete contact problems. As a prototypical example, we consider an agent-based crowd model \cite{maury2011discrete} in which agent velocities are obtained at each time step from a constrained least-squares problem. Geometric parameter variations induce significant variability in both agent positions and contact forces, leading to a slowly decaying Kolmogorov $n$-width of the solution manifold. We propose a nonlinear approach that combines a linear reduced-order model with a deep-learning-based correction. The method utilizes a greedy index selection (gIS) algorithm for compressing Lagrange multipliers and Proper Orthogonal Decomposition (POD) applied to velocity snapshots. Additionally, we explore hyper-reduction techniques, comparing the Empirical Interpolation Method (EIM) and the Empirical Quadrature (EQ) procedure from both computational complexity and accuracy perspectives.
Finally, we demonstrate the applicability of the methodology in a complex scenario involving many agents in a highly congested geometric configuration. This work represents the first attempt to apply model order reduction to a discrete contact problem of the type introduced in
 \cite{maury2011discrete} and paves the way for future advancements in nonlinear MOR specifically for this class of problems.


\end{abstract}

\section{Introduction}
\label{sec:model}
 In many applications, the task which consists in quickly computing an accurate approximation of the solution (and associated quantities of interest) of a mathematical model for a possibly wide range of parameters of various types (e.g. material properties, geometric features, or operating conditions) is of vivid interest. To alleviate the computational burden associated with the evaluation of the model for many values of the parameters, model order reduction (MOR) techniques are used to generate a reduced-order model (ROM) that computes an approximation of the solution of the original model (usually accessible through the resolution of a high-fidelity simulation code using for instance finite elements or finite volume) over a prescribed parameter range at a much lower computational cost.\\
In the present work, we employ model reduction to find the approximated solution of parametrized \emph{variational inequalities}, which arise from discrete contact problems stemming from applications related to crowd motion. Several models have been proposed to reproduce the behavior of a crowd of people in walking situations: for example, to estimate the evacuation time (e.g. in case of an emergency) or to predict areas where the density of the crowd increases; also, to estimate the interaction forces among \emph{agents}.
A large amount of models based on a microscopic description of individuals (see  \cite{helbing1995social,borgers1986city,borgers1986model}) 
or on a macroscopic description (ref. \cite{hughes2000flow,hughes2002continuum,maury2010macroscopic,santambrogio2018crowd}) have been studied over the last decades.
Among the microscopic models, some are based on a stochastic description of the individual behavior (see \cite{henderson1971statistics}), whereas others are purely deterministic (see \cite{maury2011discrete,hoogendoorn2004dynamic,hoogendoorn2004pedestrian}). In this work, we focus our efforts, for the sake of illustration, on the microscopic model described in \cite{maury2011discrete}, which has been proposed to deal with highly congested and emergency situations. 
The numerical discretization of problem \eqref{eq:DCP} with uncertain inputs (such as geometric parameters related to the shape/positions of obstacles) may require the solution of very high-dimensional discrete problems, especially for a large number of agents. More precisely, they require at each time step and for each parameter value the resolution of a high-dimensional quadratic minimization problem with linear constraints, which can be expressed as a variational linear inequality problem. \\
The proposed hybrid MOR relies on the reduced basis (RB) method (ref. 
\cite{hesthaven2016certified, ngoc2005certified, prud2002reliable}). The \emph{high fidelity} (HF) problem (also denoted as full order model (FOM) to distinguish it from the ROM) is projected onto reduced bases that are conveniently constructed from high-fidelity \emph{snapshots}.
For the \emph{a posteriori construction} of the reduced basis, we employ the proper orthogonal decomposition (POD) (see \cite{aubry1991hidden, volkwein2011model,hesthaven2022reduced}) and introduce a novel greedy algorithm (for references on greedy RB, see \cite{prud2002reliable, veroy2003posteriori}) for the compression of 
contact forces. The ROM is then obtained by a Galerkin projection onto the reduced spaces. The RB method has successfully been applied to the resolution of parametrized variational inequalities (ref. \cite{benaceur2020reduced,haasdonk2012reduced,gerner2012certified,newsum2019efficient,niakh2023stable}), mostly issued from mechanical contact problems, by generation of primal and dual reduced spaces.  In several papers (ref. \cite{rozza2007stability,haasdonk2012reduced,haasdonk2012reduced,rozza2013reduced}) and in the present work, the achievement of stability of the reduced problem (which is not guaranteed \emph{a priori}) is discussed; e.g. in \cite{le2023error} and \cite{balajewicz2016projection}, the authors achieve inf-sup stability of the reduced contact problem by a basis enrichment (driven by error indicators). We postpone the task of constructing an efficient and reliable error indicator to a further work. 
For the generation of the dual reduced basis, we propose a greedy algorithm that preserves the non-negativity of the Lagrange multipliers: we compare it with the modified cone projected greedy algorithm proposed in \cite{niakh2023stable} and in \cite{benaceur2020reduced}.  In \cite{balajewicz2016projection}, the construction of the dual basis is obtained by non-negative matrix factorization (NNMF).  In the present work we propose a greedy-based procedure, rather than the NNMF: indeed, we need to order the dual snapshots depending on their relevance to represent the entire set of dual snapshots, while the NNMF would not allow such a hierarchical construction of the dual reduced space.  Concerning the reduction of contacts, we mention also the works in \cite{kollepara2024low, kollepara2024sparse} on dictionary-based approximations and \cite{fauque2018hybrid} where a hybrid ROM for frictionless contact mechanics problems based on the reduced integration domain method is proposed.\\
Linear approximation methods are inherently inadequate for several classes of engineering problems, such as transport-dominated partial differential equations. In particular, as observed in \cite{haasdonk2013convergence}[Example $3.5$], solution fields exhibiting parameter-dependent discontinuities, sharp layers, or discontinuous coefficients cannot be accurately captured by low-dimensional linear expansions. These limitations have motivated the development of \emph{nonlinear approximation methods} (see~\cite{cohen2023nonlinear, barnett2023neural, carlberg2011efficient,franco2023deep,ehrlacher2020nonlinear} just to cite a few) to deal with these problems. The discrete contact problem we consider in this work for crowd motion modeling falls into this second class, as pointed out in section
\ref{sec:RB}. 
In this work, we investigate the potential of a nonlinear model order reduction (MOR) approach to overcome this limitation. Building on the framework proposed in \cite{cohen2023nonlinear}, the method combines a linear reduced basis (RB) approximation with a machine learning (ML)-based correction. The latter is formulated as a nonlinear function of both the system parameters and the leading generalized coordinates. This strategy seeks to enhance accuracy through nonlinear reconstruction, while remaining within the conventional projection-based model reduction paradigm. We also investigated hyper-reduction techniques, comparing the Empirical Interpolation Method (EIM) and the Empirical Quadrature (EQ) procedure from both computational complexity and numerical accuracy perspectives\\
The main contributions of this work are the following: (i) the development of a greedy-based "index selection" algorithm to construct a reduced positive cone for the compression of Lagrange multipliers, while preserving non-negativity constraints; (ii) the adaptation of a ML–based correction for the discrete contact model governing particles positions, in particular for the enriched velocity primal basis. The approach defines a correction term proportional to the discrepancy between the Galerkin generalized coordinates and the projection of the solution. This modification has the aim to stabilize the Galerkin formulation even under insufficient truncation of the basis; (iii) the numerical investigation on complex, high-dimensional and highly congested scenarios with obstacles. We also report a theoretical analysis of the greedy algorithm and the parameter identification in the context of a sphere–plane Hertz contact model.\\
The outline of the paper is the following. In Section~\ref{sec:HF}, we present the high-fidelity discrete contact model for the crowd motion problem that we consider in this work, together with its associated time discretizaton scheme; in Section~\ref{sec:RB} a reduced-basis approach is proposed for the model; also, the (lack of) effectiveness of a linear model reduction is numerically reviewed in Section~\ref{sec:RB}. The proposed nonlinear model reduction approach is described in Section~\ref{sec:ML_correction};
finally, in Section~\ref{sec:num_res}, we numerically assess the validity of the approach for two study cases. Section~\ref{sec:conclusions} wraps up this work and offers future research paths.

\section{High-fidelity crowd motion model}\label{sec:HF}

We present in this section the high-fidelity discrete contact model from~\cite{maury2011discrete} we consider in this work, together with an associated time discretization scheme. For any vector $u\in \mathbb{R}^p$ for some $p\in \mathbb{N}^*$, we denote by $|u|$ the euclidean norm of $u$. 
\medskip

We identify $N^{\rm a} \in \mathbb{N}\setminus\{0\}$ agents by rigid disks of radius $r^{\rm a}>0$, with center $\mathbf{q}_i \in \mathbb{R}^2$ for $i=1, \ldots, N^{\rm a}$. Let us also assume that there are $N^{\rm obst}\in \mathbb{N}$ obstacles in the room, each of them being represented by a convex closed (piecewise regular) subdomain $\Omega_k \subset \mathbb{R}^2$ for $k=1, \ldots, N^{\rm obst}$. For all  $\mathbf{q}  := (q_1, \ldots, q_{N^{\rm a}}) \in (\mathbb{R}^2)^{N^{\rm a}} = \mathbb{R}^{2 N^{\rm a}}$, all $1 \leq i <j \leq N^{\rm a}$ and all $1\leq k \leq N^{\rm obst}$, we denote by 
$$
D_{ij}(\mathbf{q}):=|q_i-q_j|-2r^{\rm a},
$$ 
and by 
$$
D_{ik}(\mathbf{q}) := {\rm dist}(q_i, \Omega_k) = \mathop{\inf}_{\omega_k \in \Omega_k}|q_i - \omega_k|.
$$
We then denote by 
$$
\mathbf{D}(\mathbf{q}):= \left( D_{ij}(\mathbf{q}), D_{ik}((\mathbf{q})\right)_{
1\leq i \leq N^{\rm a},
i <j \leq N^{\rm a},
1\leq k \leq N^{\rm obst}
} \in \mathbb{R}^{N^{\rm cont}}
$$
with $N^{\rm cont}:= \frac{N^{\rm a}(N^{\rm a}-1)}{2} + N^{\rm a}N^{\rm obst}$. In the following, to simplify the notation, we will denote by $\left( D_\ell(\mathbf{q})\right)_{1\leq \ell \leq N^{\rm cont}}$ the coordinates of the vector $\mathbf{D}(\mathbf{q})\in \mathbb{R}^{N_{\rm cont}}$.

To avoid collisions between the different agents of the crowd, or between the agents and the different obstacles inside the room, their positions have to belong to a set of feasible configurations (which naturally describes the positions of the centers of non overlapping disks) 
\begin{equation}
    \mathcal{Q}=\{\mathbf{q} \in \mathbb{R}^{2 N^{\rm a}} : \;  D_{\ell}(\mathbf{q}) \geq 0, \, \forall 1 \leq \ell \leq N^{\rm cont}\}.
    \label{eq:pos_feasibility}
\end{equation}

\medskip

Analogously, we introduce a feasibility set for the velocities which is a closed convex cone depending on the set of admissible positions $\mathbf{q} \in \mathcal{Q}$ defined as follows:
 \begin{equation}
     \mathcal{C}_{\mathbf{q}}:=\left\{\mathbf{v} \in \mathbb{R}^{2N^{\rm a}}: \;  \forall 1\leq \ell \leq N^{\rm cont}, \quad \left(D_{\ell}(\mathbf{q})=0 \implies G_{\ell}(\mathbf{q}) \cdot \mathbf{v} \geq 0\right)\right\},
     \label{eq:vel_cone}
 \end{equation}
where for all $1\leq \ell \leq N^{\rm cont}$ and all $\mathbf{q} \in  \mathbb{R}^{2 N^{\rm a}}$, $G_{\ell}(\mathbf{q})=\nabla D_{\ell}(\mathbf{q})\in \mathbb{R}^{2N^{\rm a}}$. The matrix $\mathbf{G}(\mathbf{q}) := (G_1(\mathbf{q}) , \ldots , G_{N^{\rm cont}}(\mathbf{q}))^T \in \mathbb{R}^{N^{\rm cont} \times 2N^{\rm a}}$ is the Jacobian matrix of the vector-valued function $D$ evaluated at $\mathbf{q}$. \\
It is assumed in the model that, for a given set of admissible positions $\mathbf{q} \in \mathcal Q$, the agents have some \emph{spontaneous velocities} which correspond to the velocity they would have had if in the absence of other agents or obstacles. The collection of these spontaneous velocities is denoted by $\boldsymbol{\upsilon}(\mathbf{q})\in \mathbb{R}^{2N^{\rm a}}$. In most models, in particularly those presented in~\cite{maury2011discrete}, the velocity field $\boldsymbol{\upsilon}$ typically depends on the the geodesic distance between each agent and the exit.

\bigskip

 The high fidelity discrete contact problem (DCP) is then formulated such that the vector of velocities of the agents at some time $t>0$ is given as the solution of a constrained minimization problem. More precisely, assuming that at time $t>0$ the agents are located at positions $\mathbf{q}\in \mathcal Q$, the actual velocity field is found as the closest feasible velocity field in $\mathcal C_{\mathbf{q}}$ to the spontaneous velocity field $\boldsymbol{\upsilon}(\mathbf{q})$ in a least-square sense. This leads to the following ODE system, which models the evolution in time of the positions of the agents in the system:
 \begin{equation}
		\begin{cases}
			&\frac{d\mathbf{q}}{dt}(t)=\texttt{P}_{\mathcal{C}_{\mathbf{q(t)}}}(\boldsymbol{\upsilon}(\mathbf{q}(t))), \quad t>0,\\
			&\mathbf{q}(0)=\mathbf{q}_0 \in \mathcal{Q},
		\end{cases}
  \label{eq:DCP}
	\end{equation}
where for all closed convex set $\mathcal C$ of $\mathbb{R}^{2N^{\rm a}}$, $\texttt{P}_{\mathcal{C}}$ denotes the euclidean projection onto $\mathcal{C}$. We point out that i) the constraints in $\mathcal{C}_{\mathbf{q}}$ can be associated with the non-overlapping condition among agents and, in the very same way, with the non-overlapping condition between agents and obstacles; ii) the closed convex cone $\mathcal{C}_{\mathbf{q}}$ does not continuously depend on $\mathbf{q}$; iii) the definition of the model \eqref{eq:DCP} ensures that for all $t>0$, $\mathbf{q}(t)$ belongs to $\mathcal Q$.\\

\subsection{Time-discretization of the high-fidelity model}
\label{sec:FOM}
We briefly present in this section the numerical time discretization scheme used for the practical computation of the solution of problem \eqref{eq:DCP} (we refer to \cite{maury2011discrete} for a more detailed description).
We consider a finite time interval denoted by $[0,T]$ for some final time $T>0$ and a constant time step denoted by $h:=T/N^T$ for some $N^T \in \mathbb{N}^{\star}$. For any $0\leq \nu \leq N^T$, the $\nu^{\rm th}$ computational time is denoted by $t^{\nu}:=\nu h$. 

\medskip

Let $1\leq \nu \leq N^T$. We denote by $\mathbf{q}^{\nu-1}\in \mathbb{R}^{2N^{\rm a}}$ the approximation of $\mathbf{q}(t^{\nu-1})$ given by the time discretization scheme.
The velocity $\mathbf{u}^{\nu}\in \mathbb{R}^{2N^{\rm a}}$ is found by solving the following projection-based problem
\begin{subequations}
\begin{equation}
\mathbf{u}^{\nu}=\texttt{P}_{\mathcal{C}^{h}_{\mathbf{q}^{\nu-1}}}(\boldsymbol{\upsilon}(\mathbf{q}^{\nu-1})),
    \label{eq:discr_vel_DCP}
\end{equation}
where for all $\mathbf{q}\in \mathcal Q$, the set $\mathcal{C}^{h}_{\mathbf{q}}$ is a discretized set of feasible velocities
\begin{equation}
\label{eq:discr_vel_cone}
    \mathcal{C}^{h}_{\mathbf{q}}=\{\mathbf{u}\in \mathbb{R}^{2N^{\rm a}}: \;   D_{\ell}(\mathbf{q})+h G_{\ell}(\mathbf{q})\cdot \mathbf{u} \geq 0, \, \forall 1\leq \ell \leq N^{\rm cont}\}.
\end{equation}
where the expression in \eqref{eq:discr_vel_cone} stems from the following first order expansion in time of the constraints: 
\begin{equation}
  \forall 1\leq \ell \leq N^{\rm cont}, \quad   D_{\ell}(\mathbf{q}^{\nu-1}+h \mathbf{u}^{\nu})=D_{\ell}(\mathbf{q}^{\nu-1})+h \mathbf{u}^{\nu} \cdot \nabla D_{\ell}(\mathbf{q}^{\nu-1}) +\mathcal{O}(h^2).
    \label{eq:D_1st_develop}
\end{equation}
Once the velocity $\mathbf{u}^{\nu}$ is known, the next position configuration is obtained: 
\begin{equation}
 \mathbf{q}^{\nu}=\mathbf{q}^{\nu-1}+h\mathbf{u}^{\nu}.
 \label{eq:discr_pos_update}
\end{equation}
\end{subequations}
Following \cite{maury2011discrete}, we solve the projection problem in \eqref{eq:discr_vel_DCP} by Uzawa algorithm. Indeed, any problem of the form  \eqref{eq:discr_vel_DCP} can be recast as a minimization problem of the following form: find $\mathbf{u}\in \mathbb{R}^{2N^{\rm a}}$ solution to
\begin{equation}
\mathbf{u}= \texttt{P}_{\mathcal C_{\mathbf{q}}^h}(\boldsymbol{\upsilon}_{\mathbf{q}}) = \underset{\mathbf{v}\in \mathcal{C}^h_{\mathbf{q}}}{\text{argmin}}\;{|\mathbf{v}-\boldsymbol{\upsilon}_{\mathbf{q}}|^2},
\label{eq:primal_eq}
\end{equation}
for some $\mathbf{q}\in \mathcal Q$ and $\boldsymbol{\upsilon}_{\mathbf{q}}:= \boldsymbol{\upsilon}(\mathbf{q})\in \mathbb{R}^{2N^{\rm a}}$. Problem \eqref{eq:primal_eq} can be equivalently expressed as follows: find $(\mathbf{u}, \boldsymbol{\lambda})\in \mathbb{R}^{2N^{\rm a}} \times \mathbb{R}_+^{N^{\rm cont}}$ solution to 
\begin{equation}
\left\{
\begin{array}{l}
\mathbf{u} = \boldsymbol{\upsilon}_{\mathbf{q}} - \mathbf{B}_{\mathbf{q}}^T \boldsymbol{\lambda},\\
\boldsymbol{\lambda} \odot (\mathbf{B}_{\mathbf{q}}\mathbf{u} - \mathbf{d}_{\mathbf{q}}) = 0,\\
\mathbf{B}_{\mathbf{q}}\mathbf{u} - \mathbf{d}_{\mathbf{q}} \leq 0, \\
\end{array}
\right.
\label{eq:DCP}
\end{equation}
where $\mathbf{B}_{\mathbf{q}}:=-h\mathbf{G}(\mathbf{q}) \in \mathbb{R}^{N^{\rm cont}\times 2N^{\rm a}}$, 
 $\mathbf{d}_{\mathbf{q}}:= \mathbf{D}(\mathbf{q})\in \mathbb{R}^{N^{\rm cont}}$ and for all $\mathbf{v}:= (v_{\ell})_{1\leq \ell \leq N^{\rm cont}}, \mathbf{w}:= (w_{\ell})_{1\leq \ell \leq N^{\rm cont}} \in \mathbb{R}^{N^{\rm cont}}$, $\mathbf{v} \odot \mathbf{w} := (v_\ell w_\ell)_{1\leq \ell \leq N^{\rm cont}} \in \mathbb{R}^{N^{\rm cont}}$.

The vector $\boldsymbol{\lambda}$ is called the Lagrange multiplier associated to the optimization problem~\eqref{eq:primal_eq}.

\medskip

The Uzawa algorithm~\cite{kepleruzawa} for the resolution of \eqref{eq:primal_eq} is an iterative algorithm that produces two sequences $(\mathbf{u}_k)_{k\geq 0} \subset \mathbb{R}^{2N^{\rm a}}$ and $(\boldsymbol{\lambda}_k)_{k\geq 0}
\subset \mathbb{R}_+^{
N^{\rm cont}}$ that solve the following scheme for $k=0,1,\ldots, $ until convergence:
\begin{subequations}
\begin{align}
	& \boldsymbol{\lambda}_0=0,\\
	&\mathbf{u}_{k+1}=\boldsymbol{\upsilon}_{\mathbf{q}}-\mathbf{B}_{\mathbf{q}}^T\boldsymbol{\lambda}_k
 \label{eq:Uzawa_u},\\
	&\boldsymbol{\lambda}_{k+1}=\texttt{P}_{+}
	\left(\boldsymbol{\lambda}_{k}+ \rho (\mathbf{B}_{\mathbf{q}} \mathbf{u}_{k+1}-\mathbf{d}_{\mathbf{q}})
	\right),
 \label{eq:Uzawa_lambda}
\end{align}
\label{eq:discr_Uzawa_system}
\end{subequations}
where $\rho>0$ is a fixed parameter, and for all $\mathbf{w}:= (w_{\ell})_{1\leq \ell \leq N^{\rm cont}} \in \mathbb{R}^{N^{\rm cont}}$, $\texttt{P}_+ (\mathbf{w}) :=\left(\max(w_\ell,0)\right)_{1\leq \ell \leq N^{\rm cont}}\in \mathbb{R}^{N^{\rm cont}}$.

\medskip

The algorithm in \eqref{eq:discr_Uzawa_system} can be shown to converge as soon as $0<\rho<\frac{2}{\|B_{\mathbf{q}}\|_2^2}$ (\cite{ciarlet1982introduction}): the sequence $(\mathbf{u}_{k})_{k\geq 0}$  converges to $\mathbf{u}$ and it can be shown that the sequence $(\boldsymbol{\lambda}_k)_{k\geq 0}$ tends to some $\boldsymbol{\lambda}\in \mathbb{R}_{+}^{N^{\rm cont}}$ such that $(\mathbf{u}, \boldsymbol{\lambda})$ is a solution to~\eqref{eq:discr_Uzawa_system}.\\

	We invite the reader to observe that algorithm \eqref{eq:discr_Uzawa_system} can be reformulated in terms of the Lagrange multipliers and recast to the equivalent version called Fixed-step Projected Gradient Descent (PGD) (cf. \cite[Algorithm $1$]{bloch2023convex}); furthermore, \cite{bloch2023convex}, Algorithm $2$ presents an accelerated version of the algorithm by means of Nesterov’ optimized step. A more involved analysis of different possible algorithms to solve \eqref{eq:DCP} can be found in \cite{bloch2023convex} and is beyond the scope of this work.\\
	We also refer to \cite{faure2015crowd} for the formulation of the crowd motion model without friction as a generalized gradient flow problem. In this framework, the flow function is generally non-convex due to the non-convexity of the feasible set of positions $\mathcal{Q}$. The possibility that the algorithm becomes trapped in a local minimum—depending on the geometric configuration of the domain and the particle ratio—is mathematically accounted for in \cite{faure2015crowd}.

\section{A linear reduced order model for the DCP}
\label{sec:RB}
The aim of this section is to present the linear reduced-order model we consider for problem \eqref{eq:DCP}: it relies on a Reduced Basis paradigm and reads a Galerkin projection-based MOR for the contact problem described in \eqref{eq:DCP}.
\subsection{Construction of the reduced bases}
\label{sec:RB_constr}

Let us now denote by $\mathcal N:= 2N^{\rm a}$, $\mathcal R:= N^{\rm cont}$, $\mathcal{V}:= \mathbb{R}^{\mathcal N}$, $\mathcal W:= \mathbb{R}^{\mathcal R}$ and $\mathcal W^+:= \mathbb{R}_+^{\mathcal R}$. We also denote by $\|\cdot\|_{\mathcal V}$ and by $\|\cdot\|_{\mathcal W}$ the euclidean norm of $\mathcal V$ and $\mathcal W$ respectively.

\medskip

Let $\mathcal P \subset \mathbb{R}^p$ be a set of parameters the discrete contact model \eqref{eq:discr_Uzawa_system} may depend on (typically describing the geometry of the obstacles for instance). For any $\mu \in \mathcal P$, the high-fidelity scheme will produce a time-discrete set of solutions $$\left(\mathbf{u}^{\nu}(\mu), \boldsymbol{\lambda}^{\nu}(\mu), \mathbf{q}^\nu(\mu)\right)_{1\leq \nu \leq N^T}.$$
In the rest of the paper, we will also use the following notation, when convenient, $\mathbf{u}(\nu, \mu):=\mathbf{u}^{\nu}(\mu)$,  $\boldsymbol{\lambda}(\nu, \mu):=\boldsymbol{\lambda}^{\nu}(\mu)$ and $\mathbf{q}(\mu,\nu):= \mathbf{q}^\nu(\mu)$ for all $0\leq \nu \leq N^T$ and $\mu \in \mathcal P$. Let us also denote by $\mathcal S:= \{0, \ldots, N^T\} \times \mathcal P$. 

\medskip

We define the velocity solution set as $\mathcal{M}^u:=\{\mathbf{u}(\nu,\mu):\;  (\nu, \mu)\in \mathcal S \}\subset \mathcal V$, the position solution set $\mathcal M^q:= \{\mathbf{q}(\nu,\mu):\;  (\nu, \mu)\in \mathcal S \}\subset \mathcal V$ and the Lagrange multipliers set $\mathcal{M}^{\lambda}:=\{\boldsymbol{\lambda}(\nu,\mu): (\nu, \mu)\in \mathcal S\} \subset \mathcal W^+$.

\medskip

The approximation of the primal solutions is seeked in a reduced subspace $\hat{\mathcal{V}}_N \subset \mathcal V$ so that $\hat{\mathcal{V}}_N$ is the vector space spanned by an orthogonal family of vectors $\{\boldsymbol{\varphi}_n\}_{n=1}^N\subset \mathcal V$ for some $N \leq \mathcal N$. Similarly, the approximation of the dual solutions is seeked in a reduced subcone $\hat{\mathcal{W}}_R^+ \subset \mathcal W^+$ so that $\hat{\mathcal{W}}_R^+$ is the non-negative cone spanned by a given family of vectors $\{\boldsymbol{\psi}_r\}_{r=1}^R\subset \mathcal W^+$ for some $R \leq \mathcal R$. The precise choice of $N$, $R$, $\{\boldsymbol{\varphi}_n\}_{n=1}^N$ and $\{\boldsymbol{\psi}_r\}_{r=1}^R$ will be detailed in the next sections. 
The reduced solutions will then be written for all $0\leq \nu \leq N^T$ as 
\begin{equation*}
\displaystyle \hat{\mathbf{u}}^{N,R}(\nu, \mu) = \sum_{n=1}^N \alpha^{N,R}_n(\nu, \mu) \boldsymbol{\varphi}_n, \qquad
\displaystyle \hat{\boldsymbol{\lambda}}^{N,R}(\nu, \mu) = \sum_{r=1}^R \beta^{N,R}_r(\nu, \mu) \boldsymbol{\psi}_r, 
\text{with}
\end{equation*}
$\boldsymbol{\alpha}^{N,R}(\nu, \mu) = \left(\alpha^{N,R}_n(\nu,\mu)\right)_{1\leq n \leq N} \in \mathbb{R}^N$, $\boldsymbol{\beta}^{N,R}(\nu,\mu) = \left(\beta^{N,R}_r(\nu, \mu)\right)_{1\leq r \leq R} \in \mathbb{R}_+^R$
the generalized coordinates of the reduced displacements and Lagrange multipliers solutions. In the following, we will denote by $\hat{V}_N \in \mathbb{R}^{\mathcal N\times N}$ the matrix composed of the coordinates of $\{\boldsymbol{\varphi}_n\}_{n=1}^N$ and by $\hat{W}_R^+\in \mathbb{R}_+^{\mathcal R\times R}$ the matrix composed of the coordinates of $\{\boldsymbol{\psi}_r\}_{r=1}^R$.
 
\medskip

To construct the reduced space and reduced cone, we assume that we are given a dataset of solutions $(\mathbf{u}(\sigma), \boldsymbol{\lambda}(\sigma))$ computed as solutions of the Full Order model \eqref{eq:DCP} for $\sigma := (\nu,\mu) \in \mathcal S_{\rm train}:= \{0, 1, \ldots, N^T\}\times \mathcal P_{\rm train}$, 
where $\mathcal P_{\rm train}$ is a finite training subset of $\mathcal P$. 

\subsubsection{The primal reduced space $\hat{\mathcal{V}}_N$ }
\label{sec:pod}

The family $\{\boldsymbol{\varphi}_n\}_{n=1}^N$ is computed in a standard way as the first $N$ POD modes (corresponding to the $N$ largest singular values) of the family of velocity snapshots $\{\mathbf{u}(\sigma)\}_{\sigma \in \mathcal S_{\rm train}}$. The reduced subspace $\hat{\mathcal{V}}_N $ is then equal to ${\rm Span}\{ \boldsymbol{\varphi}_1, \ldots, \boldsymbol{\varphi}_N\}$. The value of $N\in \mathbb{N}^*$ is chosen so that the corresponding POD (relative) error is below an error threshold $1>\epsilon>0$.

\medskip

For the sake of illustration, we show in the following figures the behaviour of the POD for the parametric DCP problem detailed in Section~\ref{sec:num_res}. We consider a family of velocity snapshots $\{\mathbf{u}(\sigma)\}_{\sigma \in \mathcal S_{\rm train}}$ where $\mathcal S_{\rm train}$ is chosen so that  $|\mathcal S_{\rm train}|= 8000$ (for more details see Section~\ref{sec:num_res}).

Let us denote by $\ell_1 \geq \ell_2 \geq \ldots$ the eigenvalues associated to the POD decomposition ranged in non-increasing order. We show in Figure~\ref{fig:eigPOD_u}\subref{fig:pod_u} the decay of $\frac{\ell_k}{\ell_1}$ as a function of the index $k$, and in Figure~\ref{fig:eigPOD_u}\subref{fig:energy_i} the associated relative squared error $E_k:= \frac{\sum_{j\geq k+1} \ell_j}{\sum_{j\geq 1} \ell_j}$.


\begin{figure}[h!]
\centering
\subfloat[Decay of POD eigenvalues.]{
  \includegraphics[scale=0.32]{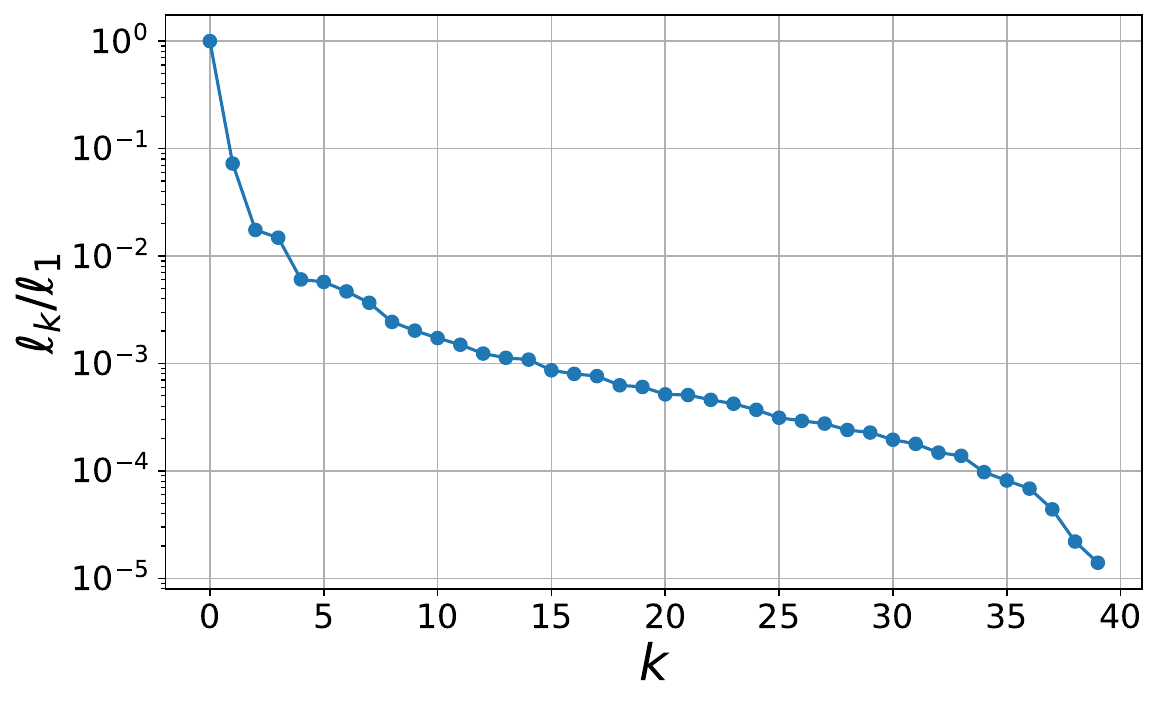}
    \label{fig:pod_u}
    }
\subfloat[Relative squared error of the truncated POD with $k$ modes.]{
	  \includegraphics[scale=0.32]{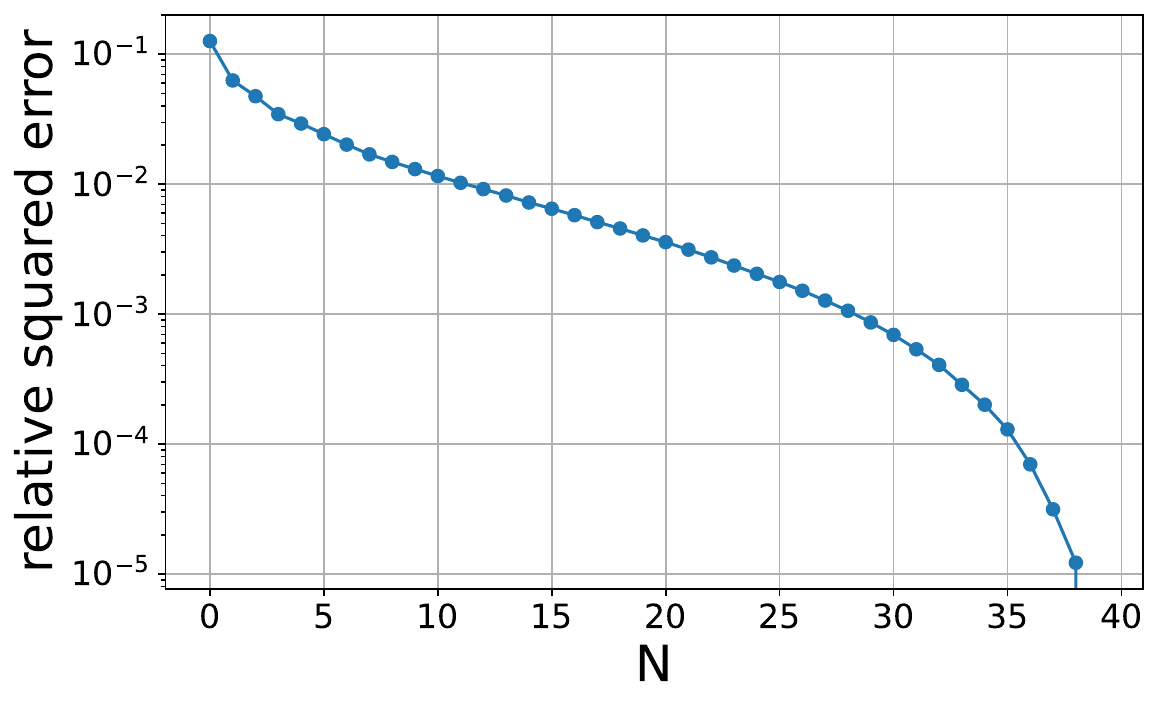}
    \label{fig:energy_i}
    }   
\caption{POD on the velocity snapshots $\{\mathbf{u}(\sigma)\}_{\sigma \in \mathcal S_{\rm train}}$.}
\label{fig:eigPOD_u}
\end{figure}

We can see in particular that the relative squared truncated POD error decays slowly with respect to the number of POD modes. This slow decay illustrates the fact that the so-called Kolmogorov width of the solution set $\mathcal M^u$ decays at a slow rate, which makes standard linear reduced-order models not well suited in our present context. This in particular motivates the use of hybrid and nonlinear model order reduction methods for problem \eqref{eq:DCP}, which is the object of the second part of this work. We postpone the description of the non-linear approximation approach we propose in this work to Section~\ref{sec:ML_correction}. \\
	We refer the interested reader to \cite{bellomo2008modelling} for a detailed analysis of crowd dynamics modeled as macroscopic first-order systems of conservation laws. In particular, the authors focus on (i) the conservation of mass, representing the preservation of pedestrian density, and (ii) the equilibrium of linear momentum, which governs the evolution of crowd velocity.
	The slow decay  of the Kolmogorov $n$-width with respect to the dimension of a linear reduced subspace is a well-known phenomenon in the approximation theory of PDEs. This behavior, as observed in Figure \ref{fig:eigPOD_u}, is supported by the mathematical connection between the crowd motion model \eqref{eq:DCP} and hyperbolic PDEs, as discussed in \cite{bellomo2008modelling}.

	

\subsubsection{The dual reduced cone $\hat{\mathcal{W}}_R^+$}
\label{sec:dual_basis}

In this work, we investigate two possible strategies to construct the dual reduced cone $\hat{\mathcal{W}}_R^+$ which are detailed below. The first strategy is a cone projected greedy algorithm which has been originally proposed in~\cite{niakh2023stable}. The second approach, which will be the one we will adopt in this work \itshape in fine \normalfont is based on a greedy index selection algorithm detailed below. In particular, both approaches are tailored in order to guarantee the fact that the vectors $\{\boldsymbol{\psi}_r\}_{r=1}^R$ belong to $\mathcal W^+$.

\paragraph{Cone projected greedy algorithm}
We implement the so-called modified cone projected greedy (mCPG) algorithm (\cite{niakh2023stable}) to find the vectors $\{\boldsymbol{\psi}_r\}_{r=1}^R$.

The procedure, described in Algorithm \ref{alg:mCPG}, takes as input the solution snapshots $\{\boldsymbol{\lambda}(\sigma)\}_{\sigma\in\mathcal S_{\rm train}}$, a positive tolerance parameter $\delta>0$ and a maximal number of iterations $R_{\rm max}\in \mathbb{N}^*$ as stopping criteria. As outputs, it produces a set of $R \in \mathbb{N}^*$ vectors such that the maximum projection error of the snapshot $\boldsymbol{\lambda}(\sigma)$ in $\mathcal{\hat{W}}_R^+=\text{Span}^+\{\{\boldsymbol{\psi}_r\}_{r\in \{1:R \}}\}$ over $\sigma \in \mathcal S_{\rm train}$ is below the requested threshold, i.e. such that:
\begin{equation}
\label{eq:mCPG_err}
    e_R =\frac{\displaystyle{\max_{\sigma \in\mathcal S_{\rm train}}} \left\|\left(\mathbb{I}-\texttt{P}_{\hat{\mathcal{W}}_R^+}\right)\boldsymbol{\lambda}(\sigma)\right\|_{\mathcal{W}}}{\displaystyle{\max_{\sigma \in\mathcal S_{\rm train}}}\left\|\boldsymbol{\lambda}(\sigma)\right\|_{\mathcal{W}}}\leq \delta, 
\end{equation}
where $\mathbb{I}$ denotes the identity map of $\mathcal W$.

At line \ref{alg_line:proj}, the mode $r+1$ is found by projecting the training snapshots onto the positive cone constructed by means of the previously computed $r$ modes; the projection error \eqref{eq:mCPG_err} of the selected snapshot $\boldsymbol{\lambda}(\sigma_{r+1})$ in the updated positive cone $\hat{\mathcal{W}}_R^+$ is performed at line \ref{alg_line:proj_err}.
\begin{algorithm}[h]
\caption{Modified cone projected greedy (mCPG)}\label{alg:mCPG}
  \begin{algorithmic}[1]
 \STATE{\textbf{Inputs}:$\{\boldsymbol{\lambda}(\sigma)\}_{\sigma\in\mathcal S_{\rm train}}$, $\delta>0$, $R_{\rm max}\in\mathbb{N}^*$ }
\STATE{\textbf{Outputs}: $R \in \mathbb{N}^*$, $\{\boldsymbol{\psi}_r\}_{r\in \{1:R \}} \subset \mathcal{W}^{+}$}
\STATE{\textbf{Initialization:} $r=0$, $I_0=\emptyset$, $\hat{\mathcal{W}}_0^+=\{0\}$, $e_0=1+\delta$}
\STATE{$\sigma_1\in \arg \max_{\sigma\in \mathcal S_{\rm train}}\left\|\boldsymbol{\lambda}(\sigma)\right\|_{\mathcal{W}}$}
\WHILE{$e_r>\delta$ and $r<R_{\rm max}$} 
	\STATE{$r\leftarrow r+1$}
	\STATE{$I_r=I_{r-1}\cup \{\sigma_r\}$}
	\STATE{
			$\boldsymbol{k}_r^{\star}=\arg \displaystyle{\min_{\begin{array}{c}
			\boldsymbol{k}:=(k_i)_{1\leq i \leq r-1} \in \mathbb{R}_+^{r-1}\\
			\boldsymbol{\lambda}(\sigma_r)-\sum_{i=1}^{r-1}k_i \boldsymbol{\psi}_i\geq 0\\
			\end{array}}} \left\|\boldsymbol{\lambda}(\sigma_r)-\sum_{i=1}^{r-1}k_i\boldsymbol{\psi}_i\right\|_{\mathcal{W}}^2 \, 
				$
   \label{alg_line:basis_func}
   }
   \STATE{$\displaystyle \boldsymbol{\gamma}_r=\sum_{i=1}^{r-1}k_{r,i}^{\star}\boldsymbol{\psi}_i$}
	\STATE{$\boldsymbol{\psi}_r=\frac{\boldsymbol{\lambda}(\sigma_r)-\boldsymbol{\gamma}_r}{\left\|\boldsymbol{\lambda}(\sigma_r)-\boldsymbol{\gamma}_r\right\|_{\mathcal{W}}}$}
	\STATE{$\hat{\mathcal{W}}_r^+= \text{Span}^+\{\boldsymbol{\psi}_1, \ldots, \boldsymbol{\psi}_r\}$}
	\STATE{$\sigma_{r+1}=\arg \displaystyle{\max_{\sigma\in \mathcal S_{\rm train}\setminus I_r}} \, \min_{\boldsymbol{k} = (k_i)_{1\leq i \leq r} \in \mathbb{R}_+^r}\left\|\boldsymbol{\lambda}(\sigma)-\sum_{i=1}^r k_i \boldsymbol{\psi}_i\right\|_{\mathcal{W}}^2 \frac{1}{\left\|\boldsymbol{\lambda}(\sigma_1)\right\|_{\mathcal{W}}}$
 \label{alg_line:proj}
 }
	\STATE{$e_r=\frac{\|\left( \mathbb{I} -\texttt{P}_{\hat{\mathcal{W}}_r^{+}}\right)\boldsymbol{\lambda}(\sigma_{r+1})\|_{\mathcal{W}}}{\|\boldsymbol{\lambda}(\sigma_1)\|_{\mathcal{W}}}$ 
 \label{alg_line:proj_err}}
\ENDWHILE
\STATE{$R=r$.}
\end{algorithmic}
\end{algorithm}

In this work, the implementation of Algorithm \ref{alg:mCPG} is based on the Python convex optimization package \texttt{cvxopt} for quadratic programming~\cite{andersen2020cvxopt}.
\paragraph{Greedy index selection algorithm}
In Algorithm \ref{alg:gIS}, we construct the functions $\{\boldsymbol{\psi}_r\}_{r\in\{1:R\}}$ by selecting (line \ref{alg_line:gIS_biggest_coord}) at each iteration $r$ the largest coordinate (denoted as $i_r$) of the Lagrange multipliers evaluated at the currently selected parameter value $\sigma_r$: this coordinate is used to construct the new selected vector (line \ref{alg_line:gIS_nur}) as the $i_r^{th}$ element of the canonical basis of $\mathbb{R}^{\mathcal R}$. In the rest of the paper, for all $1\leq i \leq \mathcal R$, we denote by $\mathbf{e}_i \in \mathcal W_+$ the $i^{th}$ vector of the canonical basis of $\mathbb{R}^\mathcal{R}$. 
\begin{algorithm}[h]
\caption{Greedy index selection (gIS)}\label{alg:gIS}
\begin{algorithmic}[1]
\STATE{\textbf{Inputs}: $\{\boldsymbol{\lambda}(\sigma)\}_{\sigma \in\mathcal S_{\rm train}}$, $\delta>0$, $R_{\rm max}\in \mathbb{N}^*$}
\STATE{\textbf{Outputs}: $R\in \mathbb{N}^*$, $\{\boldsymbol{\psi}_r\}_{r\in \{1:R \}}\subset \mathcal{W}^{+}$}
\STATE \textbf{Initialisation}: $\sigma_1 \in \arg \max_{\sigma \in \mathcal S_{\rm train}} \|\boldsymbol{\lambda}(\sigma)\|_{\mathcal{W}}$, $r=0$, $I_{0}=\emptyset$,  $e_0=1+\delta$, $\hat{\mathcal{W}}_0^+=\{0\}$
\WHILE{$e_r>\delta$ and $r<R_{\rm max}$}
    \STATE{$ r\leftarrow r+1$ }
	\STATE{$\displaystyle i_r \in \arg \max_{1\leq i\leq \mathcal R, i \notin I_{r-1}} \left(\boldsymbol{\lambda}(\sigma_{r})\right)_i$}
 \label{alg_line:gIS_biggest_coord}
	\STATE{$I_r = I_{r-1} \cup \{i_r\}$} 
	\STATE{$\boldsymbol{\psi}_r=\mathbf{e}_{i_r}$} \label{alg_line:gIS_nur}
	\STATE $\displaystyle \sigma_{r+1}\in {\rm{arg}} \max_{\sigma \in \mathcal S_{\rm train}} \left\|\boldsymbol{\lambda} (\sigma) - \sum_{r'=1}^r \left(\boldsymbol{\lambda} (\sigma)\right)_{i_{r'}} \mathbf{e}_{i_{r'}}\right\|_{\mathcal{W}}$
 \label{alg_line:gIS_max_par}
	\STATE{Compute the error $e_r=\frac{\left\|\boldsymbol{\lambda} (\sigma_{r+1}) - \sum_{r'=1}^r \left(\boldsymbol{\lambda} (\sigma_{r+1})\right)_{i_{r'}} \mathbf{e}_{i_{r'}}\right\|_{\mathcal{W}}}{\|\boldsymbol{\lambda}(\sigma_1)\|_{\mathcal W}}$}
 \label{alg_line:gIS_max_err}
\ENDWHILE
\STATE{$R=r$}
	\end{algorithmic}
  \end{algorithm}
  
\subsection{Projection-based reduced order model}
We introduce the reduced basis reduced-order model corresponding to (\ref{eq:DCP}); it serves as a benchmark for assessing the proposed nonlinear model order reduction approach.
We employ the previously described reduced space $\hat{\mathcal V}_N$ and reduced cone $\hat{\mathcal W}_R^+$ then reads as follows: for a given $\mathbf{q}\in \mathcal Q$, find $(\boldsymbol{\alpha}^{N,R}, \boldsymbol{\beta}^{N,R})\in \mathbb{R}^{N} \times \mathbb{R}_+^{R}$ solution to 
\begin{equation}
\left\{
\begin{array}{l}
\boldsymbol{\alpha}^{N,R} = \hat{\boldsymbol{\upsilon}}_{\mathbf{q}}^N - (\hat{\mathbf{B}}^{N,R}_{\mathbf{q}})^T \boldsymbol{\beta}^{N,R},\\
\boldsymbol{\beta}^{N,R} \odot (\hat{\mathbf{B}}^{N,R}_{\mathbf{q}}\boldsymbol{\alpha}^{N,R} - \hat{\mathbf{d}}_{\mathbf{q}}^R) = 0,\\
\hat{\mathbf{B}}^{N,R}_{\mathbf{q}}\boldsymbol{\alpha}^{N,R} - \hat{\mathbf{d}}_{\mathbf{q}}^R \leq 0, \\
\end{array}
\right. 
\label{eq:DCPred}
\end{equation}
where
$
\hat{\boldsymbol{\upsilon}}_{\mathbf{q}}^N:= \hat{V}_N^T\boldsymbol{\upsilon}_{\mathbf{q}}
\in\mathbb{R}^N$, $\hat{\mathbf{B}}^{N,R}_{\mathbf{q}}:= (\hat{W}_R^+)^T \mathbf{B}_{\mathbf{q}}\hat{V}_N \in \mathbb{R}^{R\times N}$ and $\hat{\mathbf{d}}_{\mathbf{q}}^R:= (\hat{W}_R^+)^T\mathbf{d}_{\mathbf{q}} \in \mathbb{R}^R
$.

\medskip

We also use in practice the Uzawa algorithm to compute a solution of (\ref{eq:DCPred}). The latter then amounts to computing two sequences $(\boldsymbol{\alpha}^{N,R}_k)_{k\geq 0} \in (\mathbb{R}^{N})^{\mathbb{N}}$ and
$(\boldsymbol{\beta}^{N,R}_k)_{k\geq 0} \in (\mathbb{R}_{+}^{R})^{\mathbb{N}}$ so that
\begin{align}
\label{eq:DCM_ROM1}
    & \boldsymbol{\beta}^{N,R}_0 =\mathbf{0}\\
    & \boldsymbol{\alpha}^{N,R}_{k+1}=\hat{\boldsymbol{\upsilon}}_{\mathbf{q}}^N-(\hat{\mathbf{B}}^{N,R}_{\mathbf{q}})^T \boldsymbol{\beta}^{N,R}_k \label{eq:DCM_ROM2}\\
    & \boldsymbol{\beta}^{N,R}_{k+1}=\texttt{P}_+\left\{
    \boldsymbol{\beta}^{N,R}_k+\rho \left[
    \hat{\mathbf{B}}^{N,R}_{\mathbf{q}}\boldsymbol{\alpha}^{N,R}_{k+1}-\hat{\mathbf{d}}_{\mathbf{q}}^R
    \right]
    \right\} \label{eq:DCM_ROM3}
\end{align}
for $k=0,1,\ldots,$ until convergence of the scheme. 

This leads to approximations of the velocity $\mathbf{u}$ and Lagrangian multiplier $\boldsymbol{\lambda}$ of the form 
\begin{equation}\label{eq:approxu}
\hat{\mathbf{u}}^{N,R} = \sum_{n=1}^N \alpha_n^{N,R} \boldsymbol{\varphi}_n
\end{equation}
and 
$$
\hat{\boldsymbol{\lambda}}^{N,R} = \sum_{r=1}^R \beta_r^{N,R} \boldsymbol{\psi}_r,
$$
where $\boldsymbol{\alpha}^{N,R} =(\alpha_n^{N,R})_{1\leq n \leq N}$ and $\boldsymbol{\beta}^{N,R} =(\beta_n^{N,R})_{1\leq r \leq R}$. 

\medskip

More precisely, for a given $\mu \in \mathcal P$ and $1\leq \nu \leq N^T$, knowing the reduced-order model approximation $\hat{\mathbf{q}}^{N,R}(\nu-1, \mu)$ obtained from the previous time step, the reduced-order model approximation $\hat{\mathbf{q}}^{N,R}(\nu, \mu)$ is computed as
$$
\hat{\mathbf{q}}^{N,R}(\nu, \mu)  = \hat{\mathbf{q}}^{N,R}(\nu-1, \mu) + h \hat{\mathbf{u}}^{N,R}(\nu, \mu),
$$
where $\hat{\mathbf{u}}^{N,R}(\nu, \mu)$ is computed through the formula (\ref{eq:approxu}) with $(\boldsymbol{\alpha}^{N,R}, \boldsymbol{\beta}^{N,R})$ solutions of the reduced-order problem (\ref{eq:DCPred}) with $\mathbf{q}:= \hat{\mathbf{q}}^{N,R}(\nu-1, \mu)$.

\paragraph{Stability of the ROM}
As noted in Section \ref{sec:FOM}, the Lagrange multiplier is, in general, not unique (see \cite{maury2011discrete} for a graphical illustration of this non-uniqueness configuration). From a mathematical standpoint, this stems from the fact that the transposed constraint operator $f:\mathbb{R}^{\mathcal{R}}\rightarrow \mathbb{R}^{\mathcal{N}}$ defined by 
\[
f:\boldsymbol{\lambda} \mapsto \mathbf{B}_{\mathbf{q}}^T\boldsymbol{\lambda}
\]
is not injective in general, due to the over-constrained state of the system \eqref{eq:DCP} .
For this reason, problem \eqref{eq:DCP} is characterized by a stability constant
\begin{equation}
\gamma^{HF}(\sigma)=\inf_{\boldsymbol{\lambda} \in \mathcal{W}^+} \sup_{\mathbf{u} \in \mathcal V} \frac{\boldsymbol{\lambda}^T\mathbf{B}_{\mathbf{q}(\sigma)}  \mathbf{u}}{\|\mathbf{u}\|_{\mathcal{V}} \|\boldsymbol{\lambda}\|_{\mathcal{W}}}
\label{eq:HF_infsup}
\end{equation}
which evaluates to $0$. Indeed, in classical saddle-point theory for contact problems, a positive inf-sup constant guarantees i) uniqueness of contact forces and ii) the fact that the velocity basis is rich enough to satisfy any applied contact constraint.
To proceed with discussing the stability of the ROM, we analogously define the reduced inf-sup constant associated with a de-correlated construction of the reduced space $\hat{\mathcal{V}}_N$ and reduced cone $\hat{\mathcal{W}}_R^+$:
\begin{equation}
\hat{\gamma}^{N,R}(\sigma)=\inf_{\hat{\boldsymbol{\lambda}} \in \hat{\mathcal{W}}_R^+} \sup_{\hat{\mathbf{u}} \in \hat{\mathcal V}_N} \frac{\hat{\boldsymbol{\lambda}}^T\mathbf{B}_{\hat{\mathbf{q}}^{N,R}(\sigma)}  \hat{\mathbf{u}}}{\|\hat{\mathbf{u}}\|_{\mathcal{V}} \|\hat{\boldsymbol{\lambda}} \|_{\mathcal{W}}},
\label{eq:red_infsup_const}
\end{equation}
which is as well expected to be $0$ for different values of the pair $(N,R)$, as the reduced dual basis $\hat{\mathcal{W}}_R^+$ inherits the null-space forces present in the high-fidelity training data.
Even though multiplier uniqueness cannot be guaranteed, it remains physically mandatory that the primal velocity basis can satisfy the active constraints. To this end, we employ \emph{enrichment strategies}. Enrichment of the primal basis has been developed in several works in the literature: we refer to \cite{rovas2003reduced, rozza2007stability, niakh2023stable,haasdonk2012reduced}, whose common feature is the joint construction of the pair $(\hat{\mathcal{V}}_N,\hat{\mathcal{W}}^{+}_R)$. 
In this work, we rely on the Projected Gradient Algorithm (PGA) proposed in \cite{niakh2023stable}: at each iteration of PGA,  the primal basis is enriched by \emph{supremizers}, which represent the image of the dual basis under the transposed constraint operator, $\mathbf{B}_{\mathbf{q}}^T \hat{\mathcal{W}}^+_R$. The enriched primal basis is defined as $\hat{V}_N+S_R(\sigma) \subset \mathcal{V}$, where $S_R(\sigma):=\text{Span}(\{\mathbf{B}_{\mathbf{q}}^T \psi_r\}_{r \in \{1:R\}})$.
This supremizer space is constructed in a progressive way until the worst projection error $\sup_{v \in S_{R}(\sigma)}\|(\mathbb{I}-\texttt{P}_{V_N+S_R})(v)\|_{\mathcal{V}}$ drops below a strict physical tolerance $\delta$. 
Since $\mathbf{B}_{\mathbf{q}}$ is parameter-dependent, the corresponding supremizers are also parameter-dependent, meaning their exact evaluation would theoretically need to be done online. Instead, PGA is performed offline across the training snapshots to provide a sufficiently accurate, globally enriched velocity space that robustly spans the constraint dynamics for all evaluated parameters. We refer to \cite{niakh2023stable} for the analysis of the method and we remind to Section \ref{sec:num_res} for the numerical application of PGA on problem \eqref{eq:DCP}.

\subsection{Hyper-reduction}
\label{sec:hyper}
One major weakness of the RB method is related to the evaluation of components in the ROM that are associated with non-affine terms in parameters (at every time iteration). The assembly of these operators still leads to the resolution of a system of  size $\mathcal N+\mathcal R$: this is a well known limitation concerning the computational time efficiency gain that the RB can yield, since the possibility to devise an offline/online MOR decomposition relies on the assumption of affine parametric dependence.
The reduced system in \eqref{eq:DCM_ROM1}-\eqref{eq:DCM_ROM3} falls in this case. We briefly present two hyper-reduction techniques: the empirical interpolation method (EIM) (we refer to \cite{barrault2004empirical,chaturantabut2010nonlinear} for the description of the method and to \cite{benaceur2020reduced, newsum2017efficient} for applications to variational inequality problems; in particular, to \cite{fauque2018hybrid} for applications of EIM on contact problems) and the Empirical Quadrature (EQ) method \cite{antil2013two, farhat2015structure,yano2019lp}. 
For all $\sigma \in \mathcal S$, we denote by
$$
\mathbf{s}_1(\sigma):= \boldsymbol {\upsilon}_{\mathbf{q}(\sigma)} \in \mathbb{R}^{p_1}, \; \mathbf{s}_2(\sigma):= \mathbf{B}_{\mathbf{q}(\sigma)} \in \mathbb{R}^{p_2} \mbox{ and } \mathbf{s}_3(\sigma):= \mathbf{d}_{\mathbf{q}(\sigma)} \in \mathbb{R}^{p_3} 
$$
the non affine variables with $p_1 = \mathcal{N}$, $p_2 = \mathcal{N}\mathcal{R} $ and $p_3 = \mathcal{R}$ (notice here that we implicitly identified $\mathbb{R}^{ \mathcal{N}\times \mathcal{R}}$ with $\mathbb{R}^{ \mathcal{N}\mathcal{R}}$). 
\paragraph{Empirical Interpolation Method}
Let $k\in \{1, 2, 3\}$. For a given $m_k^{\rm EIM}\in \mathbb{N}^*$ (such that $m_k^{\rm EIM} \leq p_k$) and any $\sigma \in \mathcal S$, an EIM approximation of $\mathbf{s}_k(\sigma)$ is computed under the following form: 
\begin{equation}
\mathbf{s}_k(\sigma) \approx \sum_{m=1}^{m_k^{\rm EIM}} c_k^{(m)}(\sigma) \mathbf{g}_k^{(m)},
\label{eq:EIM}
\end{equation}
where $\left(\mathbf{g}_k^{(1)}, \ldots, \mathbf{g}_k^{(m_k^{\rm EIM})}\right)$ is a family of parameter-independent vectors of $\mathbb{R}^{p_k}$ and $\mathbf{c}_k(\sigma) := \left(c_k^{(m)}(\sigma)\right)_{1\leq m \leq m_k^{\rm EIM}} \in \mathbb{R}^{m_k^{\rm eim}}$ is a $m_k^{\rm EIM}$-dimensional vector of coefficients that is computed through the resolution of a linear system.  More precisely, the family $\{\mathbf{g}_k^{(m)}\}_{1\leq m \leq m_k^{\rm EIM}}$ is computed offline as the first $m_k^{\rm EIM}$ POD modes of the family of vectors $\{\mathbf{s}_k(\sigma)\}_{\sigma \in \mathcal S_{\rm train}}$. Let us introduce $\mathbf{G}_k = \left( \mathbf{g}_k^{(1)}|  \cdots | \mathbf{g}_k^{(m_k^{\rm EIM})} \right) \in \mathbb{R}^{p_k \times m_k^{\rm EIM}}$ and a given subset of indices $\mathcal I \subset \{1, \ldots, p_k\}$ such that we can denote $\mathbf{G}_{k, {\mathcal I}}:= \left( (\mathbf{g}_k^{(m)})_i\right)_{i\in \mathcal I, 1\leq m \leq m_k^{\rm EIM}} \in \mathbb{R}^{|\mathcal I| \times m_k^{\rm EIM}}$. 
In the offline phase, the EIM algorithm also computes as an output a subset of indices $\mathcal{I}_k \subset \{1, \ldots, p_k\}$ such that $|\mathcal{I}_k| = m_k^{\rm EIM}$ which is constructed via a greedy procedure. 
$$
\mathbf{G}_{k, \mathcal{I}_k} \mathbf{c}_k(\sigma) = (\mathbf{s}_k(\sigma))_{\mathcal{I}_k},
$$ 
where $(\mathbf{s}_k(\sigma))_{\mathcal{I}_k}:= \left( \mathbf{s}_k(\sigma)_{j}\right)_{j \in \mathcal{I}_k}\in \mathbb{R}^{m_k^{\rm eim}}$. 
\paragraph{Empirical Quadrature Method}
For each $k \in \{1,2,3\}$, instead of approximating $\mathbf{s}_k(\sigma)$ in the full space $\mathbb{R}^{p_k}$ as in EIM (see \eqref{eq:EIM}), the empirical quadrature (EQ) method seeks to approximate directly the \emph{projected} quantities that appear in the reduced Uzawa system \eqref{eq:discr_Uzawa_system}:
\begin{equation}
\begin{aligned}
  &  \hat{\boldsymbol{\upsilon}}(\sigma) := \hat{V}^{T}_N\mathbf{s}_1(\sigma) \in \mathbb{R}^N, \\
&    \hat{\mathbf{B}}_{\mathbf{q}}(\sigma) := {\mathcal{W}_{R}^{+}}^T\mathbf{s}_2(\sigma) \hat{V}_N \in \mathbb{R}^{R \times N}, \\
 &   \hat{\mathbf{d}}_{\mathbf{q}}(\sigma) := {\mathcal{W}_{R}^{+}}^T\mathbf{s}_3(\sigma) \in \mathbb{R}^R.
    \end{aligned}
    \label{eq:EQ_projected}
\end{equation}
Introducing the notation $\Pi_k(\sigma) \in \mathbb{R}^{q_k}$ for the projected quantity associated with $\mathbf{s}_k$, with $q_1 = N$, $q_2 =RN$ and $q_3 = R$, we have the following decomposition into a a sum of $p_k$ individual contributions:
\begin{equation}
    \Pi_k(\sigma) = \sum_{i=1}^{p_k} \left(\mathbf{s}_k(\sigma)\right)_i \, \mathbf{f}_k^{(i)}(\sigma) \in \mathbb{R}^{q_k},
    \label{eq:EQ_decomp}
\end{equation}
where $\mathbf{f}_k^{(i)} \in \mathbb{R}^{q_k}$ is the contribution of the $i$-th component of $\mathbf{s}_k$ to the projected quantity. Explicitly:
\begin{equation}
\begin{aligned}
   & \mathbf{f}_1^{(i)} = {(\hat{V}_N)}_{i,:} \in \mathbb{R}^N, \\
&    \mathbf{f}_2^{(i)} = \mathrm{vec}\!\left( {(\hat{W}^{+}_{R})}_{i,:}\otimes (\mathbf{b}_i(\sigma) \hat{V_N}) \right) \in \mathbb{R}^{RN}, \\
 &   \mathbf{f}_3^{(i)} = (\hat{W}^{+}_R)_{i,:} \in \mathbb{R}^R,
    \end{aligned}
\end{equation}
where $\mathbf{b}_i(\sigma)$ denotes the $i$-th row of $\mathbf{B}_{\mathbf{q}(\sigma)}$.

The EQ method seeks a sparse index set $\mathcal{I}_k \subset \{1, \ldots, p_k\}$ with $|\mathcal{I}_k| = m_k^{\rm EQ} \ll p_k$, and weights $\mathbf{w}_k = (w_k^{(i)})_{i \in \mathcal{I}_k} \in \mathbb{R}^{m_k^{\rm EQ}}$ such that:
\begin{equation}
    \Pi_k(\sigma) \approx \sum_{i \in \mathcal{I}_k} w_k^{(i)} \left(\mathbf{s}_k(\sigma)\right)_i \mathbf{f}_k^{(i)}(\sigma), \qquad \forall \sigma \in \mathcal{S}.
    \label{eq:EQ_approx}
\end{equation}
The EIM approximation \eqref{eq:EIM} is formulated in the full space $\mathbb{R}^{p_k}$, whose dimension scales with the number of degrees of freedom $\mathcal{N}$ or the number of contact pairs $\mathcal{R}$. In contrast, the EQ approximation \eqref{eq:EQ_approx} is defined directly in the reduced space $\mathbb{R}^{q_k}$, whose dimension is independent of $p_k$. We remark that we make here a precise choice of the weights sign. For $k \in \{1, 3\}$, the weights $\mathbf{w}_k$ are allowed to take arbitrary sign, since the spontaneous velocity $\mathbf{s}_1(\sigma)$ and the distance vector $\mathbf{s}_3(\sigma)$ can take positive and negative values. For $k = 2$, the weights are constrained to be non-negative, $w_2^{(i)} \geq 0$, motivated by the physical interpretation of the contact matrix.
\paragraph{Offline phase.}
We compute offline the index set $\mathcal{I}_k$ and weights $\mathbf{w}_k$ by means of a greedy algorithm known in the literature under the name of Non-Negative Orthogonal Matching Pursuit (NNOMP) \cite{zhang2011sparse}. The algorithm operates on the Gramian matrix $\mathbf{H}_k \in \mathbb{R}^{p_k \times p_k}$ and the correlation vector $\mathbf{c}_k \in \mathbb{R}^{p_k}$
\begin{equation}
\begin{aligned}
    &\mathbf{H}_k = \sum_{\sigma \in \mathcal{S}_{\rm train}} \mathbf{G}_k(\sigma)^\top \mathbf{G}_k(\sigma) \in \mathbb{R}^{p_k \times p_k}, \\
   & \mathbf{c}_k = \sum_{\sigma \in \mathcal{S}_{\rm train}} \mathbf{G}_k(\sigma)^\top \boldsymbol{\Pi}_k(\sigma) \in \mathbb{R}^{p_k},
    \end{aligned}
    \label{eq:EQ_gramian}
\end{equation}
where $\mathbf{G}_k(\sigma) \in \mathbb{R}^{q_k \times p_k}$ has columns $\mathbf{f}_k^{(i)}(\sigma)$. The quantities in  \eqref{eq:EQ_gramian} are assembled at each iteration from training snapshots without forming the full dense matrix $\mathbf{G}_k \in \mathbb{R}^{q_k \times p_k}$ (which would be computationally prohibitive for high-dimensional contact data.). 
The offline phase is summarized in Algorithm~\ref{alg:EQ_offline}.
\begin{algorithm}[h!]
\caption{Empirical Quadrature for the DCP --- offline phase}
\label{alg:EQ_offline}
\begin{algorithmic}[1]
    \STATE \textbf{Inputs}: $k\in\{1,2,3\}$, snapshots $\{\mathbf{s}_k(\sigma)\}_{\sigma \in \mathcal{S}_{\rm train}}$, $m_k^{\rm EQ}$, reduced matrices $\hat{V}_N$, $\hat{W}^+_R$, tolerance $\varepsilon_{\rm eq}$
    \STATE \textbf{Outputs}: $\mathcal{I}_k \subset \{1,\ldots,p_k\}$, weights $\mathbf{w}_k \in \mathbb{R}^{m_k^{\rm EQ}}$
    \STATE Assemble $\mathbf{H}_k$ and $\mathbf{c}_k$ via \eqref{eq:EQ_gramian}
    \STATE Initialize: $\mathcal{I}_k = \emptyset$, $\mathbf{r} = \mathbf{c}_k$
    \FOR{$m = 1, \ldots, m_k^{\rm EQ}$}
        \STATE Select best index:
        \begin{equation*}
            i^{(m)} = \begin{cases}
                \displaystyle\arg\max_{i \notin \mathcal{I}_k} \; \left(\mathbf{H}_k\right)_{:,i}^\top \mathbf{r} & \text{if } k = 2 \;\; (\text{non-negative weights}), \\[6pt]
                \displaystyle\arg\max_{i \notin \mathcal{I}_k} \; \left|\left(\mathbf{H}_k\right)_{:,i}^\top \mathbf{r}\right| & \text{if } k \in \{1,3\} \;\; (\text{signed weights}).
            \end{cases}
        \end{equation*}
        \STATE Update active set: $\mathcal{I}_k \leftarrow \mathcal{I}_k \cup \{i^{(m)}\}$
        \STATE Solve local optimization problem on active set:
        \begin{equation*}
            \mathbf{w}_k \leftarrow \begin{cases}
                \displaystyle\arg\min_{\mathbf{w} \geq 0} \;
                \left\| \left(\mathbf{H}_k\right)_{\mathcal{I}_k,\mathcal{I}_k} \mathbf{w} - \left(\mathbf{c}_k\right)_{\mathcal{I}_k} \right\|_2^2 & \text{if } k=2, \\[6pt]
                \left(\mathbf{H}_k\right)_{\mathcal{I}_k,\mathcal{I}_k}^{-1} \left(\mathbf{c}_k\right)_{\mathcal{I}_k} & \text{if } k \in \{1,3\}.
            \end{cases}
        \end{equation*}
        \STATE Update residual: $\mathbf{r} \leftarrow \mathbf{c}_k - \left(\mathbf{H}_k\right)_{:,\mathcal{I}_k} \mathbf{w}_k$
        \IF{$\|\mathbf{r}\|_2 \leq \varepsilon_{\rm eq}$}
            \STATE \textbf{break}
        \ENDIF
    \ENDFOR
    \STATE Store $\mathcal{I}_k$ and $\mathbf{w}_k$
\end{algorithmic}
\end{algorithm}

\paragraph{Online phase.}
In the online phase, the full assembly of $\mathbf{s}_k(\sigma)$ is avoided. Only the $m_k^{\rm EQ}$ components $(\mathbf{s}_k(\sigma))_i$ for $i \in \mathcal{I}_k$ are evaluated at cost $\mathcal{O}(m_k^{\rm EQ})$, and the projected quantities are assembled directly in the reduced space via the sparse weighted sums \eqref{eq:EQ_approx}. The online phase is summarized in Algorithm~\ref{alg:EQ_online}.

\begin{algorithm}[h!]
\caption{Empirical Quadrature for the DCP  --- online phase}
\label{alg:EQ_online}
\begin{algorithmic}[1]
    \STATE \textbf{Inputs}: $\sigma \in \mathcal{S}$, index sets $\mathcal{I}_k$ and weights $\mathbf{w}_k$ for $k \in \{1,2,3\}$, matrices $\hat{V}_N$, $\hat{W}^+_R$
    \STATE \textbf{Outputs}: $\hat{\boldsymbol{\upsilon}}_{\mathbf{q}}(\sigma)$, $\hat{\mathbf{B}}_{\mathbf{q}}(\sigma)$, $\hat{\mathbf{d}}_{\mathbf{q}}(\sigma)$
    \STATE Evaluate $\left(\mathbf{s}_k(\sigma)\right)_i$ only for $i \in \mathcal{I}_k$, $k=1,2,3$ 
    \STATE Compute:
    \begin{align*}
        \hat{\boldsymbol{\upsilon}}_{\mathbf{q}}(\sigma) &= \sum_{i \in \mathcal{I}_1} w_1^{(i)} \left(\mathbf{s}_1(\sigma)\right)_i {(\hat{V}_N)}_{i,:}^\top, \\[4pt]
        \widehat{\mathbf{B}}_{\mathbf{q}}(\sigma) &= \sum_{i \in \mathcal{I}_2} w_2^{(i)} \left(\hat{W}_R^{+}\right)_{i,:}^\top \otimes \left({(\mathbf{s}_2)}_i(\sigma) \hat{V}_N\right), \\[4pt]
        \hat{\mathbf{d}}_{\mathbf{q}}(\sigma) &= \sum_{i \in \mathcal{I}_3} w_3^{(i)} \left(\mathbf{s}_3(\sigma)\right)_i {(\hat{W}_R^{+})}_{i,:}^\top.
    \end{align*}
\end{algorithmic}
\end{algorithm}

The total online costs of  EIM and EQ methods are both independent of the problem size $p_k$; taking the example of $s_2$, they are in the order of  $\mathcal{O}(m_k^{\rm EIM} \cdot R \cdot N)$ per time step for EIM and of $\mathcal{O}(m_2^{\rm EQ} \cdot R \cdot N)$  per time step for EQ. We comment in section \ref{sec:num_res} on the numerical performance of both approaches on the discrete contact model \eqref{eq:DCP}.
\subsection{A machine-learning corrected ROM}
\label{sec:ML_correction}
As illustrated in Section \ref{sec:pod}, it is expected that the parametric DCP has a slowly decaying Kolmogorov $n$-width in general: in this case, a major gain in terms of accuracy is expected by the application of appropriate nonlinear reduced models. We present in this section the approach we propose in this work, which is inspired by the work in \cite{cohen2023nonlinear, barnett2023neural}. More precisely, rather than enriching the trial space by predicting unresolved tail components, as done in the former references, we have to deal with the fact that the enriched space is supremized, thus a simple regression of the discarded POD modes is not straightfoward in this case. We seek a solution reconstruction strictly within the truncated $n$-dimensional subspace $\hat{V}_{n} = \text{Span}(\{\varphi_i\}_{i \in \{1:n\}})$. The reconstruction relies on two components: i) the standard (parameter-dependent) RB coefficients $\alpha^{N,R}$ which are solutions to \eqref{eq:DCM_ROM1}-\eqref{eq:DCM_ROM3} and ii) a non-linear parametric map designed to learn the drift between the standard Galerkin approximation and the optimal projection error. 
The map takes as input the ROM coefficients $\{\alpha^{N,R}\}_{i \in \{1:n\}}$ (possibly augmented with the geometric parameters $\mu$ and the temporal variable $\nu$), truncated at a chosen dimension $n<N$, and gives as output the optimal correction vector $\Delta \boldsymbol{\alpha} \in \mathbb{R}^n$ for all $\sigma=(\mu, \nu)$. 
The reconstructed, ML-corrected velocities assume the following form: \[ \hat{\mathbf{u}}(\sigma)=\sum_{i=1}^n \left( \alpha^{N,R}_i(\sigma) + \left(\Psi(\boldsymbol{\alpha}^{N,R}(\sigma), \sigma) \right)_i \right) \varphi_i \]
The approximation is then computed on the particle positions by using the explicit time scheme
$\hat{\mathbf{q}}^{N,R}(\nu,\mu)=\hat{\mathbf{q}}(\nu-1,\mu)+h \hat{\mathbf{u}}(\nu,\mu)$.
A key aspect in this reduction setting is represented by the learning procedure of the non-linear map $\Psi$: we empirically observed that Random Forest regression is less prone to overfitting and presents a smaller number of tuning hyper-parameters than multi-layer-perceptron architecture. The numerical tests thus rely on the former technique. The comparison among different architecture for the ML correction is beyond the scope of this work.
We also remark that the procedure preserve the ROM inf-sup stability (at least in the training parametric set, as guaranteed by PGA). Indeed, to guarantee the inf-sup stability of the contact mechanics saddle-point problem, the ROM is strictly solved in the full, PGA-enriched primal-dual spaces $(\hat{V}_N, \hat{W}^+_R)$. However, to construct a highly efficient and low-dimensional nonlinear ML map $\Psi$, we only extract the first $n$ dominant coefficients of the stable ROM solution. The ML map acts as a post-processing map that correlates these $n$ dominant stable modes to the optimal high-fidelity projection. 

 \section{Numerical results}
\label{sec:num_res}
\paragraph{Problem setup.}
We consider a training parametric set of dimension $p_{\rm train}= |\mathcal P_{\rm train}|= 200$
 so that $\mathcal P_{\rm train} = \{\mu_p\}_{p=1}^{p_{\rm train}}$ which were chosen following a random procedure: the variation of the  the exit width $l_{\rm exit}$ and the magnitude  $c_{\boldsymbol{\upsilon}}$ of the  spontaneous velocity $\boldsymbol{\upsilon}$ follows
\begin{itemize}
\item $l_{\rm exit} \sim \text{Uniform}([\bar{l}_{\rm exit}-10\%\bar{l}_{\rm exit},\bar{l}_{\rm exit}+10\%\bar{l}_{\rm exit}])$,
\item $(c_{\boldsymbol{\upsilon}})_n \sim \text{Uniform}([(\bar{c_{\boldsymbol{\upsilon}}})_i-10\%(\bar{c_{\boldsymbol{\upsilon}}})_n,(\bar{c_{\boldsymbol{\upsilon}}})_n+10\%(\bar{c_{\boldsymbol{\upsilon}}})_n])$, for $n=1, \ldots, \mathcal N$,
\end{itemize}
 with $\bar{l}_{\rm exit}$ and $\bar{\boldsymbol{\upsilon}}$ some  prescribed reference values.
For each parameter, the FOM \eqref{eq:discr_Uzawa_system} is characterized by a termination condition of $\epsilon=10^{-12}$ on the relative error between two consecutive Uzawa solutions and a maximum number of iterations equal to $10^{7}$. To take into account the lack of contacts for some time instants, the following criterion on the Lagrange multipliers error is evaluated at each iteration of \eqref{eq:discr_Uzawa_system}:
$$\frac{\|\boldsymbol{\lambda}^{k}-\boldsymbol{\lambda}^{k-1}\|_{\mathcal W}}{\|\boldsymbol{\lambda}^{k-1}\|_{\mathcal W}+1}\leq \epsilon.$$
The gradient step in \eqref{eq:Uzawa_lambda} is fixed and it is chosen inside the convergence interval (cf. section \ref{sec:FOM}): in particular, we choose $\rho=\frac{0.2}{h^2}$.
The same settings is used for the ROM \eqref{eq:DCM_ROM1}-\eqref{eq:DCM_ROM3}. \\
We assess performance based on a parametric set $\mathcal P_{\rm valid}$ of cardinality $p_{\rm valid}$ generated using the same distributions as for the training set. We consider the geometric setting depicted in Figure~\ref{fig:region}: it is characterized by a region of size $l_{\rm hall}L_{\rm hall}$---where the agents are initially placed---and two obstacles of length $l_{\rm wall}$. After having passed through the barrier walls, the crowd is counted out of the area of interest.
The particles positions initialization relies i) on a random placement of the particles, ii) the setup of local constraints for both inter-particle distances and the bounding box walls, iii) the execution of one Uzawa iteration to separate the overlapping particles within $1000$ iterations. This procedure  avoids the  $\mathcal{O}(({N^{\rm{a}}})^2)$  loop generating and rejecting thousands of samples for dense configurations and it is thus suitable to highly-congested simulations.\\
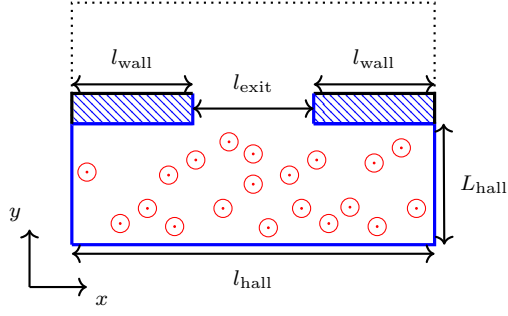
\begin{figure}[h!]
\centering
    \begin{tikzpicture}[scale=0.8]
\draw[very thick, blue] (0,0) -- (6,0) -- (6,2) -- (4,2); 
\draw[very thick, blue] (2,2) -- (0,2) -- (0,0);
\draw[white, pattern=north west lines, pattern color=blue] (4,2) -- (6,2)--(6,2.5)-- (4,2.5)--(4,2);
\draw[blue,very thick, pattern=north west lines, pattern color=blue] (4,2) -- (6,2) ;
\draw[very thick] (6,2)--(6,2.5)-- (4,2.5) ;
\draw[blue,very thick, pattern=north west lines, pattern color=blue] (4,2.5)--(4,2);
\draw[white,very thick, pattern=north west lines, pattern color=blue] (0,2) -- (2,2) -- (2,2.5) -- (0,2.5) -- (0,2); 
\draw[blue,very thick](0,2) -- (2,2) -- (2,2.5);
\draw[very thick] (2,2.5) -- (0,2.5) -- (0,2);
\draw[thick,dotted](6,2) -- (6,4) -- (0,4) -- (0,2);  
\draw[thick,->] (-0.7,-0.7) -- (0.25,-0.7) node[anchor=north west] {$x$};
\draw[thick,->] (-0.7,-0.7) -- (-0.7,0.25) node[anchor=south east] {$y$};
\draw [black, thick,<->] (0,-0.15) -- node[below=1mm] {$l_{\rm hall}$} (6,-0.15);

\draw [black, thick,<->] (6.16,0) -- node[right=1mm] {$L_{\rm hall}$}  (6.15,2);
\draw [black, thick,<->] (4,2.65) -- node[above=1mm] {$l_{\rm wall}$}(6,2.65);
\draw [black, thick,<->] (0,2.65) -- node[above=1mm] {$l_{\rm wall}$}(2,2.65);
\draw [black, thick,<->] (2,2.25) -- node[above=1mm] {$l_{\rm exit}$}(4,2.25);
\draw[red] (3,1) circle (0.15cm);
\draw[red] (3,1.5) circle (0.15cm);
\draw[red] (2.6,1.7) circle (0.15cm);
\draw[red] (0.25,1.2) circle (0.15cm);
\draw[red] (1.25,0.6) circle (0.15cm);
\draw[red] (0.8,0.35) circle (0.15cm);
\draw[red] (1.25+0.8,0.6+0.8) circle (0.15cm);
\draw[red] (0.8+0.8,0.35+0.8) circle (0.15cm);
\draw[red] (1.25+2.8,0.6+0.8) circle (0.15cm);
\draw[red] (0.8+2.8,0.35+0.8) circle (0.15cm);
\draw[red] (1.25+4.2,0.6+1) circle (0.15cm);
\draw[red] (0.8+4.2,0.35+1) circle (0.15cm);
\draw[red] (1.25+3,0.35) circle (0.15cm);
\draw[red] (0.8+3,0.6) circle (0.15cm);
\draw[red] (1.25+3.8,0.3) circle (0.15cm);
\draw[red] (0.8+3.8,0.62) circle (0.15cm);
\draw[red] (5.7,0.6) circle (0.15cm);
\draw[red] (1.7,0.3) circle (0.15cm);
\draw[red] (2.5,0.6) circle (0.15cm);
\draw[red] (3.25,0.28) circle (0.15cm);
\filldraw [red] (3,1) circle (0.4pt);
\filldraw [red] (3,1.5) circle (0.4pt);
\filldraw [red] (2.6,1.7) circle (0.4pt);
\filldraw [red] (0.25,1.2) circle (0.4pt);
\filldraw [red] (1.25,0.6) circle (0.4pt);
\filldraw [red] (0.8,0.35) circle (0.4pt);

\filldraw [red] (1.25+0.8,0.6+0.8) circle (0.4pt);
\filldraw [red] (0.8+0.8,0.35+0.8) circle (0.4pt);
\filldraw [red] (1.25+2.8,0.6+0.8) circle (0.4pt);
\filldraw [red] (0.8+2.8,0.35+0.8) circle (0.4pt);
\filldraw [red] (1.25+4.2,0.6+1) circle (0.4pt);
\filldraw [red] (0.8+4.2,0.35+1) circle (0.4pt);
\filldraw [red] (1.25+3,0.35) circle (0.4pt);
\filldraw [red] (0.8+3,0.6) circle (0.4pt);
\filldraw [red] (1.25+3.8,0.3) circle (0.4pt);
\filldraw [red] (0.8+3.8,0.62) circle (0.4pt);

\filldraw [red] (5.7,0.6) circle (0.4pt);
\filldraw [red] (1.7,0.3) circle (0.4pt);
\filldraw [red] (2.5,0.6) circle (0.4pt);
\filldraw [red] (3.25,0.28) circle (0.4pt);
\end{tikzpicture}
\caption{Problem setup. Representative sketch of the two-dimensional DCP region: obstacles and contact walls are marked in blue; particles are marked in red.}
\label{fig:region}
\end{figure}
In this first study case, we set the number of particles equal to $N^{\rm a}=20$, their radius $r^{\rm a}=L_{\rm hall}/20$ and we set  $N^{\rm obst}=2$. The total number of potentially active contacts of the setting in Figure~\ref{fig:region} is given by $\mathcal R =N^{\rm cont}=\frac{N^{\rm a} (N^{\rm a}-1)}{2}+N^{\rm a}N^{\rm obst}=210$.
To assess the accuracy of ROM, we define a out-of-sample prediction error
$E_{\rm avg}=\frac{1}{|\mathcal P_{\rm valid}|} \sum_{\mu \in \mathcal P_{\rm valid}} E^q_{\mu}$
where the $\mu-$dependent relative error is computed as
$$E_{\mu}(\mathbf{q}, \hat{\mathbf{q}}^{N,R}):=\frac{\sqrt{\sum_{\nu=1}^{N^T} (t^{\nu}-t^{\nu-1})\left\|\mathbf{q}(\nu,\mu)-\hat{\mathbf{q}}^{N,R}(\nu,\mu)\right\|_{\mathcal V}^2}
}{\sqrt{\sum_{\nu=1}^{N^T} (t^{\nu}-t^{\nu-1})\left\|\mathbf{q}(\nu,\mu)\right\|_{\mathcal V}^2}}$$
(we used the same norm definition for the velocities set $\mathcal{V}$ and the particles set $\mathcal{Q}$). 
We recall that the corresponding training geometric-time completed parameters are denoted by $\sigma \in \mathcal{S}_{\rm{train}}=\{0, \ldots, N^{T}\} \times \mathcal{P}_{\rm{train}}$. In this first study case, $N^T$is not fixed, since $N^{\rm{a}}$ is not huge and we can afford following the particle dynamics until the crowd exits the hall region (see Figure \ref{fig:region}) . 
\paragraph{Convergence of Uzawa scheme}
\begin{figure}[h]
	\centering
\includegraphics[scale=0.35]{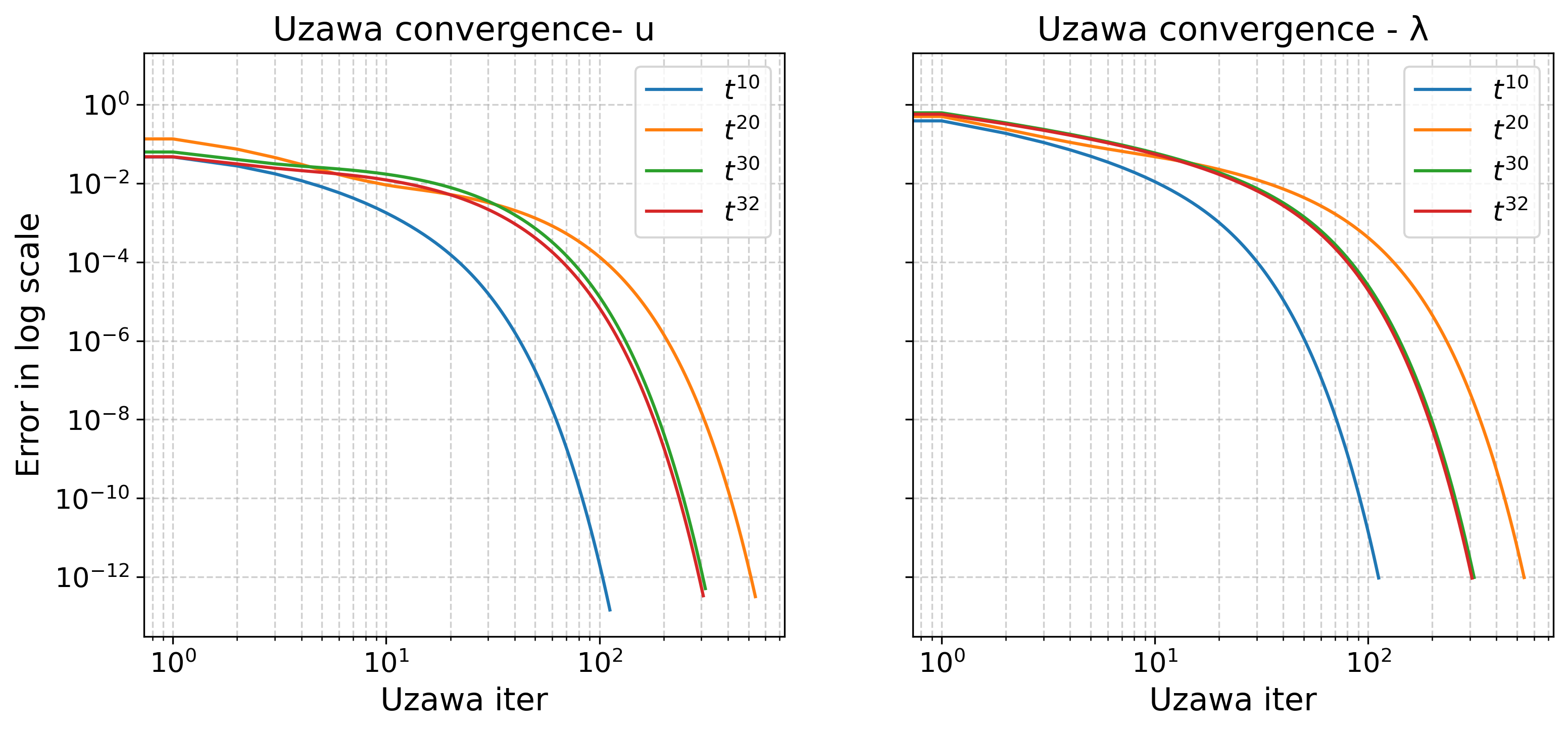}
\caption{Convergence of Uzawa scheme. Convergence history of the HF Uzawa solve for a representative training parameter and selected time steps.}
\label{fig:Uzawa_conv_hist}
\end{figure}
In Figure~\ref{fig:Uzawa_conv_hist} we display the convergence history of Uzawa scheme for the FOM \eqref{eq:discr_Uzawa_system}: in particular, the iteration errors for different time steps and for both the velocities and Lagrange multipliers. 
At  time step  $t^{10}$ the Uzawa scheme requires a smaller number of iterations than in the sequent times; also at $t^{30}$ the number of iterations is smaller than the most recent previous times: the latter cases correspond to the situation where the majority of agents has overcome the obstacles region. We expect the average number of iterations to dramatically increase with the number of agents for all the time steps: we postpone to a future work the investigation of Uzawa convergence with respect to $N^{\rm a}$.
We remark that the ROM completely bypasses the iterative Uzawa gradient approach in favor of solving a small, constrained quadratic problem directly using Non-Negative Least Squares using the python library \texttt{scipy.optimize}.
\paragraph{Dual reduced cone construction.}
We compare the performance of Algorithm~\ref{alg:mCPG} (dubbed mCPG) and Algorithm~\ref{alg:gIS} (dubbed gIS) proposed in section \ref{sec:dual_basis} for the construction of the dual reduced cone $\hat{\mathcal{W}}^+$. In Figure \ref{fig:dual_basis_constr}, we show the maximum projection errors on the Lagrange multipliers for training parameters in $\mathcal{S}_{\rm train}$.
\begin{figure}[h]
\subfloat[Maximum projection error for increasing dimension $r$ (cf. line \ref{alg_line:proj_err} of Algorithm~\ref{alg:mCPG} and line \ref{alg_line:gIS_max_err} of Algorithm~\ref{alg:gIS}).]{
\includegraphics[scale=0.28]{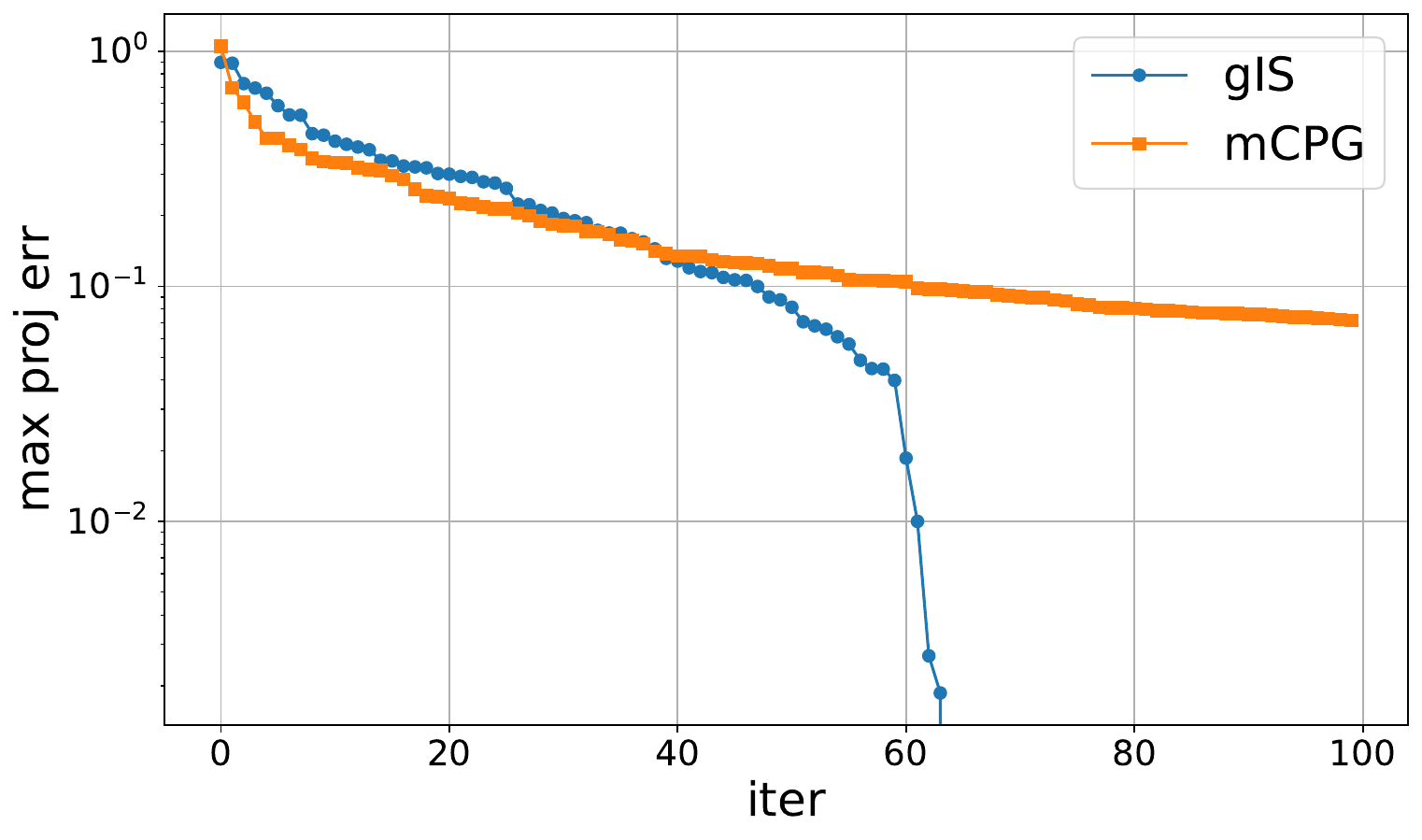}
\label{fig:dual_basis_constr}
}
\subfloat[Projection error for in-sample predictions]{\includegraphics[scale=0.28]{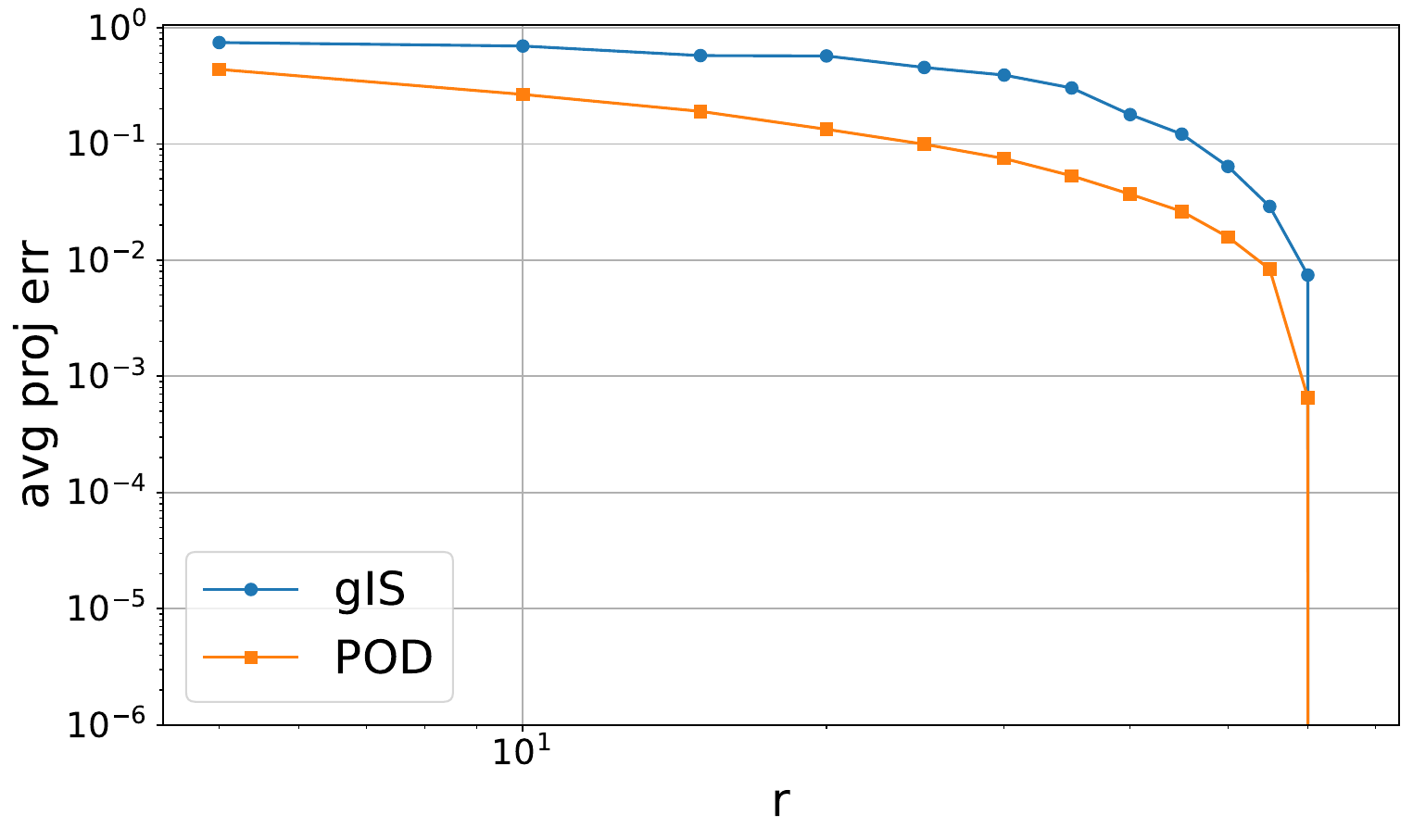}
\label{fig:dual_basis_compare}
}
\caption{Dual reduced cone construction. (a): construction of the reduced cone for the Lagrange multipliers; (b): comparison with POD for $\sigma \in \mathcal{S}_{\rm{train}}$. }
\label{fig:dual_basis}
\end{figure}
We observe that mCPG and gIS errors are comparable for approximately $r<45$; after that value, the gIS error curve is characterized by a significant decay towards $0$, (for $r=65$) rather than that of mCPG (the latter requires a number of iterations $r\gg100$ to achieve the same accuracy). The reduced basis provided by gIS method turns out to be more suitable to deal with the highly sparse snapshots $\{\boldsymbol{\lambda}(\sigma)\}_{\sigma \in \mathcal S_{\rm train}}$; furthermore, the computational time required by the gIS is $8$ times lower than the one required by mCPG, the latter being based on the solution of two constraint minimization problems at each iteration. 
For completeness of the results, the average errors for $\sigma \in \mathcal{S}_{
\rm train}$ are also computed in the case of both the gIS and POD algorithms: as we expected, the gIS construction is suboptimal compared to the POD; however, the minimum error is achieved by both algorithms for the same maximum dimension of the reduced cone (see Figure \ref{fig:dual_basis}\protect\subref{fig:dual_basis_compare}). In Appendix \ref{sec:hertz_appendix}, we further investigate, both theoretically and numerically, the performance of a greedy search strategy for contact pressure in a sphere–plane Hertz case.

\paragraph{Stability of the projection-based ROM}
\label{subsec:stability}
We consider here the stability of the ROM \eqref{eq:DCM_ROM1}-\eqref{eq:DCM_ROM3}. In Figure \ref{fig:infsup} (a) we compute the reduced inf-sup constants for a subset of the training parameters in $\mathcal{S}_{\rm train}$: in particular we chose $n_{\rm{train}}=80$. The heatmap shows the progressive loss of inf-sup stability for several primal-dual pairs, in particular when $R>N$. We further show in Figure \ref{fig:infsup}(b) the PGA algorithm convergence for two values of the projection error threshold $\delta \in \{10^{-4}, 10^{-8}\}$  and  primal-dual pair dimension $(10,15)$. In Figure \ref{fig:infsup}(c), we also show the enhanced inf-sup stability achieved by PGA algorithm also for out-of-sample parameters in $\mathcal{P}_{\rm{valid}}$. The reduced inf-sup constant for the pair $(10,15)$ without supremizer enrichment is depicted in red and is equal to $0$ for all the parameters.

\begin{figure}[h]
\centering
\subfloat[Reduced inf-sup constants for $80$ training parameters.]{
	\includegraphics[scale=0.32]{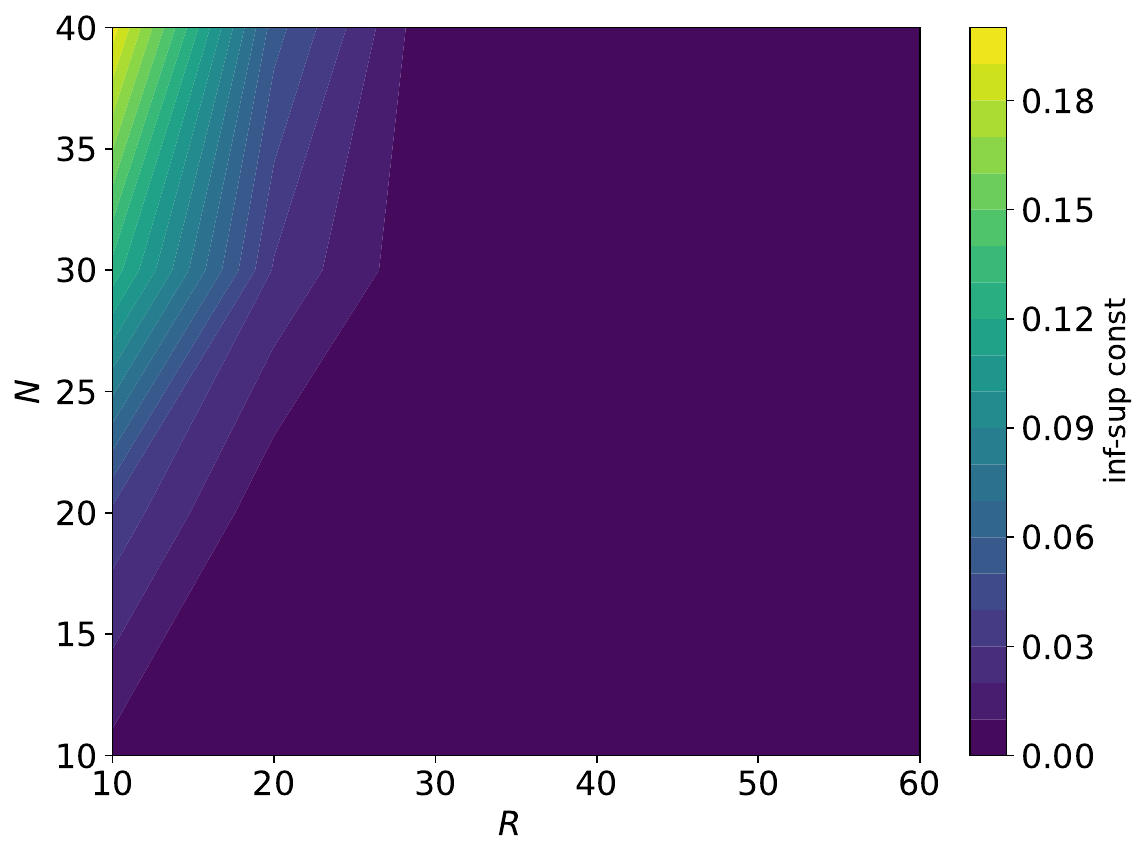}
}
\quad
\subfloat[PGA convergence for $N=10$, $R=15$ ]{
\includegraphics[scale=0.3]{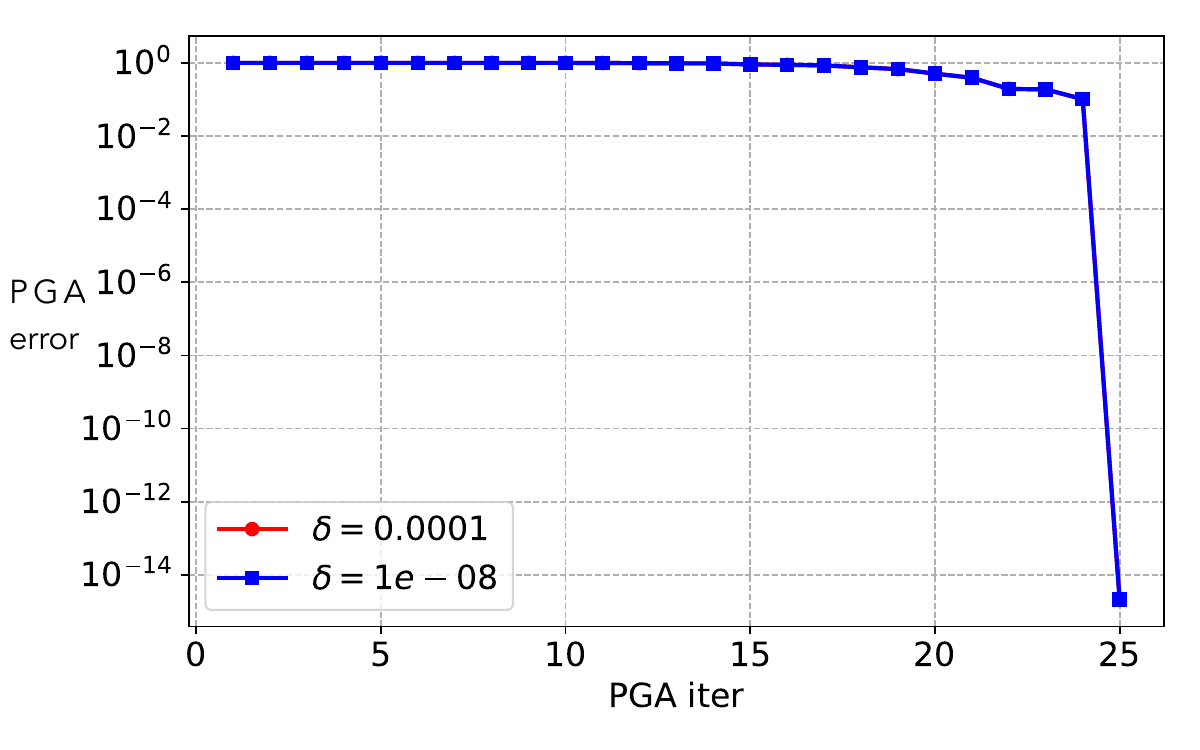}
}
 \subfloat[$\hat{\beta}^{\rm{dec}}$  and $\hat{\beta}^{\rm{on}}$ for $\mu \in \mathcal{P}_{\rm{valid}}$ and $N=10$, $R=15$ ]{
 		\includegraphics[scale=0.3]{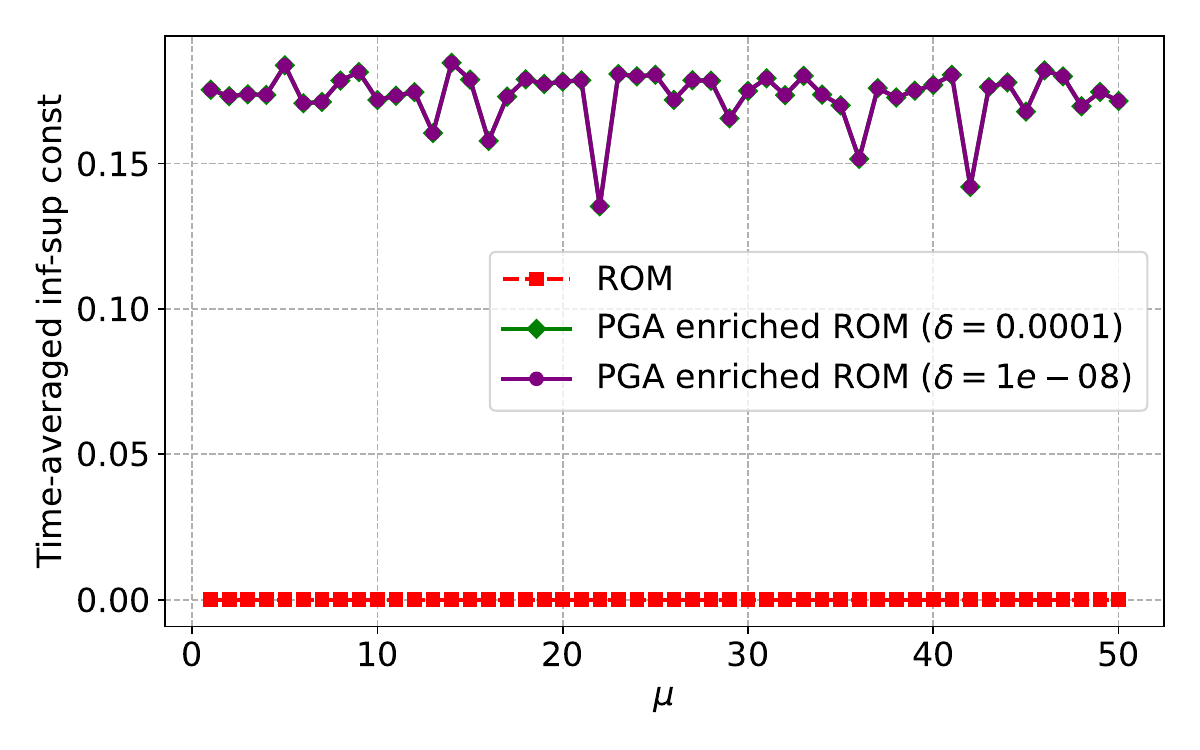}
 }
 \caption{Stability of the reduced system and convergence of the PGA algorithm.}
 \label{fig:infsup}
\end{figure}

\paragraph{Hyper-reduced model}
\label{sec:numres_hyper}
In Figure \ref{fig:eim_pod} we depict the normalized POD eigenvalues for increasing values of $m^{\rm{EIM}}=1, \ldots, m_k^{\rm{EIM}}$ for $k=1, 2,3$. The decay of the eigenvalues is rather slow, especially for the quantity $s_2$, which corresponds to the contact matrices. For instance, achieving a projection error for $s_2$ smaller than $10^{-3}$ would require more than $150$ modes. 

\begin{figure}[h]
	\centering
	\includegraphics[scale=0.38]{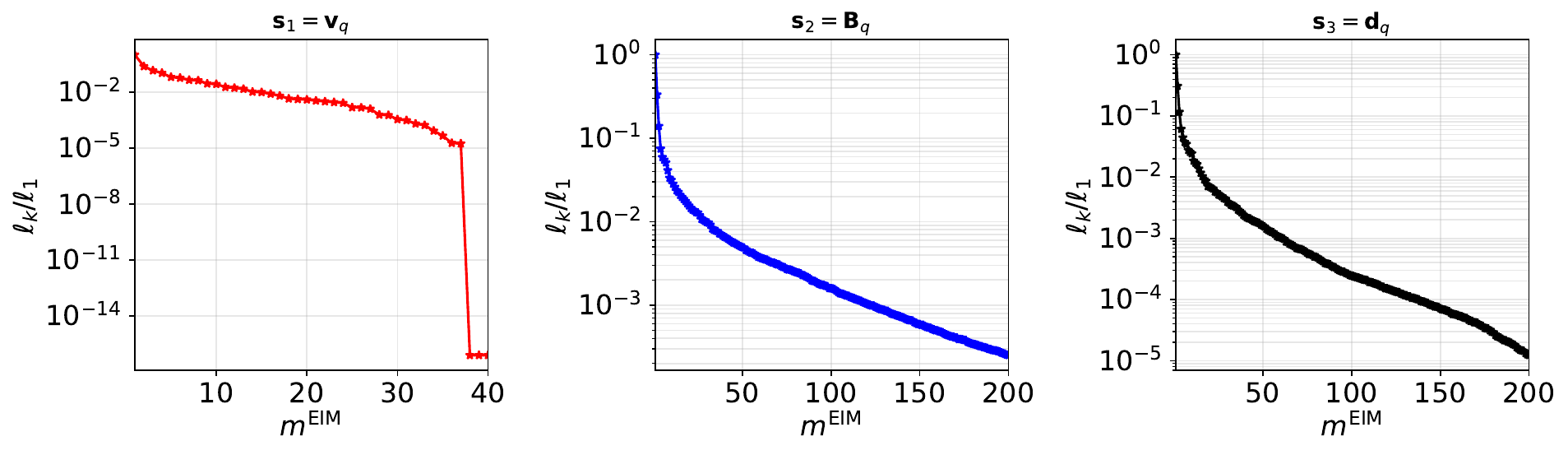}
	\caption{Hyper-reduced model . Empirical interpolation method for the DCP. The POD singular values decay associated with $\mathbf{s}_1$, $\mathbf{s}_2$ and $\mathbf{s}_3$ is depicted for increasing values of $m^{\rm{EIM}}$.}
	\label{fig:eim_pod}
\end{figure}

\begin{figure}[h]
\centering
\includegraphics[scale=0.36]{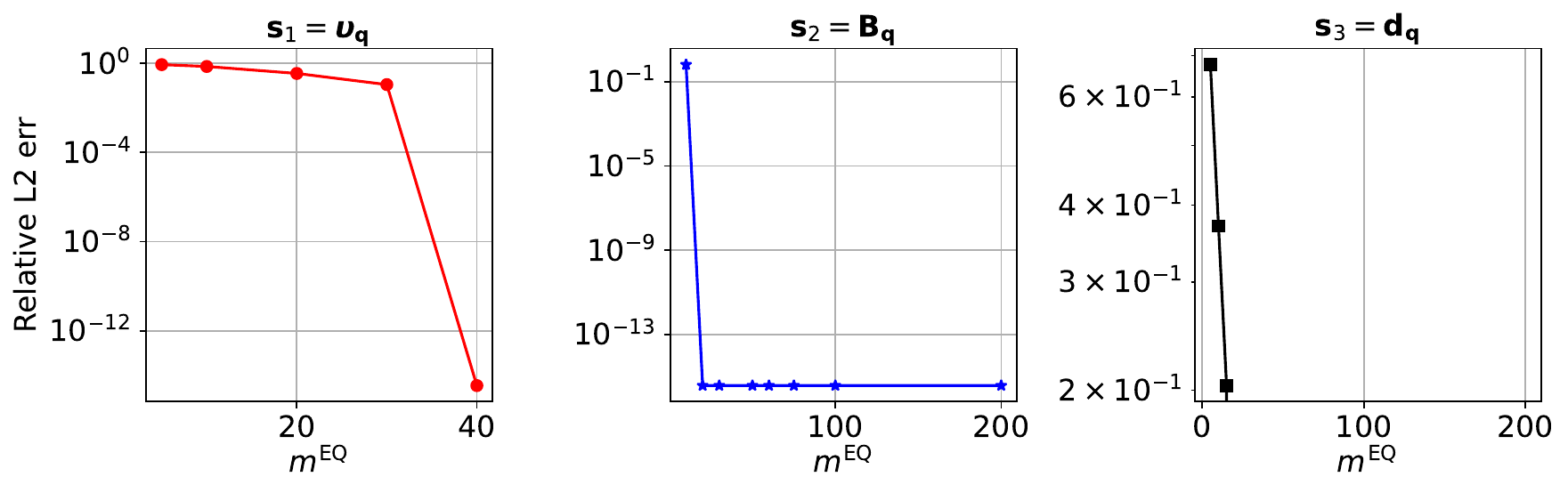}
\caption{Hyper-reduced model . Empirical quadrature procedure. Relative $L^2$ error associated with the three non affine variables $\mathbf{s}_1$, $\mathbf{s}_2$ and $\mathbf{s}_3$.}
\label{fig:EQ_error_decay}
\end{figure}
Figure \ref{fig:EQ_error_decay} reports the corresponding results obtained with the EQ hyper-reduction. As described in section \ref{sec:hyper}, both online computational costs are independent of $\mathcal{N}$ and $\mathcal{R}$, the critical difference is in the number of points required to achieve a given accuracy. Due to the slow singular value decay of the contact snapshot matrix (Figure~\ref{fig:eim_pod}), EIM requires $m_2^{\rm EIM} = \mathcal{O}(\mathcal{R})$ modes to achieve a projection error smaller than $10^{-3}$ on $\mathbf{s}_2$. In contrast, EQ operates directly in the reduced space $\mathbb{R}^{R \times (N+n^{\rm{PGA}})}$ and requires only $m_2^{\rm EQ} = \mathcal{O}(R)$ points (Figure~\ref{fig:EQ_error_decay}). Since $R \ll \mathcal{R}$, the EQ method yields a significantly lower online cost for this contact problem.\\
In Figure \ref{fig:EQ_error_in_time}, we report the $L^2$ relative error over time for several choices of the reduced spaces$(\hat{V}_N, \hat{W}_R^{+})$, where $\hat{V}_N$ is PGA-enriched. The shaded region represents the variability of the error across different validation parameters in $\mathcal{P}_{\rm{valid}}$, while the solid line denotes the mean error. Aside from the zero error at the initial condition, the error exhibits only a mild growth over time. For all configurations, $R$ Empirical Quadrature points are used to approximate $\mathbf{B}_{\mathbf{q}(\sigma)}$ ,  $2R$  to approximate $\mathbf{d}_{\mathbf{q}(\sigma)}$, and $2N$ for the approximation of $\boldsymbol{\upsilon}_{\sigma}$.
\begin{figure}[h]
    \centering
    \includegraphics[width=0.65\textwidth]{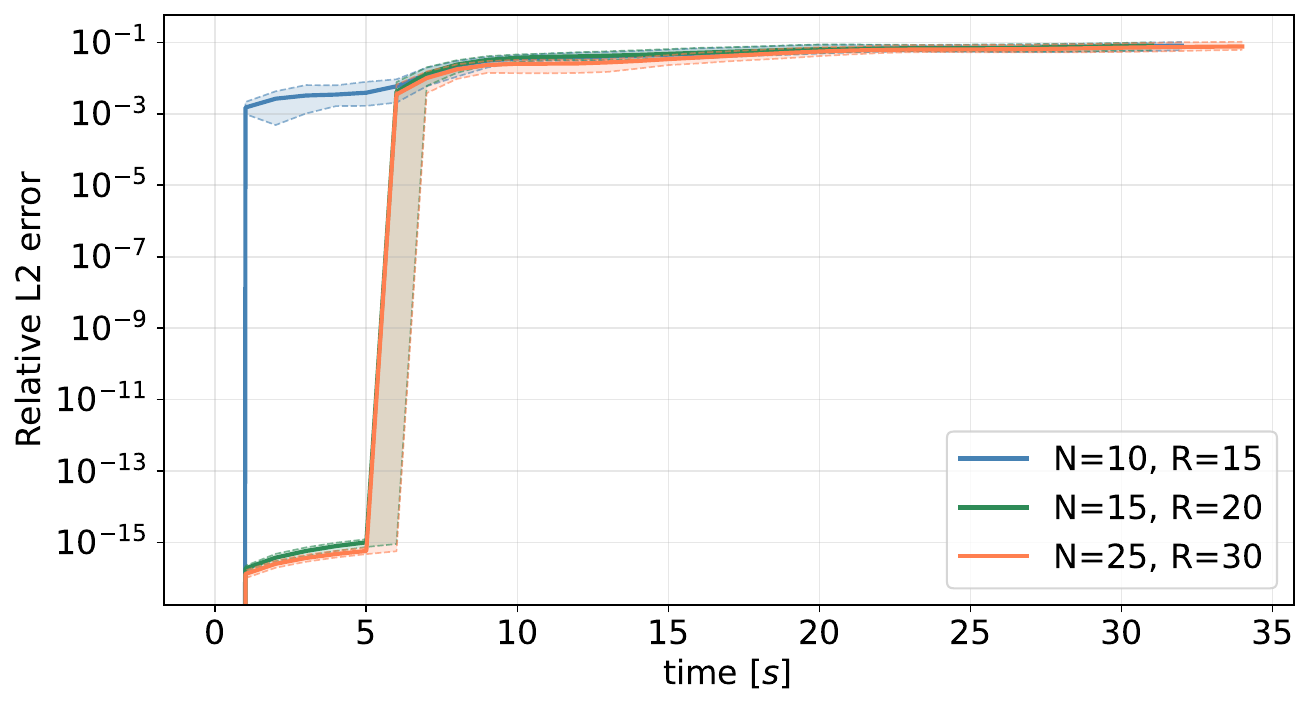}
    \caption{Hyper-reduced model . Empirical Quadrature procedure. $L^2$ error in time for increasing dimensions of $(\hat{\mathcal{V}}_N,\hat{\mathcal{W}}_R^{+})$. The solid lines indicate the average error, the colored shadowed the min-max errors for $\mu \in \mathcal{P}_{\rm{valid}}$. For each RB dimension, we used $m_{B}^{\rm{EQ}}=R$.}
    \label{fig:EQ_error_in_time}
\end{figure}
In Figure \ref{fig:ROM_EQ_speedup} we show the performance in terms of prediction accuracy vs computational gain: we can observe that a speedup in the range $30-60$ is achieved by the hyper-reduced model for different dimensions of the ROM. In Figure \ref{fig:particles_traj} we depict the trajectories of three selected particles in the crowd. The error is observed to vary with the particle index, as well as with proximity to obstacles and neighboring particles, which explains the variability in the results. In particular, particles located farther from the exit tend to be more prone to inaccurate predictions. The particles positions shown in Figure \ref{fig:particles_positions} for a validation parameter $\bar{\mu}$ at three selected time steps offer a qualitative illustration of the predictive accuracy reported in Figures \ref{fig:EQ_error_in_time} and \ref{fig:ROM_EQ_speedup}(a).
\begin{figure}[h]
\centering
\includegraphics[scale=0.32]{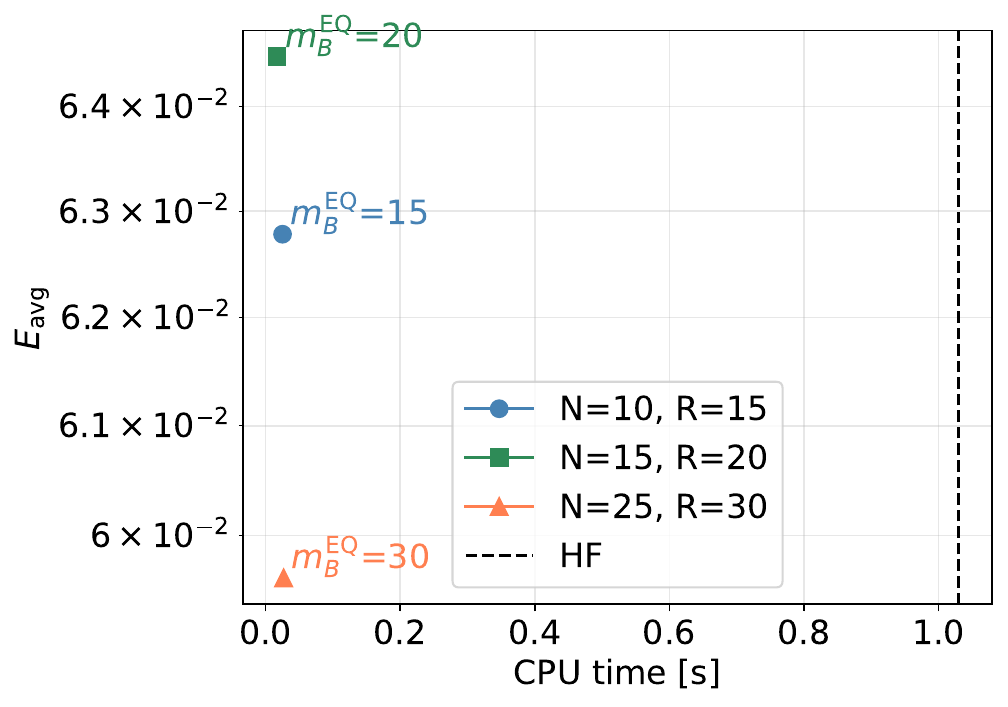}
\includegraphics[scale=0.32]{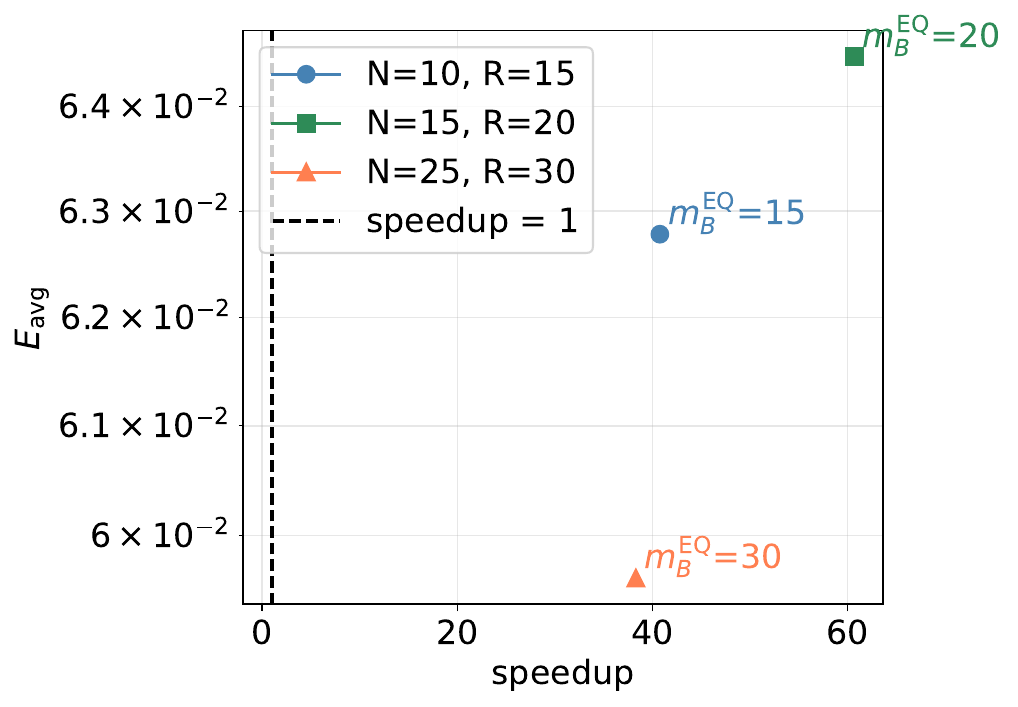}
\caption{Hyper-reduced model . Galerkin ROM performance. (a): Pareto plot showing $E{q}_{\rm{avg}}$ vs the computational cost $[s]$. (b): Pareto plot showing $E_{\rm{avg}}$ vs the computational speedup with respect to the FOM.}
\label{fig:ROM_EQ_speedup}
\end{figure}
\begin{figure}[h!]
    \centering
    \includegraphics[width=0.3\linewidth]{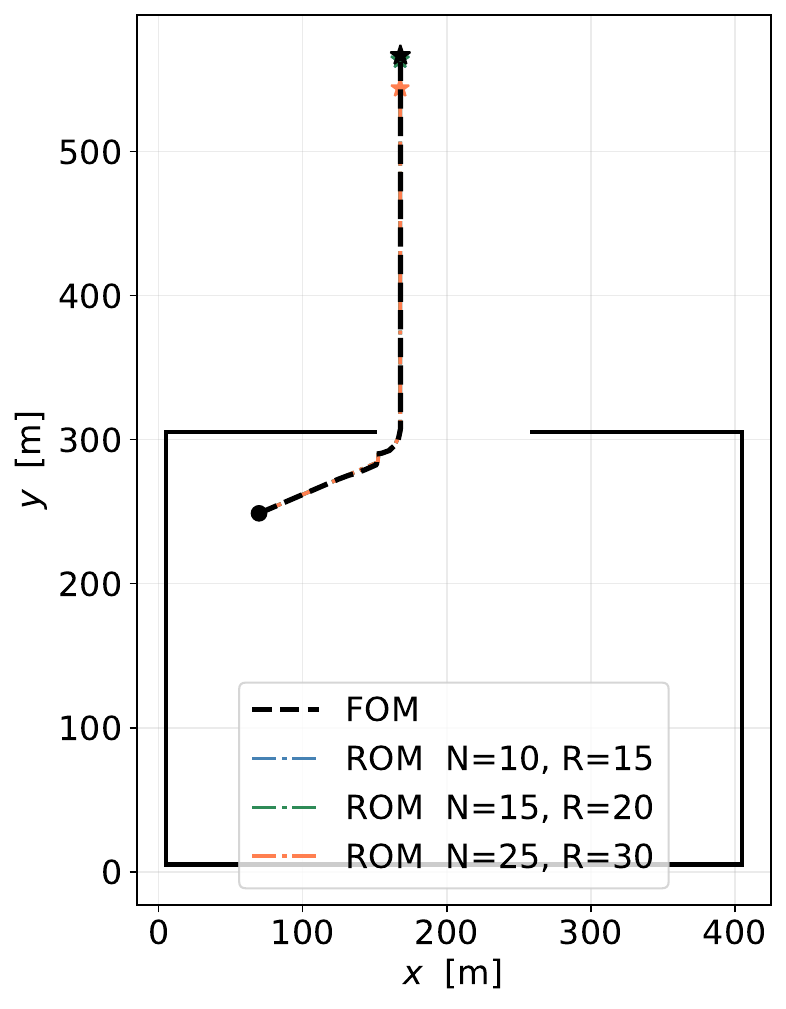}
    \includegraphics[width=0.3\linewidth]{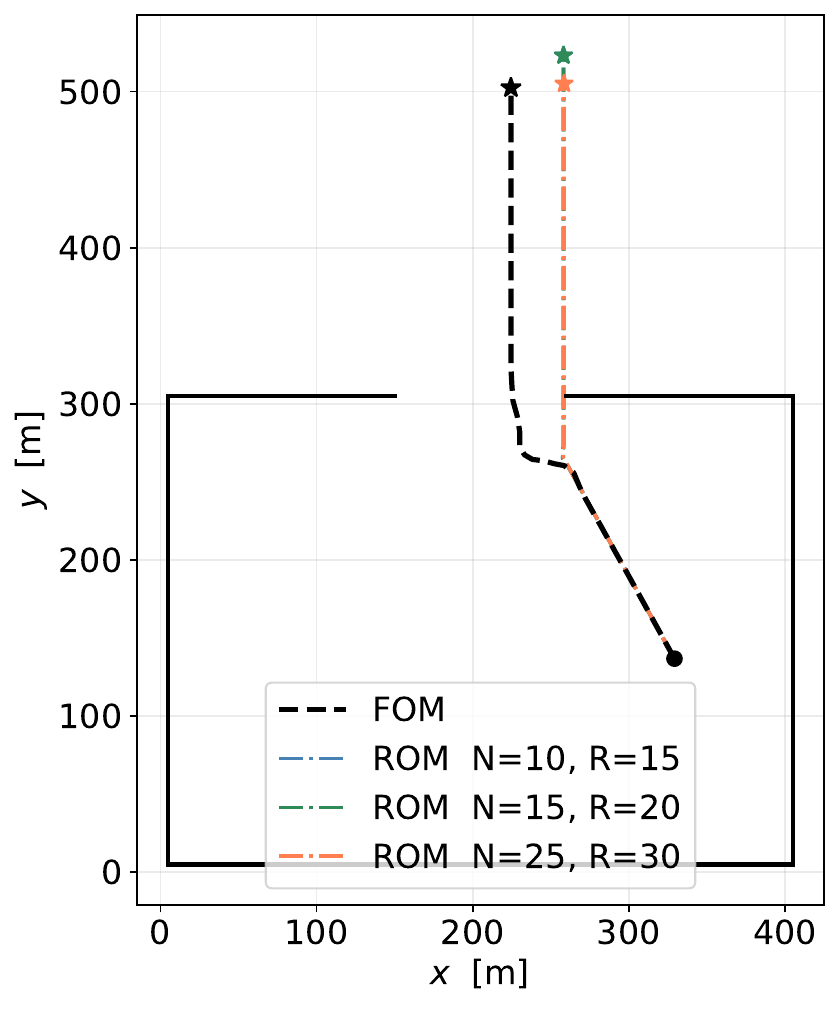}
    \includegraphics[width=0.3\linewidth]{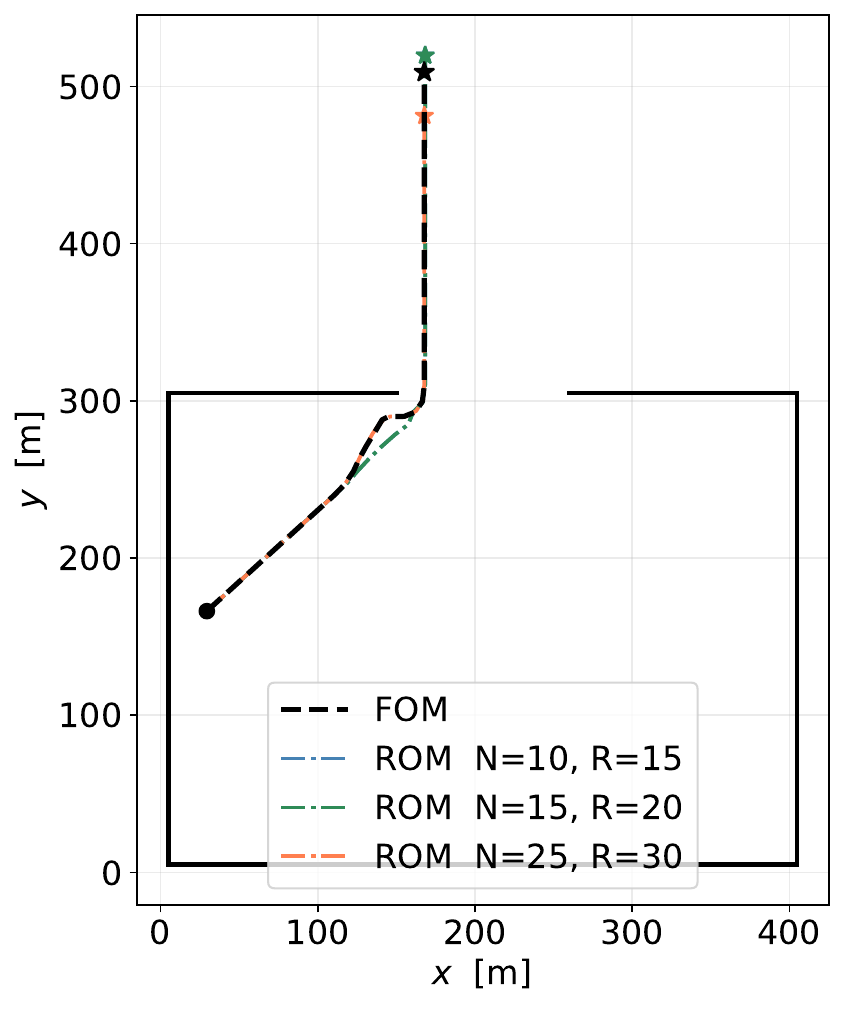}
    \caption{Hyper-reduced model . Sketch of Galerkin EQ-ROM predicted trajectories compared with FOM trajectories for a representative out-of-sample parameter $\bar{\mu} \in \mathcal{P}_{\rm{valid}}$ for three selected particles. For each dimension $(N,R)$, we choose $m_{B}^{EQ}=R$.}
    \label{fig:particles_traj}
\end{figure}

\begin{figure}
    \centering
    \includegraphics[width=0.24\linewidth]{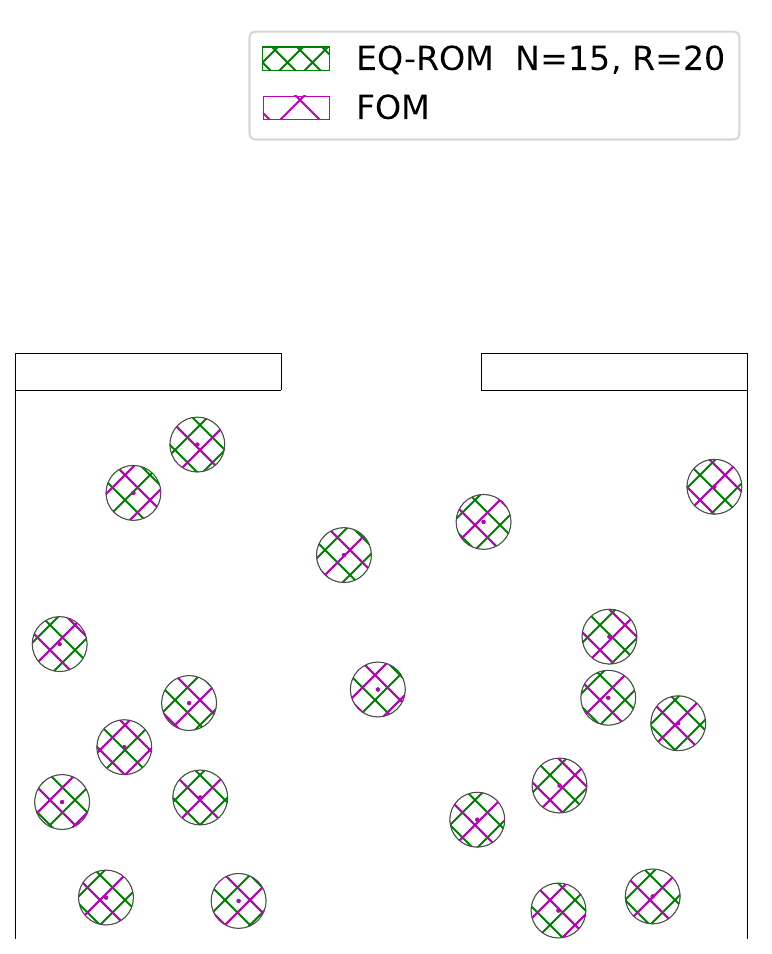}
     \includegraphics[width=0.24\linewidth]{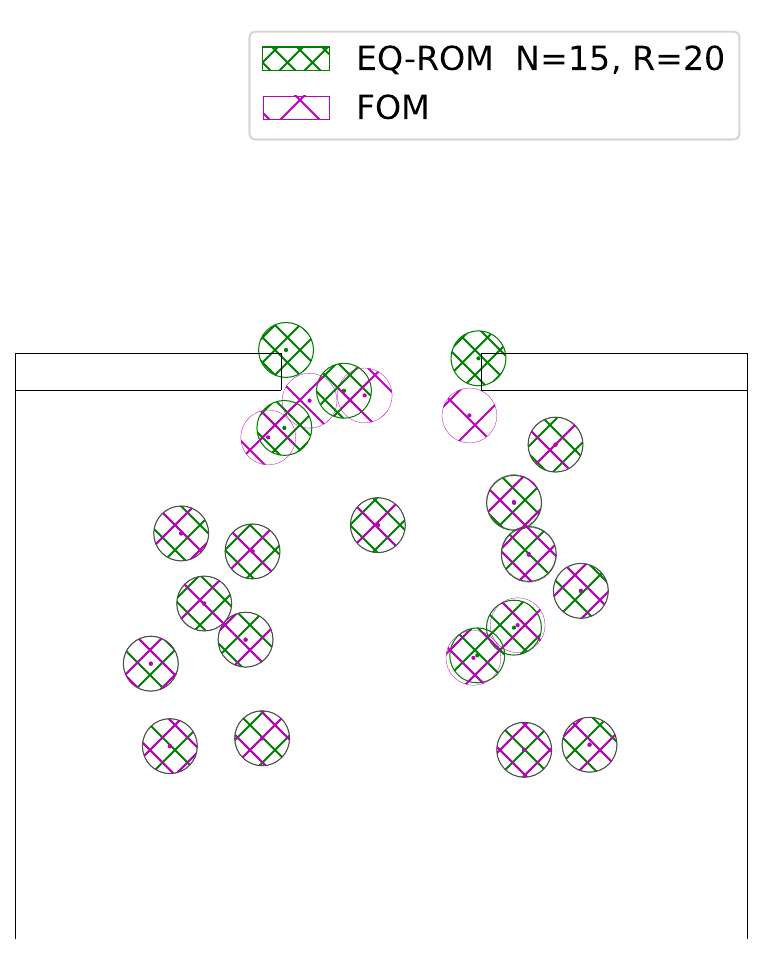}
      \includegraphics[width=0.24\linewidth]{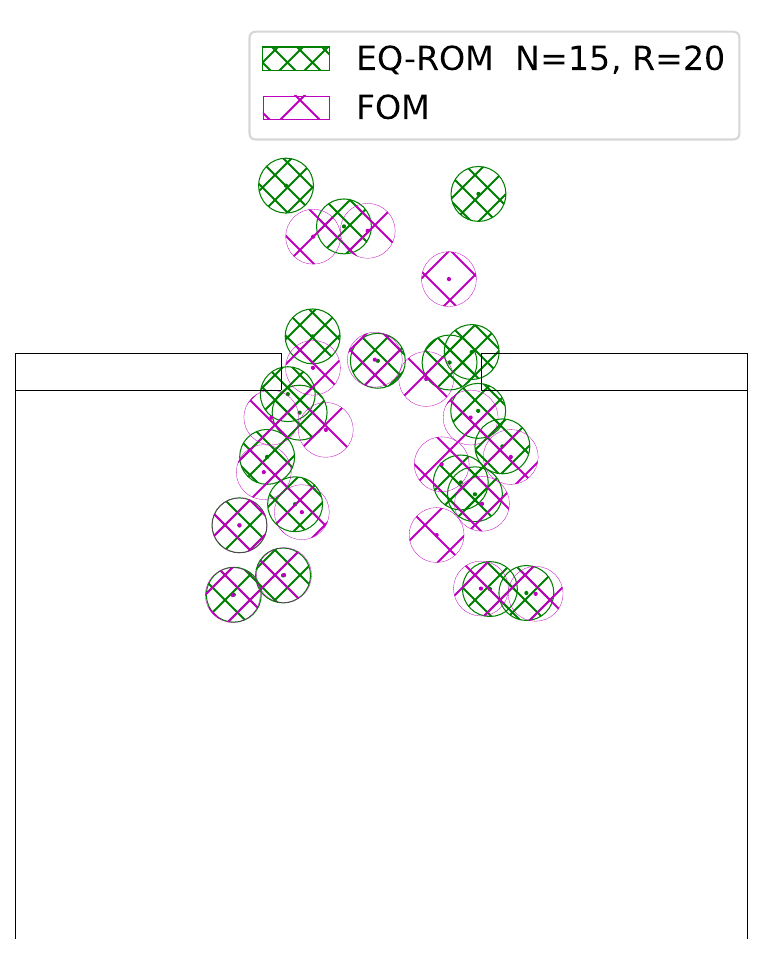}
      
    \caption{Hyper-reduced model . Particle positions at three selected times with indices $\nu \in \{0,9,18\}$ for $N=15$ and $R=20$.}
    \label{fig:particles_positions}
\end{figure}
\paragraph{ML-corrected ROM}
We trained a regressor $\Phi: \mathbb{R}^n \rightarrow \mathbb{R}^{N-n}$ to map the dominant generalized coordinates $\{\alpha_i\}_{i\in \{1:n\}}$—acquired from evaluating the Galerkin ROM conservatively offline at an optimal dimension $N$ (e.g., $N=40$) and subsequently truncating—onto the residual spectral tail $\{\alpha_j\}_{j \in \{n+1:\}}$ extracted strictly from the High-Fidelity Exact snapshot matrices.  To ensure optimal numerical conditioning and convergence during training, the input sequences are universally normalized to zero mean and unit variance using Python standard scaling. The regression is executed by a fully-connected feed-forward Neural Network featuring two hidden layers consisting of $100$ and $50$ neurons, respectively, utilizing ReLU activation functions. The optimization minimizes the Mean Squared Error (MSE) leveraging the Adam stochastic gradient descent algorithm with an initial learning rate optimally tuned to $\eta = 0.01$. The optimizer enforces convergence across a maximum budget of $2000$ epochs. For the EQ-ROM, we set $m^{\rm{EQ}}_{1}=2N$, $m^{\rm{EQ}}_2=R$ and $m^{\rm{EQ}}_{3}=2R$.
\begin{figure}[h!]
    \centering
   \subfloat[accuracy-speedup]{
   \includegraphics[scale=0.35]{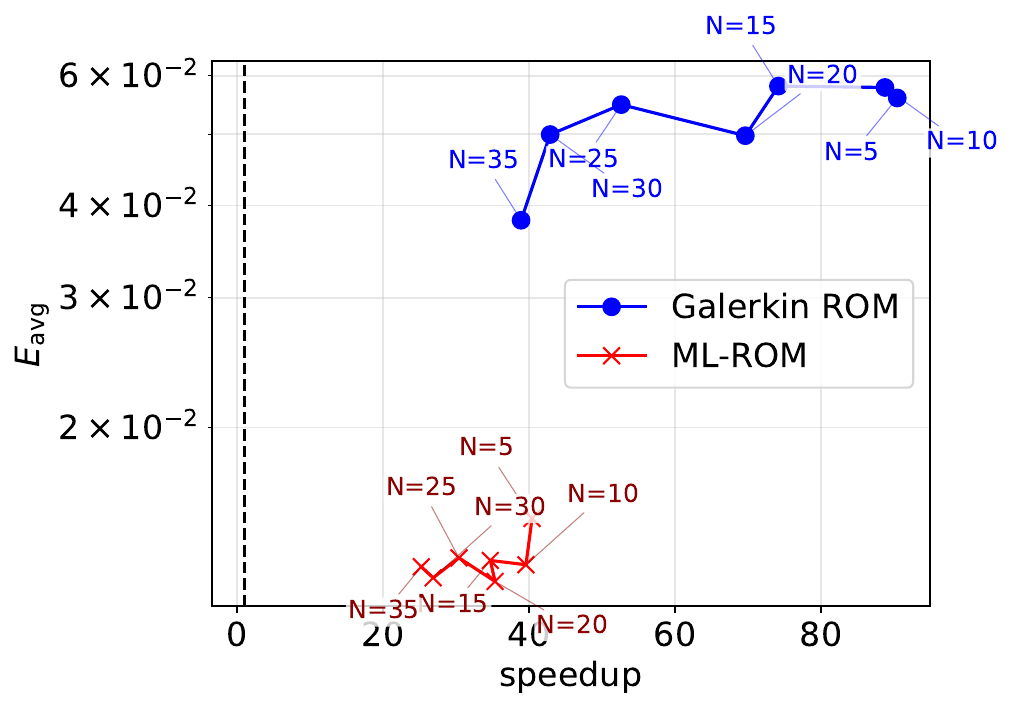}
   }
    \quad
    \subfloat[$\nu=0$]{
    \includegraphics[scale=0.24]{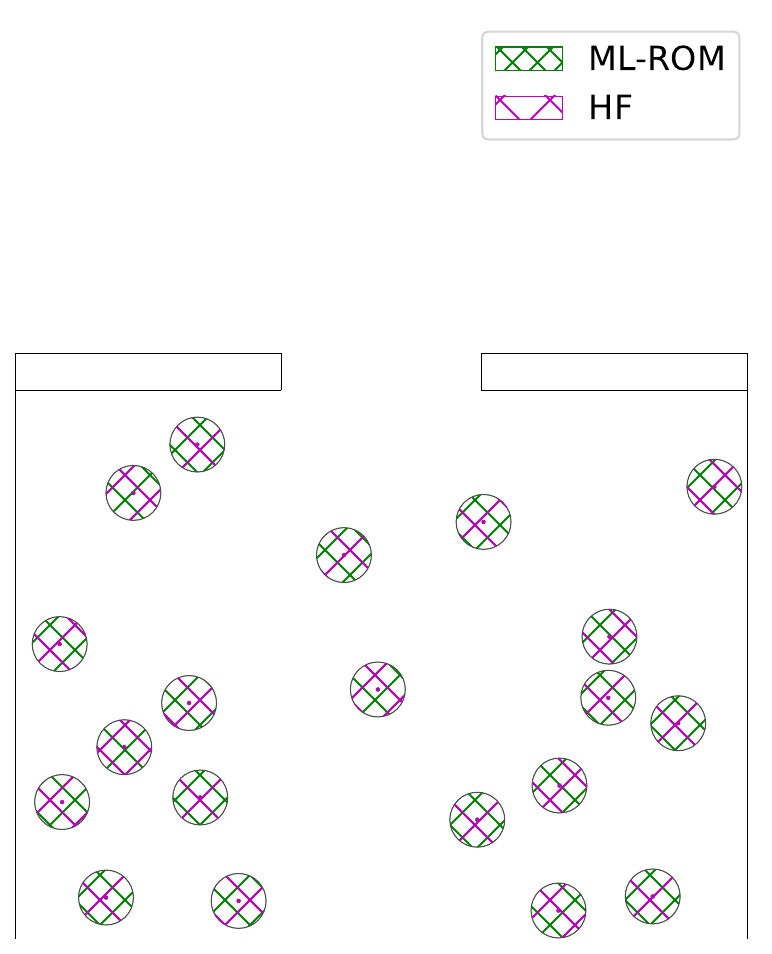}
    }
    \subfloat[$\nu=9$]{
    \includegraphics[scale=0.24]{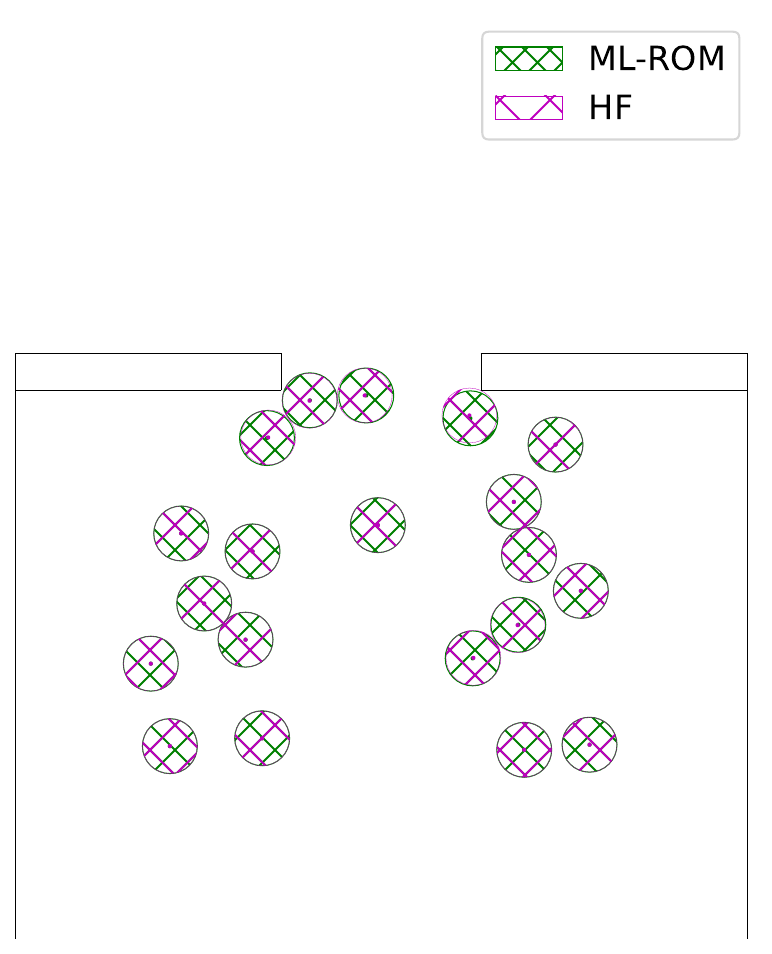}
    }
    \subfloat[$\nu=18$]{
    \includegraphics[scale=0.24]{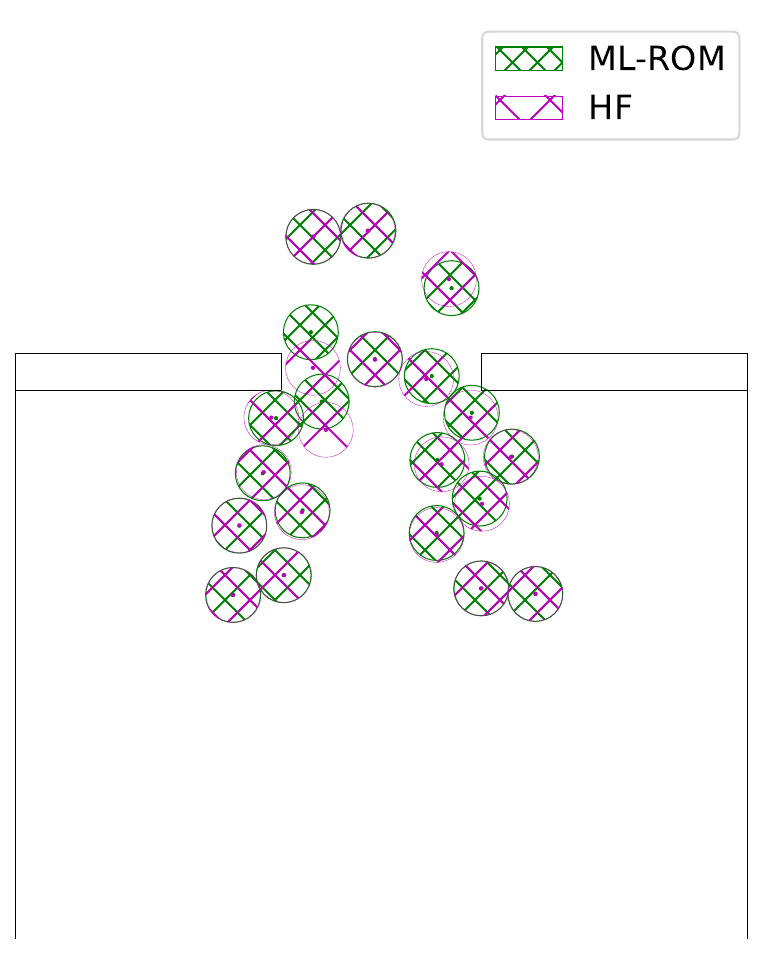}
    }
    \caption{ML corrected ROM. (a): accuracy-speedup comparison: ML-corrected Galerkin EQ ROM error compared with the baseline Galerkin EQ-ROM for several stable paris $(N,R)$ for $\mu \in \mathcal{P}_{\rm{valid}}$.  b),c),d): ML-corrected ROM positions for $n=20$ at times index $\nu \in \{0,9,18\}$.}
    \label{fig:ML_correction}
\end{figure}
We can observe that the ML-corrected ROM can be $3-5$ times more accurate than the baseline ROM, without adding any significant online overhead at the evaluation of the ML map. The ML-corrected ROM average speedup is approximately equal to $30$. The curves are constant in $n$ since the Galerkin ROM is solved for the fixed dimensions of $(\hat{V}_N, \hat{W}_{R}^{+})$ equal to $(40,65)$.
\subsection{A high-dimensional and highly-congested scenario}
\label{sec:scopi_sim}
\begin{wrapfigure}{r}{0.28\textwidth}
\includegraphics[scale=0.12]{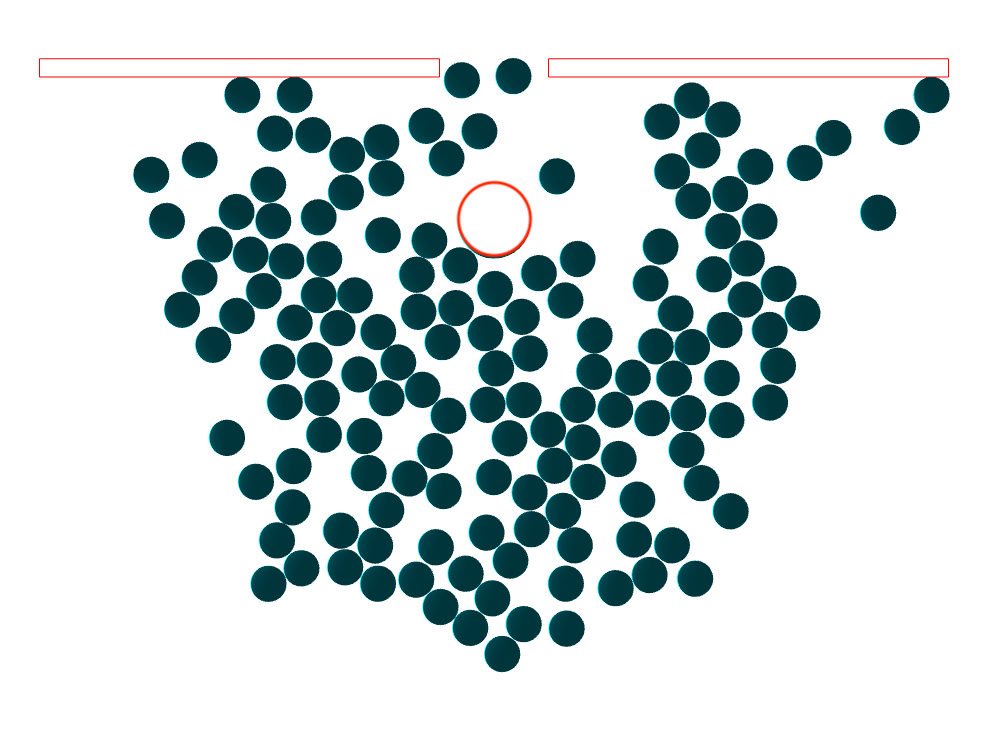}
\end{wrapfigure}
As a further numerical investigation, we show the applicability of the proposed approach to a higher-dimensional multi-particle system setting: we consider a highly congested situation of $N^{
\rm a}=150$ agents with multiple obstacles. 

In this case, the surface covered by the agents is approximately equal to the $50\%$ of the available room (in the previous case, the same ratio is approximately equal to the $12\%$). To make the study case even more challenging, we added a fixed spherical-shaped obstacle in the middle of the region where particles are initially randomly placed. 
We consider $2$ parameters also for this second study case, they represent, as in the previous study case, the exit width $l_{\rm{exit}}$ and the spontaneous velocity magnitude $c_{\boldsymbol{\upsilon}}$: the geometric parametric set is $\mathcal{P}_{\rm{train}}=[0.28,0.36]\times [4.5,5.5]$. To generate the high-fidelity datasets for particle positions, we rely on the optimized C++ code for large-scale multi-particles systems \href{https://lefebvre-lepot.perso.math.cnrs.fr/SCoPI/}{SCoPI} Simulations of Collections of Interacting Particles. We equipped this library with MOR routines. We set equal radius for all the particles $r^{\rm{a}}=0.05$ and the time step as $h=0.005$; Due to the complexity and long time horizon of the crowd dynamics in the region, we considered a fixed value of $N^T=50$ of observation time instances. We set $p_{\rm{train}}=200$ training-parameter samples, uniformly distributed in each parametric direction. At testing phase, $p_{\rm{test}}=50$ testing-parameter were considered instead, different from the training ones. 

\paragraph{Data-compression for high-dimensional data}
The computation of a suitable RB space for the positions for $\mu \in \mathcal{P}_{\rm{train}}$ introduces additional technical difficulties: when dealing with larger $\mathcal{N}$ and/or a high number of training parameters $p_{\rm{train}}$ or time instances $N^T$, the computational time and memory required by POD may become prohibitive. The case  $\textit{s}_{\rm{train}}=p_{\rm{train}}N^T  \gg N^T$ corresponds to all the study case in this work. Applying a classical POD to the full snapshots matrix for this high-dimensional test case would scale with $\mathcal{O}(\mathcal{N} s_{\rm{train}} \min\{\mathcal{N}, s_{\rm{train}}\})$. In order to speed up computations, we resort to randomized POD (rPOD). Randomized projection methods allow one to approximate the dominant POD modes of large matrices at a significantly reduced cost, typically scaling as $\mathcal{O}(\mathcal{N}s_{\rm{train}} c)$ with $c \ll \min\{\mathcal{N}, s_{\rm{train}}\}$ a small constant.
More precisely, instead of computing the singular value decomposition of the snapshot matrix $S^{u}=\{ \mathbf{u}(\nu,\mu) \}_{(\nu,\mu)\in \mathcal{S}_{\rm train}} \in \mathbb{R}^{\mathcal{N}\times s_{\rm train}}$, we use a randomized POD (rPOD) approach to approximate its dominant column space. Following the randomized range finding procedure introduced in \cite{Halko2011}, the matrix $S^{u}$ is first sampled using a Gaussian random matrix $\hat{\Omega} \in \mathbb{R}^{s_{\rm train}\times (n+c)}$, where $c$ is a small oversampling parameter, yielding $Y = S^{u}\hat{\Omega} $.
A \emph{QR} factorization $Y=\hat{Q}\hat{R}$ provides an orthonormal matrix $\hat{Q} \in \mathbb{R}^{\mathcal{N}\times (N+c)}$ that approximates the range of $S^{u}$. The matrix $S^{u}$ is then projected onto the reduced subspace represented by $\hat{Q}$ and a classical singular value decomposition is performed on the much smaller matrix $\hat{Q}^T S^{u}$ to obtain an approximation of the dominant POD modes. This procedure significantly reduces the computational cost while providing an accurate approximation of the POD basis (for further details about rPOD see \cite{Halko2011,Martinsson2020}).\\
In Figure \ref{fig:rpod}(a) we show the CPU time in $[s]$, as a function of $N$, required by the exact POD and rPOD to compute the POD space for velocities: the CPU time required by POD is not affected by $N$, while the one required by rPOD increases with $N$. For $N=100$, using rPOD allows a speedup approximately equal to $500$ with respect to POD. As a sanity check, in Figure \ref{fig:rpod}(b) we depict the eigenvalues decay associated with both POD and rPOD. We can observe that the two curves are coincident at list until $N=300$ which is equal to the number of rows of matrix $S^{u}$, equal to $2N^{\rm{a}}$: this shows that rPOD captures the dominant subspace. We remark that the exact POD is based on the method of snapshots, based on the exact SVD of the matrix ${S^{u}}^T S^{u}$. Figure \ref{fig:rpod} illustrates the convergence history of the gIS algorithm, measured in terms of the maximum and average projection errors of the Lagrange multipliers. The slow decay confirms the limited efficiency in approximating the contact manifold for this more complex study case.\\
\begin{figure}[h!]
	\centering
	\subfloat[CPU time required by POD and rPOD for increasing basis dimensions denoted by $N$.]{
	\includegraphics[scale=0.28]{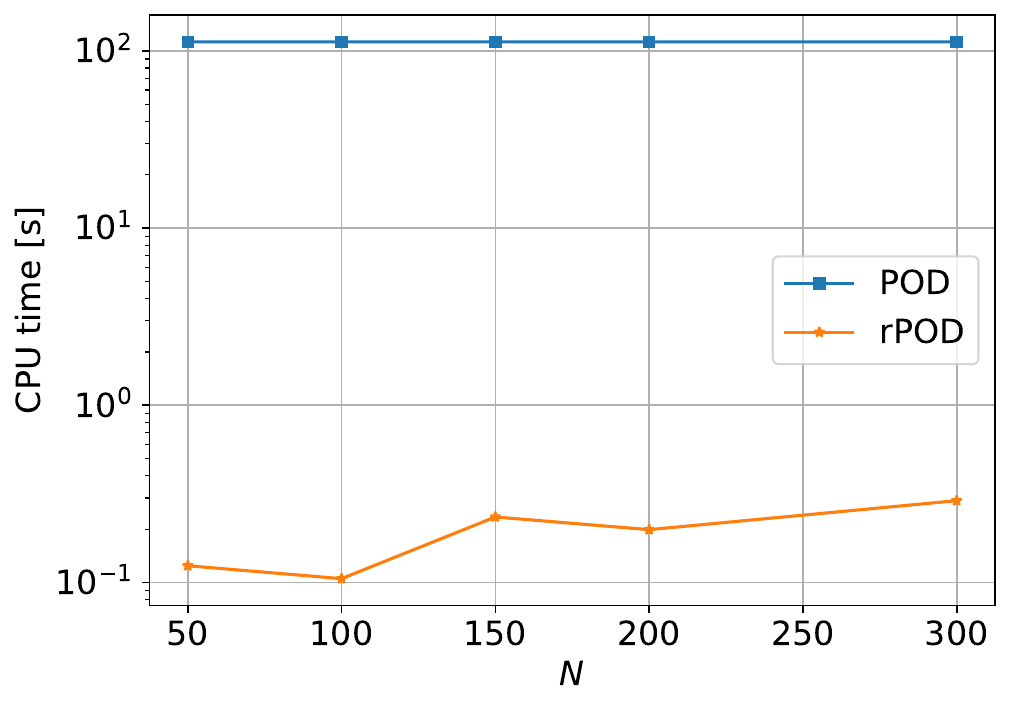}
}
\hspace{0.5cm} 
\subfloat[Eigenvalues decay associated with POD and rPOD for increasing basis dimensions. ]{
	\includegraphics[scale=0.28]{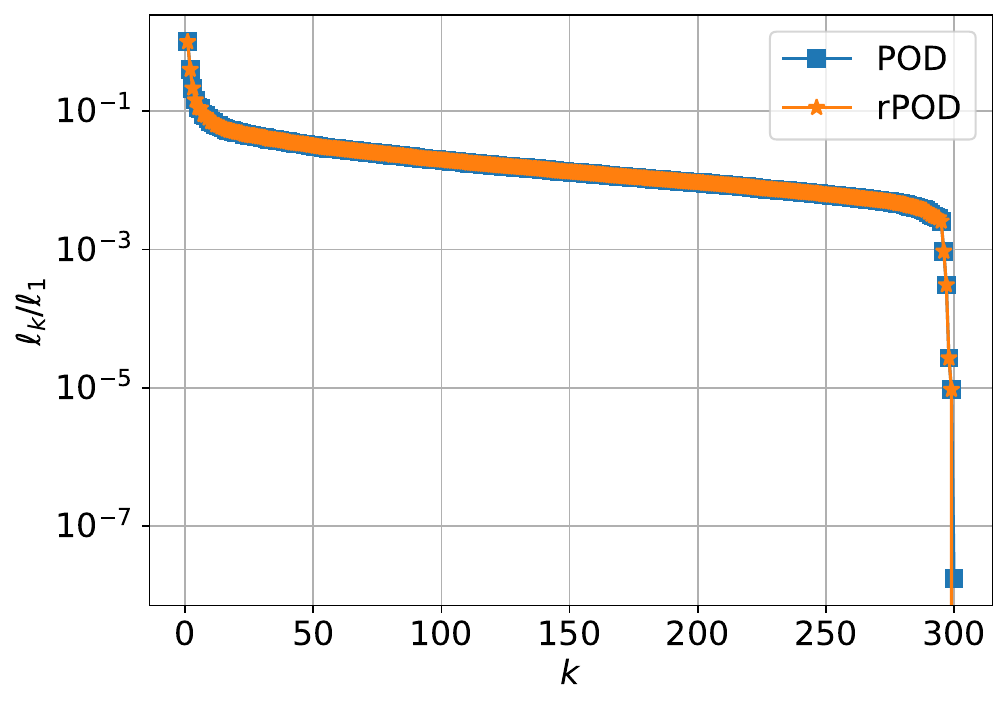}
}
\quad
\subfloat[gIS on Lagrange multipliers]{
\includegraphics[scale=0.28]{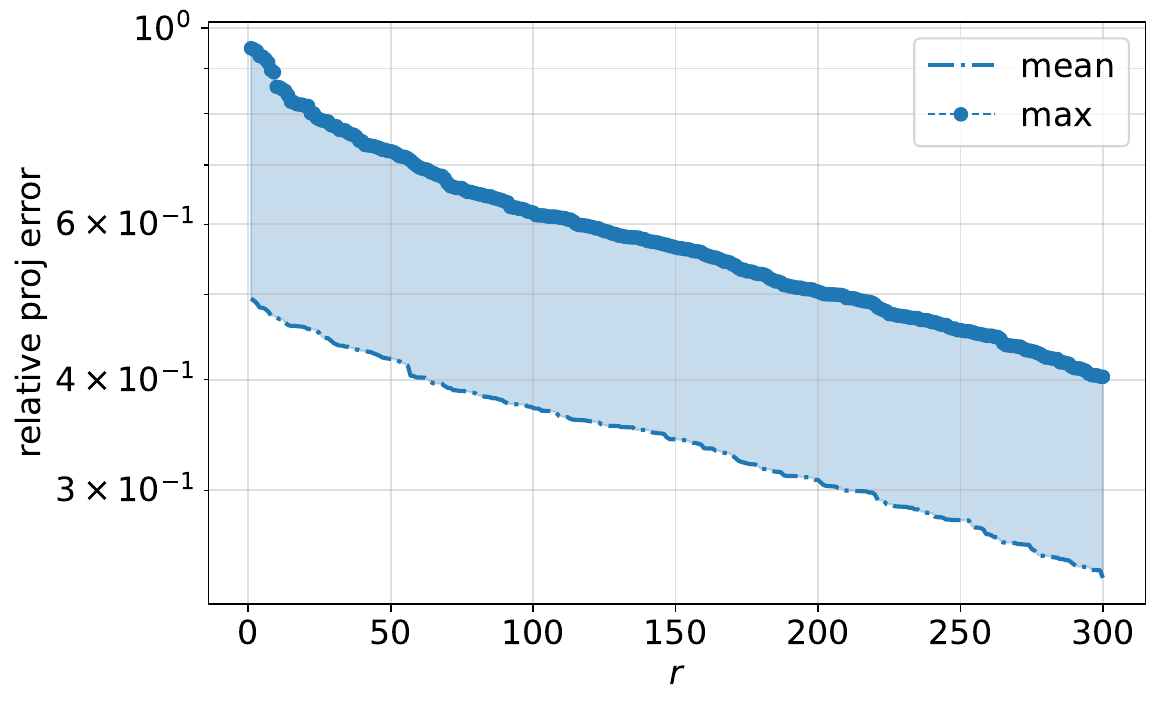}
}
	\caption{A high-dimensional and highly-congested scenario. Randomized POD (rPOD) performance vs exact POD on the velocities snapshots. (a): CPU times; (b): eigenvalues decay. c) gIS on Lagrange multipliers: maximum projection errors vs $r$.}
	\label{fig:rpod}
\end{figure}

\paragraph{Galerkin ROM and ML-corrected ROM}
We remark that while in the previous study case ($N^{\rm a}=20$) the flow is sparser, in this highly-congested and high-dimensional study case, the sheer density of particles physically forces the agents to pack together. The system results in being over-constrained. We employed PGA algorithm to drive the maximum projection residual under the target tolerance $\delta=10^{-4}$. The reduced inf-sup constant results in being $0$ since the multipliers have null space. The ROM in SCoPI utilizes an iterative optimization based solver (in particular an accelerated projected gradient with Nesterov acceleration), finding a valid set of contact forces and allowing the computation of the physical trajectories.

In Figure \ref{fig:scopi_galerkin} (a),(b) we show the ROM and ML-corrected ROM initial particle configurations. The Galerkin ROM is constructed with $(N,R)=(50,75)$ by using the rPOD + PGA and the gIS algorithms. In Figure  \ref{fig:scopi_galerkin} (c),(d) we depict the particles at a selected time step $\nu=16$. We observe that by employing the projection-based ROM, the non inter-penetration constraint is violated for several particle–particle and particle–obstacle interactions in the vicinity of the obstacles. This issue arises from the insufficient number of basis vectors used in the construction of the reduced cone $\hat{W}^{+}$ and of the reduced primal space. On the contrary, the ML-corrected ROM positions in Figure (d) are much better predicted: the accuracy of the particle trajectories is reported in Figure \ref{fig:scopi_galerkin} (a), where the relative error is depicted for the same ROM spaces dimensions. In this case, we are not computing the speedup, as the Galerkin ROM is not equipped with hyper-reduction in SCoPI software. The integration of hyper-reduction techniques into the SCoPI software is deferred to further work. We expect the speedup to increase once the ROM is enhanced with an Empirical Quadrature procedure---as shown for the case $N^{\rm{a}}=20$ described at the beginning of Section \ref{sec:numres_hyper}.
\begin{figure}[h]
    \centering
    \includegraphics[width=0.5\linewidth]{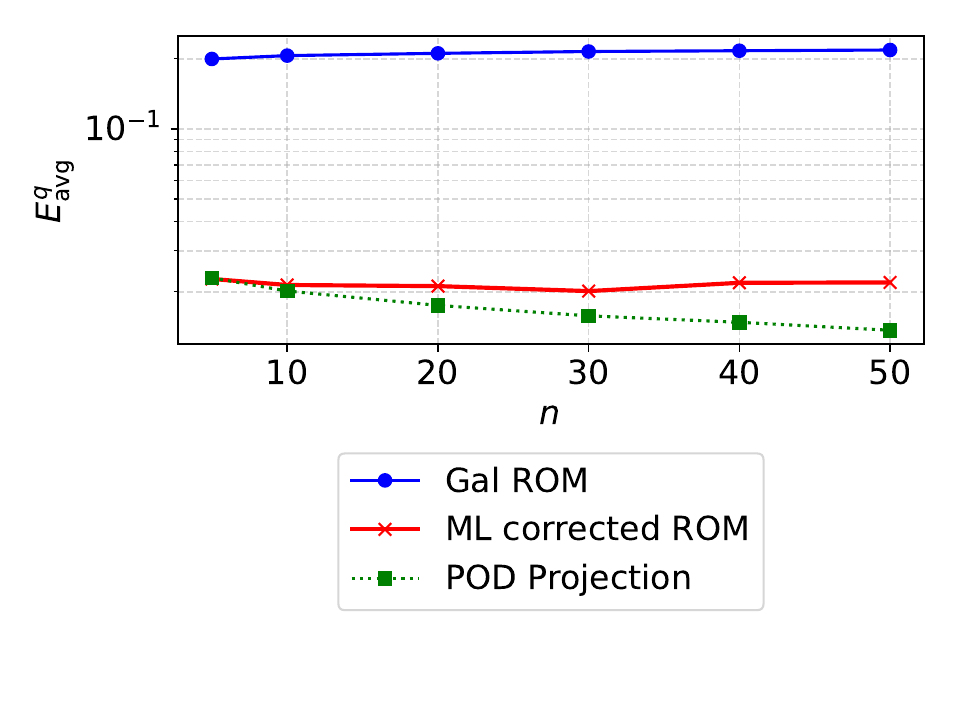}
    \quad 
    \subfloat[$\nu=0$]{
    \includegraphics[scale=0.25]{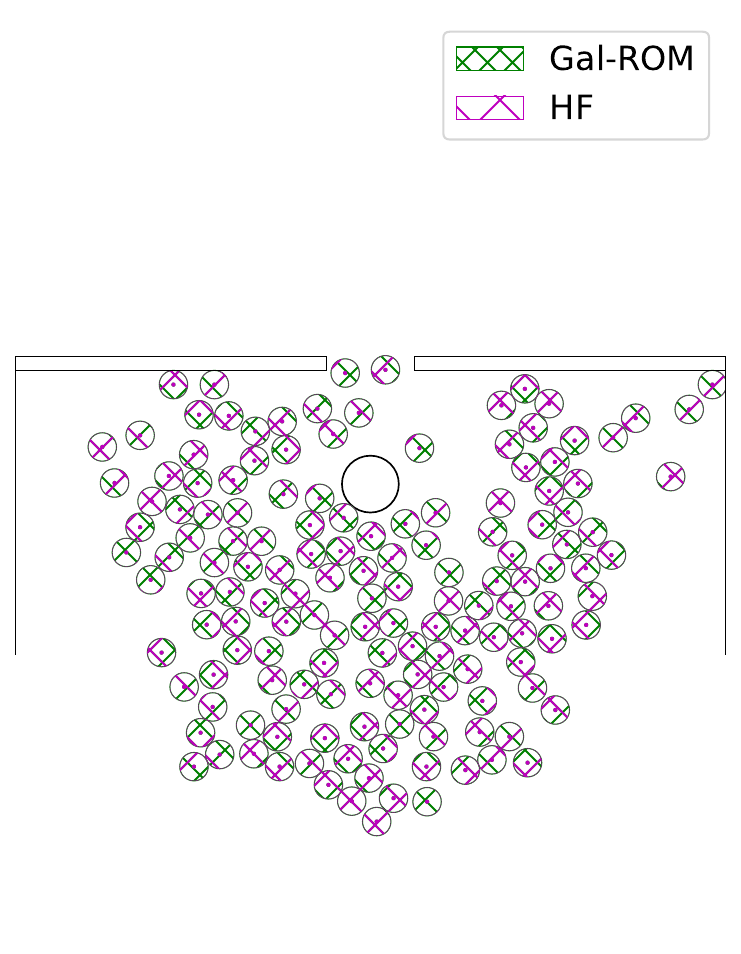}
    }
    \subfloat[$\nu=0$]{
    \includegraphics[scale=0.25]{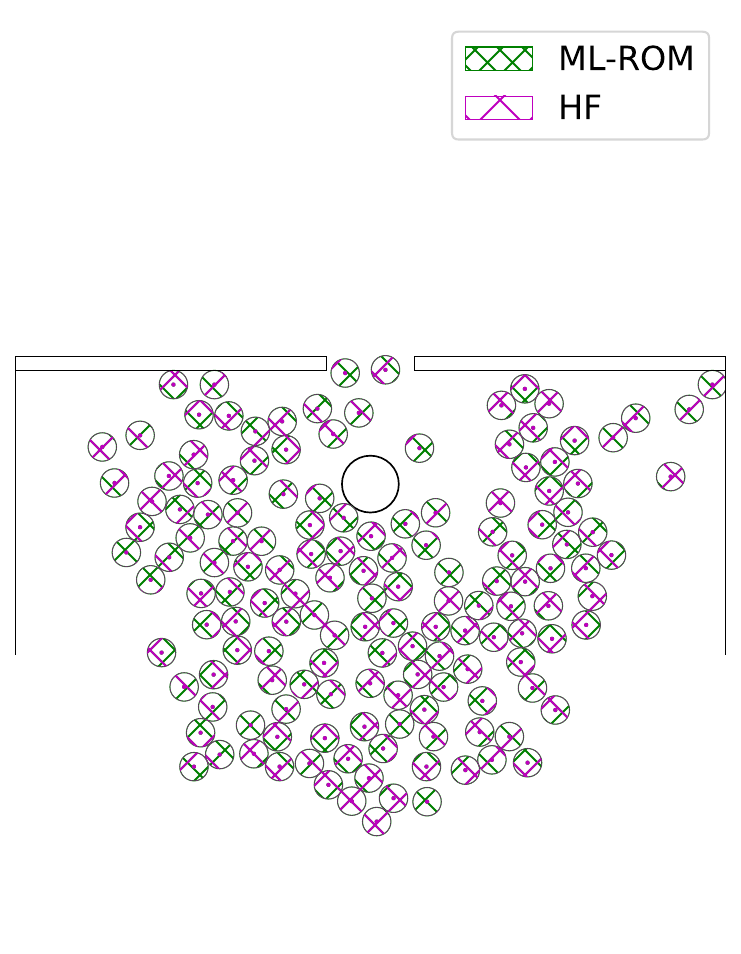}
    }
    \subfloat[$\nu=16$]{
    \includegraphics[scale=0.25]{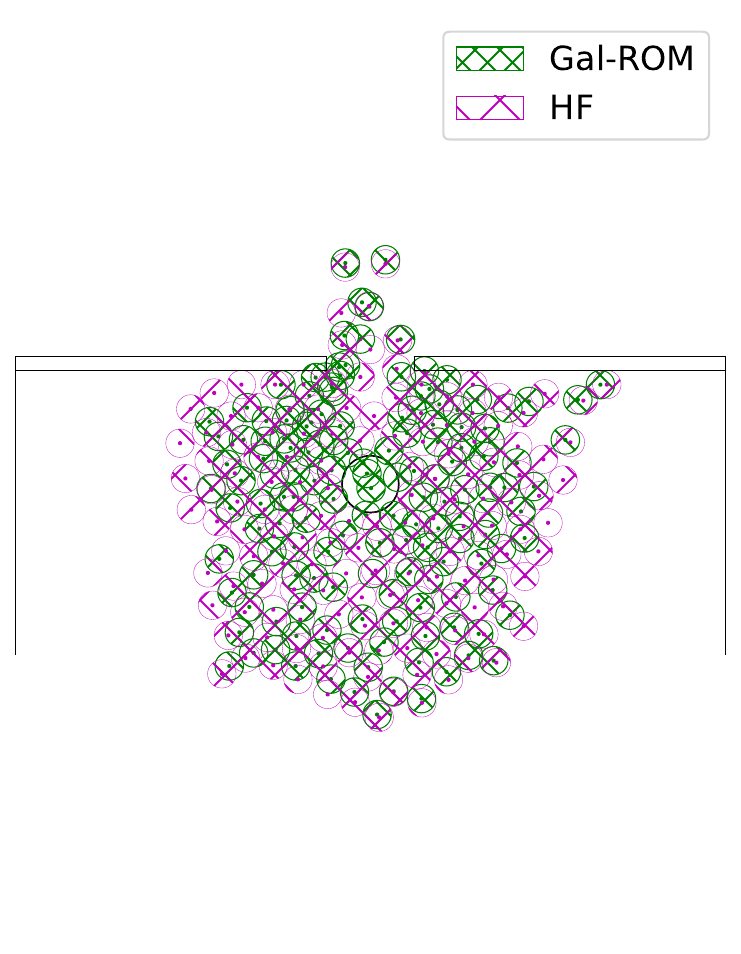}
    }
    \subfloat[$\nu=16$]{
    \includegraphics[scale=0.25]{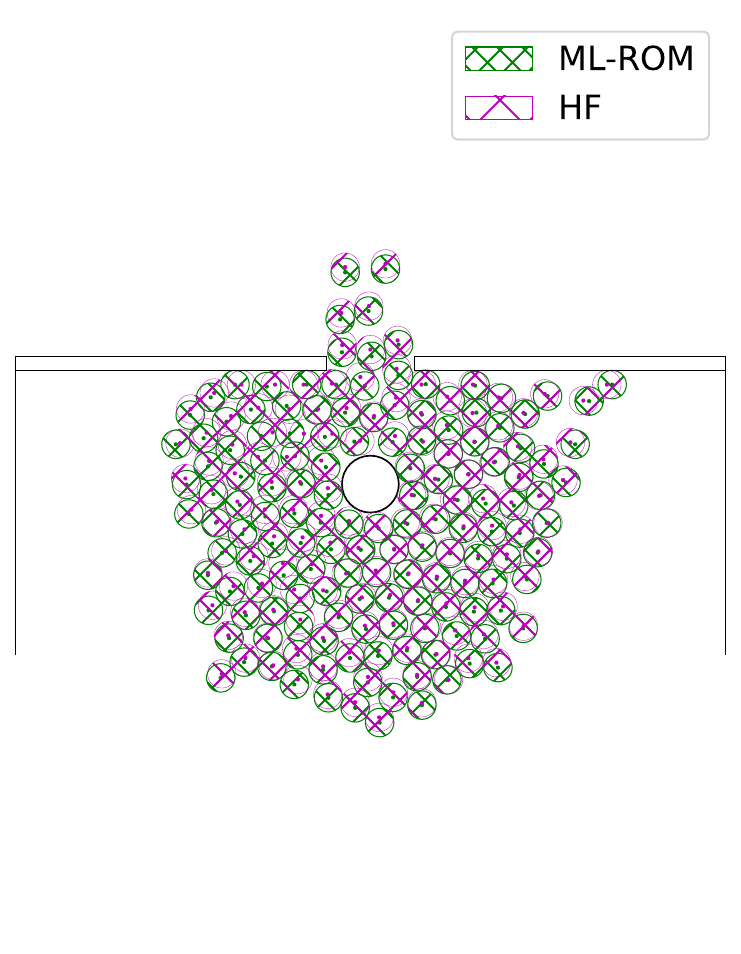}
    }
    \caption{A high-dimensional and highly-congested scenario.  a):  validation error of the ML-corrected ROM, compared with the non-corrected Galerkin ROM error and the solution projection errors. b) and c): Galerkin ROM and ML-corrected ROM particles at a time instant $\nu=0$. d) and e) : same comparisons for $\nu=16$. The ROM has been constructed with $(N,R)=(50,75)$. }
    \label{fig:scopi_galerkin}
\end{figure}

\section{Concluding remarks}
\label{sec:conclusions}
We introduced a nonlinear model reduction approach for parametrized contact problems governed by variational inequalities in a Lagrangian framework. These problems are challenging due to slow Kolmogorov n-width decay and temporal non-smoothness in velocities and contact dynamics. To address this, we developed a projection-based hyper-reduced ROM that preserves non-negativity of Lagrange multipliers and inf-sup stability across training parameters. Numerical experiments demonstrated the method’s effectiveness in highly congested multi-agent scenarios: we also showed that a machine learning correction improves accuracy without sacrificing computational efficiency. This work represents, to our knowledge, a first application of model order reduction to discrete contact problems of this type. Future work includes advanced nonlinear MOR techniques (e.g., optimal transport or morphing), as well as spatial and temporal decomposition strategies to enhance Galerkin ROM performance. Extensions to more complex systems—such as heterogeneous particles, richer geometries, second-order constraint approximations (\cite{bloch2023convex}), and granular material models (see \cite{moreau1996numerical} for frictionless contact dynamics and \cite{verdon2010contact} for the coupling with fluids)—are also of interest, along with parameter estimation and control applications.

\begin{acknowledgements}
Funded/Co-funded by the European Union (ERC, HighLEAP, 101077204).\\
The authors thank Aline Lefebvre-Lepot and Loïc Gouarin for their insightful feedback on the manuscript and for their guidance on the use and implementation of preliminary model order reduction routines within the \href{https://lefebvre-lepot.perso.math.cnrs.fr/SCoPI/}{SCoPI} library.
\end{acknowledgements}

%
%
\appendix
\section{A theoretical study on greedy algorithm for Hertz problem}
\label{sec:hertz_appendix}
\subsection{Greedy algorithm}
We consider the \textit{Hertz} benchmark test of a frictionless contact mechanics problem involving a sphere of radius $R$ and a half-plane of an elastic material, depicted in Figure \ref{fig:hertz_contact}. We consider the normal stress at any point of a sphere touching the half-plane when a load is applied on a certain point, with a maximum
contact pressure $p$. The deformation of the half-plane is denoted as $d$.
We indicate as $a \in [\underline{a}, \overline{a}]$ the radius of the contact surface. We consider the following family of functions in $\mathcal{W}=\mathbb{R}$, representing the distribution of normal pressure in the contact area as a function of distance from the center of the sphere ($r=0$) and parametrized by $a$:
\begin{equation*}
\lambda_a(r)=p\sqrt{1 -\frac{r^2}{a^2}} \mathbb{I}_{\{|r| \leq a\}}(r)
    \label{eq:pressure_profile}
\end{equation*}
\begin{figure}[h]
\centering
\begin{tikzpicture}[scale=0.62]

\draw[dashed, thick, gray]
    (-3.2, 0) -- (3.2, 0);

\draw[thick, black]
    (-3.2, 0)
    .. controls (-1.5, 0) and (-0.9, 0) ..
    (-0.7, -0.08)
    .. controls (-0.5, -0.18) and (-0.3, -0.28) ..
    (0,    -0.32)
    .. controls ( 0.3, -0.28) and ( 0.5, -0.18) ..
    ( 0.7, -0.08)
    .. controls ( 0.9,  0.0) and ( 1.5,  0.0) ..
    ( 3.2,  0);

\foreach \x in {-3.0,-2.7,...,3.0}{
    \draw[gray!60, thin] (\x, -0.55) -- ++(0.25,-0.25);
}
\draw[thick] (-3.2,-0.55) -- (3.2,-0.55);

\shade[ball color=gray!30, opacity=0.9]
    (0, 2.18) circle (2.2cm);
\draw[thick] (0, 2.18) circle (2.2cm);

\draw[->, very thick, black]
    (0, 2.18) -- (0, 0.5)
    node[midway, right=2pt, font=\large]{$F$};

\draw[<->, thick]
    (0, -0.13) -- (0.72, -0.08)
    node[midway, below=0.5pt, font=\large]{$a$};

\draw[thick] (0,   -0.18) -- (0,   -0.08);
\draw[thick] (0.72,-0.13) -- (0.72,-0.03);

\draw[<->, thick]
    (2.8, 0) -- (2.8, -0.55)
    node[midway, right=2pt, font=\large]{$d$};

\draw[dashed, gray]  (0,   0)    -- (2.85,  0);
\draw[dashed, gray]  (0,  -0.55) -- (2.85, -0.55);

\node[font=\normalsize] at (-3.0, 0.2) {$-R_0$};
\node[font=\normalsize] at ( 3.0, 0.2) {$R_0$};

\draw[thick, gray!70]
    (0, 2.18) -- (1.56, 3.74)
    node[midway, above left=1pt, font=\large, text=gray!80]{$R$};

\node[font=\normalsize, below=2pt] at (0, -0.35) {$r_0=0$};

\node[font=\normalsize] at (0.36, 0.18) {$\lambda(r)$};

\draw[thick, purple!70!black, -]
    (-0.72, 0)
    .. controls (-0.5, 0.45) and (-0.2, 0.65) ..
    (0,    0.72)
    .. controls ( 0.2, 0.65) and ( 0.5, 0.45) ..
    ( 0.72, 0);
\draw[dashed, purple!60]
    (-0.72, 0) -- (0.72, 0);

\end{tikzpicture}
\caption{Hertz contact between a sphere of radius $R$ and an elastic
half-plane. A normal load $F$ is applied, producing a contact patch
of radius $a$, an indentation depth $d$, and a pressure distribution
$\lambda(r)$ on $[-a,a]$.
The domain of the half-plane is $[-R_0, R_0]$.}
\label{fig:hertz_contact}
\end{figure}
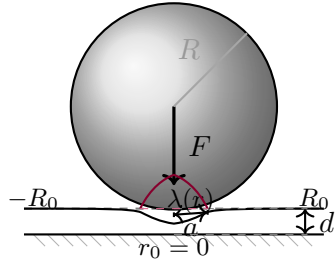
We describe the first steps in the construction of the positive cone $\mathcal{\hat{W}}^{+}=\text{span}^{+}\{\chi_1, \ldots, \chi_r\}$ using a standard greedy procedure. 
We denote by $\|\cdot \|_{\mathcal{W}}$ the $L^2(\mathbb{R})$ norm and by $<\cdot, \cdot>_{\mathcal{W}}$ the $L^2(\mathbb{R})$ inner product.
We consider a one-dimensional parameter set, with $\sigma=a$ (since the problem is steady, we can identify $\sigma$ with $\mu$). We further assume that $r_0$ is fixed. The parameter vector is $\sigma=(a,p)\in\mathcal{S}=[\underline{a},\overline{a}]\times[\underline{p},\overline{p}]$. The snapshot manifold becomes \[\mathcal{M}^{\lambda}= \Bigl\{\lambda_{a,p}(r)  = p\sqrt{1-r^2/a^2}\,\mathbb{I}_{\{|r|\leq a\}},\quad (a,p)\in\mathcal{S}\Bigr\}.\]
\begin{subequations}
The pressure profile factorizes as
\begin{equation}
\lambda_{a,p}(r) = 
p\cdot\varphi_{a}(r),
\varphi_{a}(r) := \sqrt{1-r^2/a^2}\,\mathbb{I}_{\{|r|\leq a\} }\label{eq:separability}
\end{equation} 
\paragraph{Step $1$.}
We have
\[
  \|\lambda_{a,p}\|^{2}_{\mathcal{W}}
  = p^{2}\int_{-a}^{a}\!\Bigl(1-\frac{r^{2}}{a^{2}}\Bigr)\,dr
  = \frac{4ap^{2}}{3},
\]
which is strictly increasing in both $a$ and $p^{2}$. Hence the maximum over
$\mathcal{S}$ is attained at $(\overline{a},\overline{p})$, and the
greedy initialization yields
\[
  \sigma_{1}^{\star} = (\overline{a},\,\overline{p}),
  \qquad
  \psi_{1} = \lambda_{\overline{a},\overline{p}}
           = \overline{p}\sqrt{1-r^{2}/\overline{a}^{2}}.
\]

\paragraph{Step $2$.}
We seek the parameter that is least well approximated by $\psi_{1}$:
\[
  \sigma_{2}^{\star}
  \;\in\;
  \underset{(a,p)\in\mathcal{S}}{\arg\max}\;
  \|\lambda_{a,p} - \texttt{P}_{\psi_{1}}\lambda_{a,p}\|^{2}_{\mathcal{W}}
  \;=\;
  \underset{(a,p)\in\mathcal{S}}{\arg\max}\;
  \min_{k\geq 0}\|\lambda_{a,p}-k\psi_{1}\|^{2}_{\mathcal{W}}.
\]
Expanding the squared norm gives three integral components:
\begin{itemize}
\item[$\circ$] $I_{1}(a,p)
        =\|\lambda_{a,p}\|^{2}_{\mathcal{W}}
        =\dfrac{4ap^{2}}{3}$;
\item[$\circ$]$I_{2}
        =\|\psi_{1}\|^{2}_{\mathcal{W}}
        =\dfrac{4\overline{a}\,\overline{p}^{2}}{3}$\quad (from Step~1);
\item[$\circ$]$I_{3}(a,p)
        =\langle\lambda_{a,p},\psi_{1}\rangle_{\mathcal{W}}
        = p\,\overline{p}
          \displaystyle\int_{-a}^{a}
          \sqrt{1-\frac{r^{2}}{a^{2}}}\,
          \sqrt{1-\frac{r^{2}}{\overline{a}^{2}}}\,dr
        =: p\,\overline{p}\,J(a)$,
\end{itemize}
where, using the substitution $r=a\sin\theta$,
\begin{equation}
  J(a)
  = a\int_{-\pi/2}^{\pi/2}
    \cos^{2}\!\theta\,
    \sqrt{1-\frac{a^{2}}{\overline{a}^{2}}\sin^{2}\!\theta}
    \;d\theta
  \label{eq:elliptic}
\end{equation}
is a complete elliptic integral that depends only on the shape parameter $a$.
Since both $\lambda_{a,p}$ and $\psi_{1}$ are non-negative, $I_{3}>0$ for all
$(a,p)\in\mathcal{S}$, and the optimal positive projection coefficient is
\begin{equation}
  k_{1}^{\star}(a,p)
  = \frac{I_{3}(a,p)}{I_{2}}
  = \frac{3p\,J(a)}{4\,\overline{a}\,\overline{p}}
  \;\geq\;0.
  \label{eq:optimal_k_2d}
\end{equation}
The minimum residual is therefore
\begin{equation}
  \mathcal{F}(a,p)
  = I_{1}(a,p) - \frac{I_{3}(a,p)^{2}}{I_{2}}
  = \frac{4ap^{2}}{3} - \frac{3\,p^{2}\overline{p}^{2}J(a)^{2}}{4\,\overline{a}\,\overline{p}^{2}}
  = p^{2}\,G(a),
  \label{eq:residual_2d}
\end{equation}
where we define
\begin{equation}
  G(a) := \frac{4a}{3} - \frac{3\,J(a)^{2}}{4\,\overline{a}}.
  \label{eq:G}
\end{equation}

\begin{proposition}
  The second selected parameter is $\sigma_{2}^{\star}=(\underline{a},\overline{p})$,
  and the greedy selected function is
  \[
    \psi_{2} = \lambda_{\underline{a},\overline{p}}
             = \overline{p}\sqrt{1-r^{2}/\underline{a}^{2}}.
  \]
\end{proposition}

\begin{proof}
The factored form $\mathcal{F}(a,p)=p^{2}G(a)$ in~\eqref{eq:residual_2d}
decouples the maximization over the two parameters:
\begin{equation}
  \sigma_{2}^{\star}
  = \underset{(a,p)\in\mathcal{S}}{\arg\max}\;p^{2}G(a)
  = \Bigl(\underset{a\in[\underline{a},\overline{a}]}{\arg\max}\;G(a),\;
          \underset{p\in[\underline{p},\overline{p}]}{\arg\max}\;p^{2}\Bigr).
  \label{eq:decoupled}
\end{equation}

\emph{Maximization in $p$.}
Since $p\mapsto p^{2}$ is strictly increasing on $[\underline{p},\overline{p}]$,
the maximum is attained at $p^{\star}=\overline{p}$.

\emph{Maximization in $a$.}
$G$ is
strictly decreasing on $[\underline{a},\overline{a}]$: indeed,
$G(\overline{a})=0$ (since $k_{1}^{\star}(\overline{a},\cdot)=1$ recovers
$\psi_{1}$ exactly), while for $a<\overline{a}$ the profiles $\varphi_{a}$ and
$\varphi_{\overline{a}}$ are not proportional, so the Cauchy--Schwarz inequality
gives $J(a)<\sqrt{I_{1}(a,\cdot)/p^{2}\cdot I_{2}/\overline{p}^{2}}$ strictly,
and $G(a)>0$. Moreover, $J(a)$ is strictly increasing in $a$ (as seen from the
elliptic-integral representation~\eqref{eq:elliptic}) and its derivative grows
faster than $\sqrt{a}$, so $G'(a)<0$ for all $a\in[\underline{a},\overline{a})$.
Therefore $G$ attains its maximum at $a^{\star}=\underline{a}$.

Combining both optimizations via~\eqref{eq:decoupled} yields
\[
  \sigma_{2}^{\star} = (\underline{a},\,\overline{p}),
\]
and the corresponding basis function is
\begin{equation}
  \psi_{2}
  = \lambda_{\sigma_{2}^{\star}}
  = \overline{p}\sqrt{1-r^{2}/\underline{a}^{2}}.
\end{equation}
\qed
\end{proof}
\end{subequations}
\begin{remark}
We note that these first two steps results, up to a normalization of the functions, can also be obtained by applying the modified cone-projected greedy method described in Algorithm\ref{alg:mCPG}: indeed, at Step $2$, one is required to solve the following constrained minimization problem:
\[
k^{\star}=\arg \min_{\substack{
    k \in \mathbb{R}_+ \\
    \lambda_{\sigma_2}-k \psi_1 \geq 0
}} \|\lambda_{\sigma_2}-k \psi_1 \|_{\mathcal{W}}^2 
\]
which gives $k^{\star}=0$ and thus $\psi_2=\lambda_{\sigma_2}/\|\lambda_{\sigma_2}\|_{\mathcal{W}}$.
\end{remark}
In Figure \ref{fig:hertz_cone}, we depict in red the selected functions $\psi_1$ and $\psi_2$: they lie entirely outside the blue cone for several intermediate values of $a$. No non-negative combination $k_1 \psi_1 + k_2 \psi_2$ can reproduce a smooth intermediate arch, so two basis functions are insufficient to construct $\hat{\mathcal{W}}^+$. Consequently, the greedy algorithm must continue selecting $\sigma_3$ at an interior $a \in [\underline{a}, \overline{a}]$ to capture the intermediate arch shapes. 

\begin{figure}[h]
\centering
\includegraphics[scale=0.48]{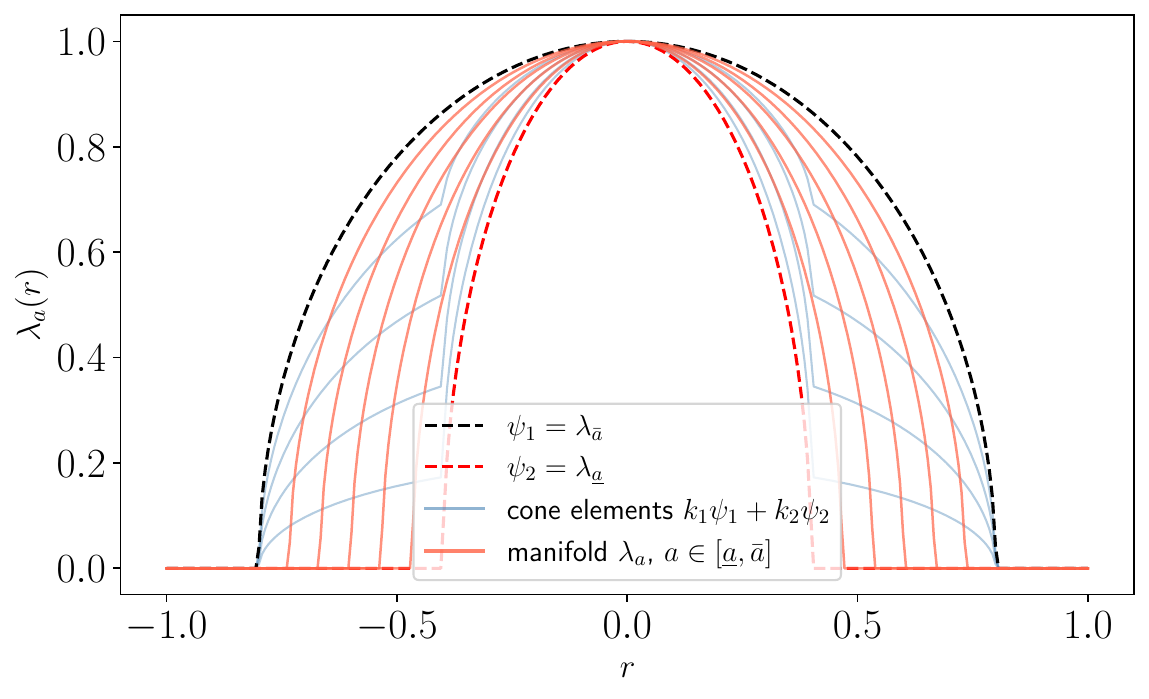}
\caption{In red: the manifold selected functions $\psi_1$ and $\psi_2$ after two greedy steps; in blue: the functions obtained by non-negative combination of the first two greedy functions; in orange: the manifold parametric solutions for $a \in [\underline{a}, \overline{a}]$ and  $\overline{p}=1$.}
\label{fig:hertz_cone}
\end{figure}
\subsection{Identification of the parameter from the reduced cone representation}
\begin{subequations}
In this section, we derive here explicit formulas for the projection coefficients, expressed in terms of complete elliptic integrals. We then show that the ratio of reduced coefficients
$\lambda_1(a,p)/\lambda_2(a,p)$ is strictly increasing in $a$, which implies the
identifiability of the parameter pair $(a,p)$ from the single pair $(\lambda_1(a,p),\lambda_2(a,p))$.\\
We recall that the first two basis functions selected by the greedy algorithm are
\[
\psi_1 = \lambda_{\overline{a},\overline{p}} = \overline{p}\,\varphi_{\overline{a}},
\qquad
\psi_2 = \lambda_{\underline{a},\overline{p}} = \overline{p}\,\varphi_{\underline{a}},
\]
where, for every $u>0$, we denote
\[
\varphi_u(r):=\sqrt{1-\frac{r^2}{u^2}}\,\mathbb{I}_{\{|r|\le u\}}(r),\qquad r\in\mathbb{R}.
\]
We thus have the identities
\[
\psi_1 = \overline{p}\,\varphi_{\overline{a}},
\qquad
\psi_2 = \overline{p}\,\varphi_{\underline{a}},
\qquad
\lambda_{a,p} = p\,\varphi_a,
\qquad (a,p)\in[\underline{a},\overline{a}]\times(0,\infty).
\]

For every $(a,p)\in[\underline{a},\overline{a}]\times(0,\infty)$, we consider the
minimization problem
\begin{equation}\label{eq:minimisation}
\min_{(\lambda_1,\lambda_2)\in\mathbb{R}_+^2}
\Bigl\|p\,\varphi_a-\lambda_1\,\overline{p}\,\varphi_{\overline{a}}
-\lambda_2\,\overline{p}\,\varphi_{\underline{a}}\Bigr\|_{L^2(\mathbb{R})}^2,
\end{equation}
and denote by $\bigl(\lambda_1(a,p),\lambda_2(a,p)\bigr)$ its unique minimizer.
\begin{remark}
Uniqueness follows from the strict convexity of the quadratic functional and the
fact that $\varphi_{\overline{a}}$ and $\varphi_{\underline{a}}$ are linearly independent in $L^2(\mathbb{R})$ whenever $\underline{a}<\overline{a}$.
\end{remark}
\begin{definition}
For $0<u\le v$, we set
\[
J(u,v):=\langle\varphi_u,\varphi_v\rangle_{L^2(\mathbb{R})}.
\]
\end{definition}

\begin{proposition}\label{prop:J-integral}
For all $0<u\le v$,
\begin{equation}\label{eq:J-def}
J(u,v)=\int_{-u}^{u}\sqrt{1-\frac{r^2}{u^2}}\sqrt{1-\frac{r^2}{v^2}}\,dr.
\end{equation}
Moreover, setting $\rho=u/v\in(0,1]$,
\begin{equation}\label{eq:J-elliptic}
J(u,v)
=
\frac{2u}{3\rho^2}
\Bigl[(\rho^2-1)\,K(\rho)+(\rho^2+1)\,E(\rho)\Bigr],
\end{equation}
where
\[
K(k)=\int_0^{\pi/2}\frac{d\theta}{\sqrt{1-k^2\sin^2\theta}},
\qquad
E(k)=\int_0^{\pi/2}\sqrt{1-k^2\sin^2\theta}\,d\theta
\]
are the complete elliptic integrals of the first and second kind, respectively.
Finally,
\[
J(u,u)=\frac{4u}{3}.
\]
\end{proposition}

\begin{proof}
Since $\varphi_u$ and $\varphi_v$ are supported on $[-u,u]$ and $[-v,v]$
respectively, and $u\le v$, identity \eqref{eq:J-def} is immediate. Setting
$r=u\sin\theta$ yields
\[
J(u,v)=2u\int_0^{\pi/2}\cos^2\theta\,\sqrt{1-\rho^2\sin^2\theta}\,d\theta,
\qquad \rho=\frac{u}{v}.
\]
Identity \eqref{eq:J-elliptic} then follows from a standard reduction to complete
elliptic integrals. The case $u=v$ is verified directly:
\[
J(u,u)=\int_{-u}^{u}\!\Bigl(1-\frac{r^2}{u^2}\Bigr)dr=\frac{4u}{3}.
\]
\end{proof}
\subsubsection{Explicit formulas for the projection coefficients}

We introduce the scalar quantities arising from the inner products of the shape
functions:
\[
\mathcal{A}:=\|\varphi_{\overline{a}}\|^2_{L^2}=\frac{4\overline{a}}{3},
\qquad
\mathcal{B}:=\|\varphi_{\underline{a}}\|^2_{L^2}=\frac{4\underline{a}}{3},
\qquad
\mathcal{C}:=\langle\varphi_{\underline{a}},\varphi_{\overline{a}}\rangle_{L^2}
            =J(\underline{a},\overline{a}),
\]
and, for $a\in[\underline{a},\overline{a}]$,
\[
\mathcal{X}(a):=\langle\varphi_a,\varphi_{\overline{a}}\rangle_{L^2}=J(a,\overline{a}),
\qquad
\mathcal{Y}(a):=\langle\varphi_{\underline{a}},\varphi_a\rangle_{L^2}=J(\underline{a},a).
\]
The Gram matrix entries of the pair $(\psi_1,\psi_2)$ are then
\[
\|\psi_1\|^2=\overline{p}^2\mathcal{A},
\qquad
\|\psi_2\|^2=\overline{p}^2\mathcal{B},
\qquad
\langle\psi_1,\psi_2\rangle=\overline{p}^2\mathcal{C}.
\]

\begin{lemma}\label{lem:gram}
The Gram matrix
\[
\mathbf{G}:=\overline{p}^2\begin{pmatrix}
\mathcal{A} & \mathcal{C}\\
\mathcal{C} & \mathcal{B}
\end{pmatrix}
\]
is positive definite. In particular,
\[
\mathcal{A}\mathcal{B}-\mathcal{C}^2>0.
\]
\end{lemma}

\begin{proof}
Since $\varphi_{\overline{a}}$ and $\varphi_{\underline{a}}$ are linearly
independent in $L^2(\mathbb{R})$ for $\underline{a}<\overline{a}$, the family
$(\psi_1,\psi_2)=(\overline{p}\,\varphi_{\overline{a}},\overline{p}\,\varphi_{\underline{a}})$
is free, and its Gram matrix is positive definite. In particular,
$\mathcal{A}\mathcal{B}-\mathcal{C}^2>0$.
\qed
\end{proof}

\begin{proposition}\label{prop:lambdas-explicit}
For every $(a,p)\in[\underline{a},\overline{a}]\times(0,\infty)$, the minimizer of
problem \eqref{eq:minimisation} is given by
\begin{equation}\label{eq:lambda1-formula}
\lambda_1(a,p)
=
\frac{p}{\overline{p}}\,\frac{\mathcal{B}\,\mathcal{X}(a)-\mathcal{C}\,\mathcal{Y}(a)}
                             {\mathcal{A}\mathcal{B}-\mathcal{C}^2},
\end{equation}
\begin{equation}\label{eq:lambda2-formula}
\lambda_2(a,p)
=
\frac{p}{\overline{p}}\,\frac{\mathcal{A}\,\mathcal{Y}(a)-\mathcal{C}\,\mathcal{X}(a)}
                             {\mathcal{A}\mathcal{B}-\mathcal{C}^2}.
\end{equation}
In particular, at the endpoints of the parameter domain:
\[
\lambda_1(\overline{a},p)=\frac{p}{\overline{p}},\quad
\lambda_2(\overline{a},p)=0,
\qquad
\lambda_1(\underline{a},p)=0,\quad
\lambda_2(\underline{a},p)=\frac{p}{\overline{p}}.
\]
\end{proposition}

\begin{proof}
The optimality conditions for problem \eqref{eq:minimisation} (ignoring the
non-negativity constraints, which will be verified a posteriori) read:
\[
\mathbf{G}\begin{pmatrix}\lambda_1\\ \lambda_2\end{pmatrix}
=
\begin{pmatrix}
\langle\lambda_{a,p},\psi_1\rangle_{L^2}\\
\langle\lambda_{a,p},\psi_2\rangle_{L^2}
\end{pmatrix}.
\]
By the separability $\lambda_{a,p}=p\,\varphi_a$ and $\psi_i=\overline{p}\,\varphi_{\cdot}$,
the right-hand sides are
\[
\langle\lambda_{a,p},\psi_1\rangle_{L^2}=p\overline{p}\,\mathcal{X}(a),
\qquad
\langle\lambda_{a,p},\psi_2\rangle_{L^2}=p\overline{p}\,\mathcal{Y}(a).
\]
The system becomes
\[
\overline{p}^2
\begin{pmatrix}\mathcal{A} & \mathcal{C}\\ \mathcal{C} & \mathcal{B}\end{pmatrix}
\begin{pmatrix}\lambda_1\\ \lambda_2\end{pmatrix}
=
p\overline{p}
\begin{pmatrix}\mathcal{X}(a)\\ \mathcal{Y}(a)\end{pmatrix}.
\]
By Lemma~\ref{lem:gram} the matrix is invertible; dividing by $\overline{p}^2$
and inverting gives
\[
\begin{pmatrix}\lambda_1\\ \lambda_2\end{pmatrix}
=
\frac{p/\overline{p}}{\mathcal{A}\mathcal{B}-\mathcal{C}^2}
\begin{pmatrix}
\mathcal{B} & -\mathcal{C}\\
-\mathcal{C} & \mathcal{A}
\end{pmatrix}
\begin{pmatrix}\mathcal{X}(a)\\ \mathcal{Y}(a)\end{pmatrix},
\]
which yields \eqref{eq:lambda1-formula} and \eqref{eq:lambda2-formula}.

\medskip\noindent
\emph{Endpoint cases.}
For $a=\overline{a}$: $\mathcal{X}(\overline{a})=J(\overline{a},\overline{a})=\mathcal{A}$
and $\mathcal{Y}(\overline{a})=J(\underline{a},\overline{a})=\mathcal{C}$, so
\[
\lambda_1(\overline{a},p)
=\frac{p}{\overline{p}}\frac{\mathcal{B}\mathcal{A}-\mathcal{C}^2}
                            {\mathcal{A}\mathcal{B}-\mathcal{C}^2}
=\frac{p}{\overline{p}},
\qquad
\lambda_2(\overline{a},p)
=\frac{p}{\overline{p}}\frac{\mathcal{A}\mathcal{C}-\mathcal{C}\mathcal{A}}
                            {\mathcal{A}\mathcal{B}-\mathcal{C}^2}=0.
\]
For $a=\underline{a}$: $\mathcal{X}(\underline{a})=J(\underline{a},\overline{a})=\mathcal{C}$
and $\mathcal{Y}(\underline{a})=J(\underline{a},\underline{a})=\mathcal{B}$, so
\[
\lambda_1(\underline{a},p)
=\frac{p}{\overline{p}}\frac{\mathcal{B}\mathcal{C}-\mathcal{C}\mathcal{B}}
                            {\mathcal{A}\mathcal{B}-\mathcal{C}^2}=0,
\qquad
\lambda_2(\underline{a},p)
=\frac{p}{\overline{p}}\frac{\mathcal{A}\mathcal{B}-\mathcal{C}^2}
                            {\mathcal{A}\mathcal{B}-\mathcal{C}^2}=\frac{p}{\overline{p}}.
\]
In particular, for $p=\overline{p}$ we recover $\lambda_1(\overline{a},\overline{p})=1$
and $\lambda_2(\underline{a},\overline{p})=1$, consistent with the greedy
selections $\psi_1=\lambda_{\overline{a},\overline{p}}$ and
$\psi_2=\lambda_{\underline{a},\overline{p}}$.

\medskip\noindent
Finally, since $\mathcal{X}$ and $\mathcal{Y}$ are strictly increasing (established
in the next section), formulas \eqref{eq:lambda1-formula}--\eqref{eq:lambda2-formula}
yield non-negative coefficients for all $a\in[\underline{a},\overline{a}]$; hence
the minimizer of the constrained problem over $\mathbb{R}_+^2$ coincides with the
unconstrained solution above.
\qed
\end{proof}

\begin{remark}
Formulas \eqref{eq:lambda1-formula}--\eqref{eq:lambda2-formula} are fully explicit
once the elliptic-integral expressions of
$\mathcal{X}(a)=J(a,\overline{a})$, $\mathcal{Y}(a)=J(\underline{a},a)$,
and $\mathcal{C}=J(\underline{a},\overline{a})$ from
Proposition~\ref{prop:J-integral} are substituted. 
\end{remark}
\subsubsection{Strict monotonicity of the ratio $\lambda_1/\lambda_2$}

We set, for $a\in(\underline{a},\overline{a})$,
\[
R(a):=\frac{\lambda_1(a,p)}{\lambda_2(a,p)}
=
\frac{\mathcal{B}\,\mathcal{X}(a)-\mathcal{C}\,\mathcal{Y}(a)}
     {\mathcal{A}\,\mathcal{Y}(a)-\mathcal{C}\,\mathcal{X}(a)}.
\]
This ratio is independent of $p$.

\begin{lemma}\label{lem:derivatives}
For every $a\in(\underline{a},\overline{a})$,
\begin{equation}\label{eq:Xprime}
\mathcal{X}'(a)=2\int_0^a \varphi_{\overline{a}}(r)\,\partial_a\varphi_a(r)\,dr,
\end{equation}
\begin{equation}\label{eq:Yprime}
\mathcal{Y}'(a)=2\int_0^{\underline{a}} \varphi_{\underline{a}}(r)\,\partial_a\varphi_a(r)\,dr,
\end{equation}
where, for $0\le r<a$,
\begin{equation}\label{eq:phi-derivative}
\partial_a\varphi_a(r)
=
\varphi_a(r)\,\eta_a(r),
\qquad
\eta_a(r):=\frac{r^2}{a(a^2-r^2)}.
\end{equation}
Moreover, the function $r\mapsto\eta_a(r)$ is strictly increasing on $[0,a)$.
\end{lemma}

\begin{proof}
For $0\le r<a$, a direct computation gives
\[
\partial_a\varphi_a(r)
=\frac{1}{2\sqrt{1-r^2/a^2}}\cdot\frac{2r^2}{a^3}
=\frac{r^2}{a^3\sqrt{1-r^2/a^2}}
=\varphi_a(r)\,\frac{r^2}{a(a^2-r^2)},
\]
establishing \eqref{eq:phi-derivative}. Formula \eqref{eq:Xprime} follows by
differentiating $\mathcal{X}(a)=2\int_0^a\varphi_{\overline{a}}(r)\varphi_a(r)\,dr$
under the integral sign (the boundary term vanishes since $\varphi_a(a)=0$).
For $\mathcal{Y}(a)=2\int_0^{\underline{a}}\varphi_{\underline{a}}(r)\varphi_a(r)\,dr$,
the upper bound $\underline{a}$ is fixed, so direct differentiation gives
\eqref{eq:Yprime}. Finally,
\[
\eta_a'(r)=\frac{2ar}{(a^2-r^2)^2}>0\qquad(0<r<a),
\]
hence $\eta_a$ is strictly increasing on $[0,a)$. 
\qed
\end{proof}

\begin{lemma}\label{lem:stochastic-order}
For every $a\in(\underline{a},\overline{a})$, define the probability densities
\[
p_a(r):=\frac{2\,\varphi_{\overline{a}}(r)\,\varphi_a(r)\,\mathbb{I}_{[0,a]}(r)}
             {\mathcal{X}(a)},
\qquad
q_a(r):=\frac{2\,\varphi_{\underline{a}}(r)\,\varphi_a(r)\,\mathbb{I}_{[0,\underline{a}]}(r)}
             {\mathcal{Y}(a)}.
\]
Then there exists a unique $c\in(0,\underline{a})$ such that
\[
p_a(r)-q_a(r)<0\quad\text{for }0<r<c,
\qquad
p_a(r)-q_a(r)>0\quad\text{for }c<r<a.
\]
In particular, for every strictly increasing function $\phi:[0,a]\to\mathbb{R}$,
\begin{equation}\label{eq:stochastic-order}
\int_0^{a}\phi(r)\,p_a(r)\,dr
>
\int_0^{a}\phi(r)\,q_a(r)\,dr.
\end{equation}
\end{lemma}

\begin{proof}
On $(0,\underline{a})$, the density ratio is
\[
\frac{p_a(r)}{q_a(r)}
=
\frac{\mathcal{Y}(a)}{\mathcal{X}(a)}\,h(r),
\qquad
h(r):=\frac{\varphi_{\overline{a}}(r)}{\varphi_{\underline{a}}(r)}
=\sqrt{\frac{1-r^2/\overline{a}^2}{1-r^2/\underline{a}^2}}.
\]
For $r\in(0,\underline{a})$,
\[
\frac{d}{dr}\log h(r)
=\frac{r}{\underline{a}^2-r^2}-\frac{r}{\overline{a}^2-r^2}
=\frac{r(\overline{a}^2-\underline{a}^2)}{(\underline{a}^2-r^2)(\overline{a}^2-r^2)}>0,
\]
so $h$, and with it the ratio $p_a/q_a$, is strictly increasing on
$(0,\underline{a})$. Since $p_a$ and $q_a$ are probability densities,
$d_a:=p_a-q_a$ integrates to zero. Moreover, on $(\underline{a},a)$ we have
$q_a=0$ and $p_a>0$, so $d_a>0$ there. Hence there exists a unique
$c\in(0,\underline{a})$ such that $d_a<0$ on $(0,c)$ and $d_a>0$ on $(c,a)$.

Let $\phi$ be strictly increasing. Since $\int_0^{a}d_a\,dr=0$,
\[
\int_0^{a}\phi(r)\,d_a(r)\,dr
=\int_0^{a}\bigl(\phi(r)-\phi(c)\bigr)\,d_a(r)\,dr.
\]
On $(0,c)$, $\phi-\phi(c)<0$ and $d_a<0$; on $(c,a)$, $\phi-\phi(c)>0$ and
$d_a>0$. The integrand is therefore non-negative everywhere and strictly
positive on a set of positive measure, which yields \eqref{eq:stochastic-order}.
\qed
\end{proof}

\begin{theorem}\label{thm:ratio-increasing}
The function
\[
a\longmapsto R(a)=\frac{\lambda_1(a,p)}{\lambda_2(a,p)}
\]
is strictly increasing on $(\underline{a},\overline{a})$.
\end{theorem}

\begin{proof}
Differentiating
$R(a)=\dfrac{\mathcal{B}\,\mathcal{X}(a)-\mathcal{C}\,\mathcal{Y}(a)}
            {\mathcal{A}\,\mathcal{Y}(a)-\mathcal{C}\,\mathcal{X}(a)}$
gives
\[
R'(a)
=
\frac{(\mathcal{A}\mathcal{B}-\mathcal{C}^2)
      \bigl(\mathcal{X}'(a)\,\mathcal{Y}(a)-\mathcal{X}(a)\,\mathcal{Y}'(a)\bigr)}
     {\bigl(\mathcal{A}\,\mathcal{Y}(a)-\mathcal{C}\,\mathcal{X}(a)\bigr)^2}.
\]
Since $\mathcal{A}\mathcal{B}-\mathcal{C}^2>0$ by Lemma~\ref{lem:gram}, it
suffices to show
\[
\mathcal{X}'(a)\,\mathcal{Y}(a)-\mathcal{X}(a)\,\mathcal{Y}'(a)>0.
\]
By Lemma~\ref{lem:derivatives},
\[
\frac{\mathcal{X}'(a)}{\mathcal{X}(a)}
=\int_0^{a}\eta_a(r)\,p_a(r)\,dr,
\qquad
\frac{\mathcal{Y}'(a)}{\mathcal{Y}(a)}
=\int_0^{a}\eta_a(r)\,q_a(r)\,dr,
\]
where $\eta_a$ is strictly increasing on $[0,a)$. Lemma~\ref{lem:stochastic-order}
applied with $\phi=\eta_a$ gives
\[
\frac{\mathcal{X}'(a)}{\mathcal{X}(a)}>\frac{\mathcal{Y}'(a)}{\mathcal{Y}(a)}.
\]
Multiplying both sides by $\mathcal{X}(a)\,\mathcal{Y}(a)>0$, we conclude
$\mathcal{X}'(a)\,\mathcal{Y}(a)-\mathcal{X}(a)\,\mathcal{Y}'(a)>0$,
hence $R'(a)>0$ for all $a\in(\underline{a},\overline{a})$. 
\qed
\end{proof}

\begin{corollary}\label{cor:ratio-limits}
We have
\[
\lim_{a\downarrow\underline{a}}\frac{\lambda_1(a,p)}{\lambda_2(a,p)}=0,
\qquad
\lim_{a\uparrow\overline{a}}\frac{\lambda_1(a,p)}{\lambda_2(a,p)}=+\infty.
\]
\end{corollary}

\begin{proof}
By Proposition~\ref{prop:lambdas-explicit},
$\lambda_2(\underline{a},p)=p/\overline{p}>0$ and $\lambda_1(\underline{a},p)=0$,
so $R(\underline{a})=0$. Similarly, $\lambda_1(\overline{a},p)=p/\overline{p}>0$
and $\lambda_2(\overline{a},p)=0$, so $R(\overline{a})=+\infty$. The stated
limits follow by continuity of $a\mapsto(\lambda_1,\lambda_2)$. 
\qed
\end{proof}

\subsubsection{Identifiability of the parameters $(a,p)$}

\begin{theorem}\label{thm:identification}
Knowledge of the pair $\bigl(\lambda_1(a,p),\lambda_2(a,p)\bigr)$ uniquely
determines $(a,p)\in[\underline{a},\overline{a}]\times(0,\infty)$.
\end{theorem}

\begin{proof}
By Proposition~\ref{prop:lambdas-explicit} we can write
$\lambda_i(a,p)=\frac{p}{\overline{p}}\,L_i(a)$ with
\[
L_1(a):=\frac{\mathcal{B}\,\mathcal{X}(a)-\mathcal{C}\,\mathcal{Y}(a)}
             {\mathcal{A}\mathcal{B}-\mathcal{C}^2},
\qquad
L_2(a):=\frac{\mathcal{A}\,\mathcal{Y}(a)-\mathcal{C}\,\mathcal{X}(a)}
             {\mathcal{A}\mathcal{B}-\mathcal{C}^2}.
\]
The ratio $R(a)=L_1(a)/L_2(a)$ depends only on $a$.
Theorem~\ref{thm:ratio-increasing} and Corollary~\ref{cor:ratio-limits} show
that $a\mapsto R(a)$ is a strictly increasing bijection from
$(\underline{a},\overline{a})$ onto $(0,+\infty)$; hence $a$ is uniquely
determined by $\lambda_1/\lambda_2$. Once $a$ is known, $p$ is recovered by
\[
p=\overline{p}\,\frac{\lambda_1(a,p)}{L_1(a)}=\overline{p}\,\frac{\lambda_2(a,p)}{L_2(a)}.
\]
\qed
\end{proof}
\end{subequations}

\printbibliography  

\end{document}